\documentclass{amsart}
\usepackage{graphicx} 
\usepackage{amsmath,amstext,amssymb,bm}
\usepackage{leftidx}

\usepackage{xcolor} 
\usepackage{soul} 
\usepackage{tikz}
\usetikzlibrary{shapes,arrows}
\usepackage{mathrsfs}
\usepackage{epstopdf}
\usepackage{color}
\usepackage{multirow}
\usepackage{tabularx}
\numberwithin{equation}{section}
\newtheorem{definition}{Definition}[section]

\newtheorem{remark}{Remark}[section]
\newtheorem{theorem}{Theorem}[section]
\newtheorem{corollary}{Corollary}[section]
\newtheorem{lemma}{{Lemma}}[section]

\allowdisplaybreaks[4]

\newcommand{\cP}{\mathcal{P}}

\newcommand{\E}{\mathcal{E}}
\newcommand{\cT}{\mathcal{T}}
\renewcommand{\div}{\mbox{\rm div\,}}

\newcommand{\cF}{\mathcal{F}}

\newcommand{\mP}{\mathbb{P}}
\newcommand{\mE}{\mathbb{E}}
\newcommand{\mV}{\mathbb{V}}
\newcommand{\mW}{\mathbb{W}}

\newcommand{\veps}{\varepsilon}
\newcommand{\Ome}{\Omega}
\newcommand{\p}{\partial}
\newcommand{\nab}{\nabla}
\newcommand{\vu}{{\bf u}}
\newcommand{\vB}{{\bf B}}
\newcommand{\vW}{{\bf W}}
\newcommand{\vf}{{\bf f}}
\newcommand{\vH}{{\bf H}}
\newcommand{\bH}{\mathbb{H}}
\newcommand{\bV}{\mathbb{V}}

\newcommand{\vL}{{\bf L}}
\newcommand{\vv}{{\bf v}}
\newcommand{\ve}{{\bf e}}
\newcommand{\e}{\pmb{\varepsilon}}
\newcommand{\ta}{\pmb{\theta}}

\begin{document}
	
	\title[Chorin-Type Projection Methods for Stochastic Stokes Equations]{Analysis of  Chorin-Type Projection Methods for the Stochastic Stokes Equations with General  Multiplicative Noises\dag}
	\thanks{\dag This work was partially supported by the NSF grants DMS-1620168 and 
		DMS-2012414.} 
	%\markboth{X. FENG and M. SUTTON}{A NEW THEORY OF WEAK FRACTIONAL CALCULUS}
	
	\author{Xiaobing Feng}
	\address{Department of Mathematics, The University of Tennessee, Knoxville, TN 37996, U.S.A. }
	\email{xfeng@math.utk.edu}
	%\thanks{\dag Department of Mathematics, The University of Tennessee, Knoxville, TN 37996, U.S.A. (xfeng@math.utk.edu).
	%	The work of this author was partially supported by the NSF grant: DMS-1620168.}
	
	\author{Liet Vo}
	\address{Department of Mathematics, The University of Tennessee, Knoxville, TN 37996, U.S.A. }
	\email{lvo6@vols.utk.edu}
	%\thanks{\ddag Department of Mathematics, The University of Tennessee, Knoxville, TN 37996, U.S.A. (lvo6@vols.utk.edu).
	%	The work of this author was partially supported by the NSF grant: DMS-1620168.}
	
	%\date{}
	
	%\maketitle
	
	%\thispagestyle{empty}

	\begin{abstract}
		This paper is concerned with numerical analysis of two fully discrete 
		Chorin-type projection methods for the stochastic Stokes equations with general 
		non-solenoidal multiplicative noise. The first scheme is the standard Chorin scheme 
		and the second one is a modified Chorin scheme which is designed by employing the 
		Helmholtz decomposition on the noise function at each time step to produce 
		a projected divergence-free noise and a ``pseudo pressure" after combining the original 
		pressure and the curl-free part of the decomposition.   An $O(k^\frac14)$ rate of convergence is proved for the standard Chorin scheme, which is sharp but not optimal due to the use of non-solenoidal noise, where $k$ denotes the time mesh size. On the other hand, an optimal convergence rate  $O(k^\frac12)$ is established for the modified Chorin scheme. 
		%It is crucial to measure the errors in appropriate norms. 
		The fully discrete finite element methods are formulated by discretizing both semi-discrete Chorin schemes in space by the standard finite element method. Suboptimal 
		order error estimates are derived for both fully discrete methods. It is proved that 
		all spatial error constants contain a growth factor $k^{-\frac12}$, where $k$ denotes the 
		time step size, which explains the deteriorating performance of the standard Chorin scheme 
		when $k\to 0$ and the space mesh size is fixed as observed earlier in the numerical tests of \cite{CHP12}.  
		Numerical results are also provided to guage the performance of the proposed 
		numerical methods and to validate the sharpness of the theoretical error estimates.
		
	\end{abstract}
	
	%\end{keywords}

	
	\maketitle
	
	\tableofcontents

	%%%%%%%%
	\section{Introduction}\label{sec-1}
This paper is concerned with developing and analyzing Chorin-type projection finite element methods for the  following time-dependent stochastic Stokes problem:
\begin{subequations}\label{eq1.1}
	\begin{alignat}{2} \label{eq1.1a}
		d\vu &=\bigl[\nu\Delta \vu-\nabla p + { \vf}\bigr] dt + \vB(\vu) d\vW(t)  &&\qquad\mbox{a.s. in}\, D_T:=(0,T)\times D,\\
		\div \vu &=0 &&\qquad\mbox{a.s. in}\, D_T,\label{eq1.1b}\\
		%\vu &=0 &&\qquad\mbox{a.s. on}\, \partial D_T:=(0,T)\times \partial D, \label{eq1.1c}\\
		\vu(0)&= \vu_0 &&\qquad\mbox{a.s. in}\, D,\label{eq1.1d}
	\end{alignat}
\end{subequations}
where $D = (0,L)^2 \subset \mathbb{R}^d \, (d=2,3)$ represents a period of the periodic domain in $\mathbb{R}^d$, $\vu$ and $p$ stand for respectively the velocity field and the pressure of the fluid,
$\vB$ is an operator-valued random field, $\{\vW(t); t\geq 0\}$ denotes an $\vL^2(D)$-valued
$Q$-Wiener process, and $\vf$ is a body force function (see Section \ref{sec2} for their precise definitions).  Here we seek periodic-in-space solutions $(\vu,p)$ with period $L$, that is, 
$\vu(t,{\bf x} + L{\bf e}_i) = \vu(t,{\bf x})$ and $p(t,{\bf x}+L{\bf e}_i)=p(t,{\bf x})$ 
almost surely 
and for any $(t, {\bf x})\in (0,T)\times \mathbb{R}^d$  and $1\leq i\leq d$, where 
$\{\bf e_i\}_{i=1}^d$ denotes the canonical basis of $\mathbb{R}^d$.

The system \eqref{eq1.1a} is a stochastic perturbation of the deterministic Stokes system by introducing 
a multiplicative noise force term $\vB(\cdot)d\vW(s)$ 
and it has been used to model turbulent fluids (cf. \cite{Bensoussan95, BT1973,HM06,Temam01}).  The stochastic Stokes system is a simplified model of the full stochastic Navier-Stokes equations by omitting the nonlinear term $(\vu\cdot\nab) \vu$ in the drift part of the stochastic Navier-Stokes equations. 
Although the deterministic Stokes equations is a linear PDE system  which has been well studied 
in the literature (cf. \cite{Girault_Raviart86, Temam01}  and the references therein),  the stochastic Stokes system  
\eqref{eq1.1a}  is intrinsically  nonlinear because the diffusion coefficient 
$\vB$ is nonlinear in the velocity $\vu$.  Due to the introduction of random forces it has been well known that 
the solution of problem \eqref{eq1.1} has very low regularities in time.  We refer the reader to  \cite{Bensoussan95, LRS03, PZ92} and the references therein for a detailed account about the well-posedness 
and regularities of the solution for system \eqref{eq1.1}.  

Besides their mathematical and practical importance, the stochastic Stokes (and Navier-Stokes) 
equations have been used as prototypical stochastic PDEs for developing efficient 
numerical methods and general numerical analysis techniques  for analyzing 
numerical methods for stochastic PDEs. In that regard several works have been reported 
in the literature \cite{CP12,Feng,Feng1,Breit,CHP12,BBM14}. Euler-Maruyama time dsicretization and divergence-free finite element space discretization was proposed and analyzed in 
\cite{CP12} in the case of divergence-free noises (i.e., $\mathbf{B}(\vu)$ is divergence-free). 
Optimal order error estimates in strong norm for the velocity approximation were obtained. 
In \cite{Feng, Feng1} the authors considered the general noise and analyzed the standard 
and a modified mixed finite element methods as well as pressure stabilized methods for 
space discretization, suboptimal order error estimates were proved in \cite{Feng} for the velocity approximation in strong norm and for the pressure approximation in a time-averaged 
norm, all these suboptimal order error estimates were improved to optimal order for 
a Helmholtz projection-enhanced mixed finite element 
in \cite{Feng1} (also see \cite{Breit} for a similar approach). 
It should be noted that the reason for measuring the pressure errors 
in a time-averaged norm is because the low regularity of the pressure field which is 
only a distribution in general and the numerical tests of \cite{Feng,Feng1} suggest 
that these error estimates are sharp. 
In \cite{CHP12} the authors proposed a Chorin time-splitting finite element method
for problem \eqref{eq1.1} and proved a suboptimal convergence rate in strong norm 
for the velocity approximation in the case of divergence-free noises.
In \cite{BBM14} the authors proposed an iterative splitting scheme for 
stochastic Navier-Stokes equations and a strong convergence in probability
was established in the 2-D case for the velocity approximation. 
In a recent work \cite{BM18}, the authors proposed another time-splitting scheme 
and proved its strong $L^2$ convergence for the velocity approximation.

Compared to the recent advances on mixed finite element 
methods \cite{CP12, Feng, Feng1}, the numerical analysis of the well-known Chorin  projection/splitting scheme for the stochastic Stokes equations lags behind.  To the 
best of our knowledge, the only analysis result obtained in 
\cite{CHP12} is the optimal convergence in the energy norm for the velocity 
approximation in the case of divergence-free noises (i.e., $\mathbf{B}(\vu)$
is divergence-free). 
Several natural and important questions arise and must be addressed for a better  
understanding of the 
Chorin projection scheme for problem \eqref{eq1.1}. Among them are {\em (i) Does the pressure approximation
	converge even when the noise is divergence-free? If so, in what sense and what rate? 
	(ii) Does the Chorin projection scheme converge (for both the velocity and pressure
	approximations) for general noises? If so, in what sense and what rate? 
	(iii) Could the performance of the standard Chorin projection scheme be improved
	one way or another in the case of general noises?}  
The primary objective this paper is to provide a positive answer to each of 
the above questions. 

As it was shown in \cite{CHP12}, the adaptation of the standard deterministic Chorin 
projection scheme to problem \eqref{eq1.1} is straightforward 
(see Algorithm 1 of Section \ref{sec3}).
The idea of the Chorin scheme is to separate the computation of the velocity 
and pressure at each time step which is done by solving two decoupled Poisson problems
and the divergence-free constraint for the velocity approximation is enforced by 
a Helmholtz projection technique which can be easily obtained using the solutions of
the two Poisson problems. 
The Chorin scheme also can be compactly rewritten as a pressure stabilization scheme 
at each time step as follows (cf. \cite{CHP12}):
\begin{subequations}\label{eq1.2}
	\begin{alignat}{2}\label{eq1.2a}
		\tilde{\vu}^{n+1} - \tilde{\vu}^n - k\nu\Delta \tilde{\vu}^{n+1} + k\nab p^{n} &= k\vf^{n+1} + \vB(\tilde{\vu}^n)\Delta W_{n+1} &&\qquad \mbox{a.s. in }D_T,\\
		\div \tilde{\vu}^{n+1} - k\Delta p^{n+1} &= 0 &&\qquad\mbox{a.s. in }D_T, \label{eq1.2b}\\
		\p_{\bf n} p^{n+1} &= 0 &&\qquad\mbox{a.s. on }\p D_T,  \label{eq1.2c}
	\end{alignat}
	where $\p_{\bf n} p^{n+1}$  
	denotes the normal  derivative of $p^{n+1}$ and $k$ is the time step size.
\end{subequations}
One of advantages of the above Chorin scheme is that the spatial approximation  
spaces for $\tilde{\vu}^{n+1}$ and $p^{n+1}$ can be chosen independently, so unlike in the mixed finite element method, they are not required to satisfy an {\em inf-sup} condition. 
Notice that a time lag on pressure appears in equation \eqref{eq1.2a} which causes most of difficulties 
in the convergence analysis (cf. \cite{Ran, Guermond06, Prohl97, CHP12}). We also note that the term $-k\Delta p^{n+1}$ in equation \eqref{eq1.2b} is known as a pressure stabilization term.

To improve the convergence of the standard Chorin scheme, we adopt a
Helmholtz projection technique as used in \cite{Feng1} (also see \cite{Breit}). 
At each time step we first perform the Helmholtz decomposition  $\vB(\tilde{\vu}^n)=\pmb{\eta}^n+\nabla \xi^n$ and then rewrite \eqref{eq1.2a} as 
\begin{align}\label{eq1.3}
	\tilde{\vu}^{n+1} - \tilde{\vu}^n - k\nu\Delta \tilde{\vu}^{n+1} + k\nab r^{n} = k\vf^{n+1} +  \pmb{\eta}^n\Delta W_{n+1} \qquad \mbox{a.s. in }D_T,
\end{align}
where $r^n= p^n -k^{-1} \xi^n\Delta W_{n+1}$. Our modified Chorin scheme consists of  
\eqref{eq1.3}, \eqref{eq1.2b}--\eqref{eq1.2c} and the Helmholtz decomposition 
$\vB(\tilde{\vu}^n)=\pmb{\eta}^n+\nabla \xi^n$. Since $\pmb{\eta}^n$ is divergence-free, 
it turns out that the finite element approximation of the modified Chorin scheme 
has better convergence properties. Notice that $p^n$ can be recovered from 
$r^n$ via the simple algebraic relation $p^n= r^n +k^{-1} \xi^n\Delta W_{n+1}$. 
	
	The main contributions of this paper are summarized below.
	
	\begin{itemize}
		\item We proved the following error estimates in strong norms 
		for the Chorin-$P_1$ finite element method (see Algorithm 3) for problem \eqref{eq1.1}
		with general multiplicative noises:
		
		\begin{align*}
		\biggl(\mE\Bigl[k\sum_{m=0}^M\|\vu(t_m) - \tilde{\vu}^m_h\|^2\Bigr]\biggr)^{\frac12} &+ \max_{0\leq \ell \leq M} \biggl(\mE\Bigl[\Bigl\|k\sum_{m=0}^{\ell}\nab(\vu(t_m) - \tilde{\vu}^m_h)\Bigr\|^2\Bigr]\biggr)^{\frac12}\\\nonumber &\qquad\qquad\qquad\leq C\Bigl(k^{\frac14} + hk^{-\frac12}\Bigr),\\
		\biggl(\mE\Bigl[k\sum_{m=0}^M\Bigl\|P(t_m) &- k\sum_{n=0}^m p^n_h\Bigr\|^2\Bigr]\biggr)^{\frac12} \leq C\Bigl(k^{\frac14} + hk^{-\frac12}\Bigr),
	\end{align*}
		where $(\vu(t_m), P(t_m) )$ are the solution to 
		problem \eqref{eq1.1} while $(\tilde{\vu}_{h}^m, p_h^m)$ are the discrete solution of Algorithm 3, see Sections \ref{sec2} and \ref{sec4} for their precise definitions. 
		
		\item We proposed a modified Chorin-$P_1$ finite element method (see Algorithm 4) and 
		proved the following error estimates in strong norms for problem \eqref{eq1.1}
		with general multiplicative noises:
		\begin{align*}
			&\max_{1\leq m \leq M} \Bigl(\mathbb{E}\bigl[ \bigl\|\vu(t_m)-\tilde{\vu}^m_{h} \bigr\|^2\,\bigl]\Bigr)^{\frac12}
			+\biggl( \mathbb{E}\biggl[ k\sum_{m=1}^{M} \bigl\|\nabla (\vu(t_n) -\vu^n_{h}) \bigr\|^2\,\biggr] \biggr)^{\frac12} \\
			&\hskip 1.9in \leq  C\Bigl( k^{\frac12} + h + k^{-\frac12} h^2 \Bigr), \\ 
			&\biggl(\mathbb{E} \biggl[ \bigg\|R(t_m) -k\sum^m_{n=1}r^n_{h} \biggr\|^2\, \biggr]\biggr)^{\frac12} 	
			+ \biggl( \mathbb{E}\biggl[\bigg\|P(t_m) -k\sum^m_{n=1}p^n_{h} \biggr\|^2\,\biggr] \biggr)^{\frac12} \\
			&\hskip 1.9in \leq C\Bigl(k^{\frac12} + h + k^{-\frac12} h^2 \Bigr)  .
		\end{align*}
		where $(\vu(t_m), P(t_m))$ is the solution to problem \eqref{eq1.1} and $R(t)$ is defined
		as 	the time-average of the pseudo pressure $r(t)$ while $(\vu^m_h, r^m_h, p^m_h)$ is 
		the solution of Algorithm 4, see Sections \ref{sec2} and \ref{sec4} for their precise definitions. 
	\end{itemize}
	We note that all spatial error constants contain a growth factor $k^{-\frac12}$, which 
	explains the deteriorating performance of the standard (and modified) Chorin scheme 
	when $k\to 0$ and the mesh size $h$ is fixed as observed in the numerical tests of \cite{CHP12}. The numerical 
	experiments to be given in Section \ref{sec5} indicate that the dependence 
	on factor $k^{-\frac12}$ is sharp.

	The remainder of this paper is organized as follows. In Section \ref{sec2}, we first 
	introduce some space notations and state the assumptions on the initial data and on $\vB$ as well as recall the definition of solutions to \eqref{eq1.1}. We then state and prove a H\"older continuity property for the pressure $p$ in a time-averaged norm. In Section \ref{sec3}, we 
	define the standard Chorin projection scheme as Algorithm 1 for problem \eqref{eq1.1} in Subsection \ref{sub3.1} and the modified Chorin scheme as Algorithm 2 in Subsection \ref{sub3.2}. The highlights of this section are to prove some uniform (in $k$) 
	stability estimates which are very useful for error analysis later. In Section \ref{sec4}, we formulate the finite element 
	spatial discretization for both the standard Chorin and modified Chorin schemes in Algorithm 3 and 4, respectively and prove the quasi-optimal error estimates for both algorithms as 
	summarized above. In Section \ref{sec5}, we present several numerical experiments 
	to gauge the performance of the proposed numerical methods and 
	to validate the sharpness of the proved error estimates.
	
	\section{Preliminaries}\label{sec2}
	Standard function and space notation will be adopted in this paper. 
	Let $\vH^1_0(D)$ denote the subspace of $\vH^1(D)$ whose ${\mathbb R}^d$-valued functions have zero trace on $\p D$, and $(\cdot,\cdot):=(\cdot,\cdot)_D$ denote the standard $L^2$-inner product, with induced norm $\Vert \cdot \Vert$. We also denote ${\bf L}^p_{per}(D)$ and ${\bf H}^{k}_{per}(D)$ as the Lebesgue and Sobolev spaces of the functions that are periodic and have vanishing mean, respectively. 
	Let $(\Omega,\cF, \{\cF_t\},\mP)$ be a filtered probability space with the probability measure $\mP$, the 
	$\sigma$-algebra $\cF$ and the continuous  filtration $\{\cF_t\} \subset \cF$. For a random variable $v$ 
	defined on $(\Omega,\cF, \{\cF_t\},\mP)$,
	${\mathbb E}[v]$ denotes the expected value of $v$. 
	For a vector space $X$ with norm $\|\cdot\|_{X}$,  and $1 \leq p < \infty$, we define the Bochner space
	$\bigl(L^p(\Omega,X); \|v\|_{L^p(\Omega,X)} \bigr)$, where
	$\|v\|_{L^p(\Omega,X)}:=\bigl({\mathbb E} [ \Vert v \Vert_X^p]\bigr)^{\frac1p}$.
	We also define 
	\begin{align*}
		{\mathbb H} &:= \bigl\{{\bf v}\in  \vL^2_{per}(D) ;\,\div {\bf v}=0 \mbox{ in }D\, \bigr\}\, , \\
		{\mathbb V} &:=\bigl\{{\bf v}\in  \vH^1_{per}(D) ;\,\div {\bf v}=0 \mbox{ in }D \bigr\}\, .
	\end{align*}
	
	We recall from \cite{Girault_Raviart86} that the (orthogonal) Helmholtz projection 
	${\bf P}_{{\mathbb H}}: \vL^2_{per}(D) \rightarrow {\mathbb H}$ is defined 
	by ${\bf P}_{{\mathbb H}} {\bf v} = \pmb{\eta}$ for every ${\bf v} \in \vL^2_{per}(D)$, 
	where $(\pmb{\eta}, \xi) \in {\mathbb H} \times H^1_{per}(D)/\mathbb{R}$ is a unique tuple such that 
	$${\bf v} = \pmb{\eta} + \nabla \xi\, , $$
	and $\xi\in H^1_{per}(D)/\mathbb{R}$ solves the following Poisson problem 
	with the homogeneous Neumann boundary condition:
	\begin{equation}\label{poisson}
		\Delta \xi = \div {\bf v}.
	\end{equation}
	We also define the Stokes operator ${\bf A} := -{\bf P}_{\mathbb H} \Delta: {\mathbb V} \cap \vH^2_{per}(D) \rightarrow {\mathbb H}$.

	Throughout this paper we assume that  ${\bf B}: \vL^2_{per}(D) \rightarrow \vL^2_{per}(D)$ is a 
	Lipschitz continuous mapping and has linear growth, that is, 
	there exists a constant $C > 0$ such that for all ${\bf v}, {\bf w} \in \vL^2_{per}(D)$ 
	\begin{subequations}\label{eq2.6}
		\begin{align}\label{eq2.6a}
			\|{\bf B}({\bf v})-{\bf B}({\bf w})\|  &\leq C\|{\bf v}-{\bf w}\|\, , \\
			\|{\bf B}({\bf v})\|  &\leq C \bigl(1+ \|{\bf v}\| \bigr)\, ,   \label{eq2.6b}
			%\|\mathcal{D} \mathbf{B}\|_* &\leq C,\label{eq2.6c}
		\end{align}
		%but we emphasize that the results of this paper still hold without this assumption.}
	\end{subequations}
	
	Since we shall not explicitly track the dependence of all constants on $\nu$, 
	for ease of the presentation, unless it is stated otherwise, we shall set $\nu = 1$ in the rest of 
	the paper and assume that $\vf \in L^2(\Ome;L^2_{per}(D))$. In addition,  
	we shall use $C$ to denote a generic positive constant
	which may depend on $T$, the datum functions $\vu_0$ and $\vf$, and the domain $D$ 
	but is independent of the mesh parameter $h$ and $k$.
	
	\subsection{Variational formulation of the stochastic Stokes equations}\label{sec-2.2}
	We first define the variational solution concept for \eqref{eq1.1} and refer the reader to \cite{Chow07,PZ92} for a proof of its existence and uniqueness.
	
	\begin{definition}\label{def2.1} % (cf. \cite{LRS03} )
		Given $(\Omega,\cF, \{\cF_t\},\mP)$, let $W$ be an ${\mathbb R}$-valued Wiener process on it. 
		Suppose ${\bf u}_0\in L^2(\Omega, {\mathbb V})$ and $\vf \in L^2(\Ome;L^2((0,T);L^2_{per}(D)))$.
		An $\{\cF_t\}$-adapted stochastic process  $\{{\bf u}(t) ; 0\leq t\leq T\}$ is called
		a variational solution of \eqref{eq1.1} if ${\bf u} \in  L^2\bigl(\Omega; C([0,T]; {\mathbb V})) %\cap L^2\bigl(0,T; {\mathbb V}) \bigr)$	
		\cap L^2\bigl(\Ome;0,T;\vH^2_{per}(D)\bigr)$,
		and satisfies $\mP$-a.s.~for all $t\in (0,T]$
		\begin{align}\label{eq2.8a}
			\bigl({\bf u}(t),  {\bf v} \bigr) + \int_0^t  \bigl(\nab {\bf u}(s), \nab {\bf v} \bigr) 
			\,  ds&=({\bf u}_0, {\bf v})+ \int_0^t \big(\vf(s), \vv\big) \, ds \\\nonumber
			& \qquad+  {\int_0^t  \Bigl( {\bf B}\bigl({\bf u}(s)\bigr), {\bf v} \Bigr)\, dW(s)}  \qquad\forall  \, {\bf v}\in {\mathbb V}\, . 
		\end{align}

	\end{definition}

	\smallskip
	We cite the following H\"older continuity estimates for the variational solution whose proofs
	can be found in \cite{CHP12, Feng}. 
	
	\begin{lemma}\label{thm2.2}
		Suppose ${\bf u}_0 \in L^2\bigl(\Omega; {\mathbb V} \cap \vH^2_{per}(D)\bigr)$ and $\vf \in L^2(\Ome;C^{\frac12}([0,T]);H^1_{per}(D))$. Then there exists a constant $C \equiv C(D_T, \vu_0, \vf)>0$, such that the variational solution to problem \eqref{eq1.1} satisfies
		for $s,t \in [0,T]$
		\begin{subequations}
			\begin{align}\label{eq2.20a}
				& {\mathbb E}\bigl[\|{\bf u}(t)-{\bf u}(s)\|^2 \bigr]+ {\mathbb E}\Bigl[\int_s^t \|\nabla\bigl({\bf u}(\tau)-{\bf u}(s)\bigr)\| ^2 \, d\tau \Bigr]
				\leq C|t-s|\, ,\\ \label{eq2.20b}
				& {\mathbb E}\bigl[\|\nabla \bigl({\bf u}(t)-{\bf u}(s)\bigr)\|^2 \bigr]+ {\mathbb E}\Bigl[\int_s^t \|{\bf A}\bigl({\bf u}(\tau)-{\bf u}(s)\bigr)\|^2\, d\tau\Bigr] 
				\leq C|t-s|\, .
			\end{align}
		\end{subequations}
	\end{lemma}

	\begin{remark}
		We note that to ensure the H\"older continuity estimage \eqref{eq2.20b} is the only reason for restricting to  the periodic boundary condition in this paper. 
	\end{remark}
	
	%%%%%%
	%\subsection{Definition of the pressure}\label{sec-2.3}
	Definition \ref{def2.1} only defines the velocity $\mathbf{u}$ for \eqref{eq1.1}, 
	its associated pressure $p$ is subtle to define. In that regard we quote the following 
	theorem from \cite{Feng1}.

	\begin{theorem}\label{thm 2.2}
		Let $\{{\bf u}(t) ; 0\leq t\leq T\}$ be the variational solution of \eqref{eq1.1}. There exists a unique adapted process 
		{$P\in {L^2\bigl(\Omega, L^2(0,T; H^1_{per}(D)/{\mathbb R})\bigr)}$} such that $(\mathbf{u}, P)$ satisfies $\mP$-a.s.~for all $t\in (0,T]$
		\begin{subequations}\label{eq2.100}
			\begin{align}\label{eq2.10a}
				&\bigl({\bf u}(t),  {\bf v} \bigr) + \int_0^t  \bigl(\nab {\bf u}(s), \nab {\bf v} \bigr) \, ds
				- \bigl(  \div \mathbf{v}, P(t) \bigr) \\
				&=({\bf u}_0, {\bf v}) + \int_0^t \big(\vf(s), \vv\big) \, ds 
				+  {\int_0^t  \bigl( {\bf B}\bigl({\bf u}(s)\bigr), {\bf v} \bigr)\, dW(s)}  \,\,\, \forall  \, {\bf v}\in \vH^1_{per}(D; \mathbb{R}^d)\, , \nonumber \\ 
				&\bigl(\div {\bf u}, q \bigr) =0 \qquad\forall \, q\in L^2_0(D) := \{ q \in L^2_{per}(D):\, (q,1) = 0\}\,  .  \label{eq2.10b}
			\end{align}
		\end{subequations}
	\end{theorem}
	
	System \eqref{eq2.100} is a mixed formulation for the stochastic Stokes 
	system \eqref{eq1.1}, where the (time-averaged) pressure $P$ is defined.
	The distributional derivative $p := \frac{\partial P}{\partial t}$, which was
	shown to belong to $L^1\bigl(\Omega; W^{-1,\infty}(0,T; H^1_{per}(D)/{\mathbb R})\bigr)$,
	was defined as the pressure. Below, we also define another time-averaged ``pressure" 
	\begin{equation}\label{R_t}
		R(t) := P(t) - \int_0^t\xi(s)\, dW(s),
	\end{equation} 
	using the Helmholtz decomposition ${\bf B}(\vu(t)) = \pmb{\eta}(t) + \nabla \xi(t)$, where 
	$\xi\in H^1_{per}(D)/\mathbb{R}$ ${\mathbb P}\mbox{-a.s.}$ such that 
	\begin{equation}\label{eq2.8f} 
		\bigl(\nabla \xi(t), \nabla \phi \bigr) =  \bigl( {\bf B}(\vu(t)) , \nabla \phi 
		\bigr)\qquad \forall\, \phi \in H^1_{per}(D)\, .
	\end{equation}
	Then,  it is easy to check that ${\mathbb P}\mbox{-a.s.}$
	{\begin{equation}\label{eq2.8e}
			\nabla R(t)= -\mathbf{u}(t)+ \int_0^t \vu(s)\, ds+\mathbf{u}_0 + \int_0^t \vf(s)\, ds +\int_0^t \pmb{\eta}(s)\, dW(s) 
			\quad\forall  \, t\in (0,T).
		\end{equation}
	}
	The process $\{R(t); 0\leq t\leq T\}$ will  also be approximated 
	in our numerical methods.

	Next, we establish some stability estimates for the velocity $\vu$ and the pressure $P$ in the following lemma.
		
		\begin{lemma}
			Suppose that $\vu_0 \in L^2(\Ome;\mV)$. Let $(\vu, P)$ solve \eqref{eq2.100}. Then there exists a constant $C \equiv C(D_T, \vu_0,\vf)$ such that
			\begin{align}
				\label{eq2.9}	\mE\Bigl[\sup_{0\leq t \leq T}\|\nab\vu(t)\|^2\Bigr] + \mE\biggl[\int_0^T\|{\bf A}\vu(s)\|^2\, ds\biggr]&\leq C,\\
				\label{eq210}	\sup_{0\leq t \leq T}\mE\Bigl[\|\nab P(t)\|^2\Bigr] &\leq C.
			\end{align}
		\end{lemma}
		\begin{proof}
			The proof of \eqref{eq2.9} is standard by using It\^o's formula and can be found in \cite[Theorem 4.4 and Section 5]{CHP12}. To prove \eqref{eq210},  using 
			\eqref{eq2.10a} can easily get  
			\begin{align}\label{eq2.10}
				\mE\big[\|\nab P(t)\|^2\big] &\leq C\bigg(\mE\big[\|\vu(t)\|^2\big] + \mE\bigg[\int_0^t\|\Delta \vu(s)\|^2\, ds\bigg] + \mE\bigg[\int_0^t \|\vu(s)\|^2\,ds\bigg] \\\nonumber
				&\qquad + \mE\bigg[\int_0^t\|\vf(s)\|^2\, ds\bigg]\bigg).
			\end{align}
			Then, the desired estimate from immediately after taking supreme over all $t \in [0,T]$ and using \eqref{eq2.9}.
		\end{proof}

	We finish this section by establishing the following H\"older continuity result for $P$, which is used for the error analysis in sections.  
	
	\begin{lemma}\label{thm2.4} 
		Suppose that $\vu \in L^2(\Ome;C([0,T];\bV)) \cap L^2(\Ome;0,T;\vH^2_{per}(D))$, $\vf \in L^2(\Ome;$ $0,T; \vL^2_{per}(D))$ and $P \in L^2(\Ome;L^2(0,T;H^1_{per}(D/\mathbb{R})))$. Then, there holds
		\begin{align}\label{holder_pressure}
			\mE\big[\|\nab\big(P(s) - P(t)\big)\|^2\big] \leq C|s-t|\qquad \forall s,t \in [0,T],
		\end{align}
		where the constant $C>0$ depends on $D_T$, $\vu_0$ and $\vf$.
	\end{lemma}
	
	\begin{proof}
		From its definition of $P(t)$, we get 
		\begin{align}
			\nab P(t) = -\vu(t) + \int_0^t \Delta\vu(s)\,ds + \int_0^t \vf(s)\, ds + \int_0^t \vB(\vu(s))\, d\vW(s)
		\end{align}
		Therefore, for any $s,t \in [0,T]$, we have 
		\begin{align}
			\nab\big(P(t) - P(s)\big) &= - \big(\vu(t) - \vu(s)\big) + \int_s^t \Delta \vu(\tau)\, d\tau + \int_{s}^t \vf(\tau)\, d\tau \\\nonumber
			& \qquad + \int_s^t\vB(\vu(\tau))\,dW(\tau).
		\end{align}
		Thus, 
		\begin{align}
			\mE\big[\|\nab\big(P(t) - P(s)\big)\|^2\big] &\leq 4\mE\big[\|\vu(t) - \vu(s)\|^2\big] + 4\mE\bigg[\bigg\|\int_s^t \Delta \vu(\tau)\,d\tau\bigg\|^2\bigg]\\\nonumber
			&\quad  + 4\mE\bigg[\bigg\|\int_{s}^t \vf(\tau)\, d\tau\bigg\|^2  +  \bigg\|\int_s^t \vB(\vu(\tau))\,dW(\tau)\bigg\|^2\bigg]\\\nonumber
			& = \mbox{\tt I} + \mbox{\tt II} + \mbox{\tt III} + \mbox{\tt IV}.
		\end{align}
		The first term \mbox{\tt I} can be controlled by using \eqref{eq2.10a}. For \mbox{\tt II} and \mbox{\tt III}, by  Schwarz inequality, we have
		\begin{align}
			\mbox{\tt II} + \mbox{\tt III} \leq 4\mE\bigg[\int_s^t\|\Delta \vu(\tau)\|^2\, d\tau\bigg]|t-s| + 4\mE\bigg[\int_s^t\|\vf(\tau)\|^2\, d\tau\bigg]|t-s|.
		\end{align}
		In addition, using It\^o isometry and \eqref{eq2.6b}, we obtain
		\begin{align}
			\mbox{\tt IV} = 4\mE\bigg[\int_s^t\|\vB(\vu(\tau))\|^2\, d\tau\bigg] 
			\leq C\mE\bigg[\int_s^t\|\vu(\tau)\|^2\, d\tau\bigg].
		\end{align}
		In summary, we have 
		\begin{align}
			\mE\big[\|\nab\big(P(t) - P(s)\big)\|^2\big] &\leq C|t-s| + 4\mE\bigg[\int_s^t\|\Delta \vu(\tau)\|^2\, d\tau\bigg]|t-s|  \\\nonumber
			&\quad  + 4\mE\bigg[\int_s^t\|\vf(\tau)\|^2\, d\tau\,|t-s| +  C \int_s^t\|\vu(\tau)\|^2\, d\tau\bigg]\\\nonumber
			&\leq \bigg(C + 4\mE\bigg[\int_s^t\|\Delta \vu(\tau)\|^2\, d\tau   +\int_s^t\|\vf(\tau)\|^2\, d\tau\bigg] \\
			&\quad + C\max_{\tau \in [0,T]}\mE\big[\|\vu(\tau)\|^2\big]\bigg)|t-s|.\nonumber
		\end{align}
		Finally, the desired estimate follows from the assumptions on $\vu$ and $\vf$. 
	\end{proof}
	
	%%%%%%
	\section{Two Chorin-type time-stepping schemes}\label{sec3} 
	In this section, we first formulate two Chorin-type semi-discrete-in-time 
	schemes for problem \eqref{eq1.1}. The first scheme is the standard Chorin scheme 
	and the second one is a Helmholtz decomposition enhanced nonstandard Chorin scheme. 
	We then present a complete convergence analysis for each scheme which include 
	establishing their stability and error estimates in strong norms for both velocity and pressure approximations.

	\subsection{Standard Chorin projection scheme}\label{sub3.1} 
	We first consider the standard Chorin scheme for \eqref{eq1.1},
	its formulation is a straightforward adaptation of the well-known scheme for the 
	deterministic Stokes problem and is given in Algorithm 1 below. 
	As mentioned earlier, the standard Chorin scheme for \eqref{eq1.1} was already 
	studied in \cite{CHP12} in the special case when the noise is divergence-free  
	and error estimates were only obtained for the velocity approximation. 
	In contrast, here we consider the Chorin scheme for 
	general multiplicative noise and to derive error estimates 
	in strong norms not only for the velocity but also for pressure approximations
	and to achieve a full understanding about the scheme.

	\subsubsection{\bf Formulation of the standard Chorin scheme}
	Let $M$ be a (large) positive integer and $k:=T/M$ be the time step size. 
	Set $t_j=jk$ for $j=0,1,2,\cdots, M$, then $\{t_j\}_{j=0}^M$ forms a
	uniform mesh for $(0,T)$. The standard Chorin projection scheme is 
	given   as follows (cf. \cite{CHP12,Temam01,Girault_Raviart86}):
	
	\medskip
	\noindent
	{\bf Algorithm 1}
	\smallskip
	
	Let $\tilde{\vu}^0=\vu^0=\vu_0$. For $n=0,1,2,\cdots, M-1$, do the following three steps.
	
	{\em Step 1:} Given $\vu^n \in L^2(\Ome;\bH)$ and $\tilde{\vu}^n \in L^2(\Ome;\vH^1_{per}(D))$, find $\tilde{\vu}^{n+1} \in L^2(\Ome;\vH_{per}^1(D))$ such that $\mP$-a.s.
	\begin{align}
		\label{eq3.1}		\tilde{\vu}^{n+1} - k \Delta\tilde{\vu}^{n+1} & = \vu^n  + k\vf^{n+1} + \vB(\tilde{\vu}^n)\Delta W_{n+1}\qquad\mbox{in } D.
	\end{align}

	\smallskip
	{\em Step 2:} Find $p^{n+1} \in L^2(\Ome;H^1_{per}(D)/\mathbb{R})$ such that $\mP$-a.s.
	%\begin{subequations}
	\begin{align}
		\label{eq3.2a}		-\Delta p^{n+1} &= -\frac{1}{k}\div \tilde{\vu}^{n+1}\qquad \mbox{in } D.
		%	\p_{\bf n} p^{n+1} &= 0\qquad \mbox{on } \p D.
	\end{align}
	
	\smallskip
	{\em Step 3:} Define $\vu^{n+1} \in L^2(\Ome; \bH)$ by
	\begin{align}
		\label{eq3.3a}	\vu^{n+1} &= \tilde{\vu}^{n+1} - k \nab p^{n+1}.
	\end{align}
	
	\begin{remark}
		(a)	The above formulation is written in the way in which the scheme is implemented, it is slightly different from the traditional writing which combines Step 2 and 3 together.
		
		(b) It is easy to check $(\tilde{\vu}^{n+1}, {\vu}^{n+1}, p^{n+1})$ 
		satisfies the following system:
		\begin{subequations}\label{eq3.4}
			\begin{alignat}{2}
				\label{eq3.4a}	\tilde{\vu}^{n+1} - \tilde{\vu}^n - k\Delta \tilde{\vu}^{n+1} + k\nab p^{n} &= k\vf^{n+1} + \vB(\tilde{\vu}^n)\Delta W_{n+1} &&\qquad \mbox{in }D,\\
				\label{eq3.4b}		\div \tilde{\vu}^{n+1} - k\Delta p^{n+1} &= 0 &&\qquad\mbox{in }D,
			\end{alignat}
		\end{subequations}
		where $\tilde{u}^0 = u_0$.
	\end{remark}

	\subsubsection{\bf Stability estimates for the standard Chorin method} 
	{The goal of this subsection is to establish some stability estimates for 
		Algorithm 1 in strong norms.} %These estimates will play an important role in the error estimations for the the fully discrete finite element Chorin scheme to be given in the next section.
	
	\begin{lemma}\label{lemma3.1} 
		The discrete processes $\{(\tilde{\vu}^n,p^n)\}_{n=0}^M$ 
		defined in \eqref{eq3.4} satisfy
		\begin{subequations}
			\begin{align}\label{eq3.5}
				\max_{0\leq n \leq M}\mE [\|\tilde{\vu}^n\|^2] + \mE \bigg[\sum_{n=1}^{M} \|\tilde{\vu}^n - \tilde{\vu}^{n-1}\|^2\bigg] + \mE \bigg[k\sum_{n=0}^{M} \|\nab \tilde{\vu}^n\|^2\bigg] 
				&\leq C, \\
			\label{eq3.5b}	\mE \bigg[k\sum_{n=0}^{M}\|\nab p^n\|^2\bigg] &\leq \frac{C}k,\\
			\label{eq3.56b} 	\max_{0\leq n \leq M}\mE [\|\nab\tilde{\vu}^n\|^2]  + \mE \bigg[k\sum_{n=0}^{M} \|\Delta \tilde{\vu}^n\|^2\bigg]
			&\leq \frac{C}{k},
			\end{align}
		\end{subequations}
		where $C>0$ depends only on $D_T, u_0, \vf$.	
	\end{lemma}
	
	\begin{proof} 
	We first prove \eqref{eq3.5} and \eqref{eq3.5b}.	Testing \eqref{eq3.4a} by $\tilde{u}^{n+1}$ and \eqref{eq3.4b} by $p^n$, and integrating  by parts, we obtain
		\begin{align}\label{eq3.6}		&\big(\tilde{\vu}^{n+1}-\tilde{\vu}^n,\tilde{\vu}^{n+1}\big) 
			+ k\|\nab\tilde{\vu}^{n+1}\|^2 + k\big(\nab p^{n},\tilde{\vu}^{n+1}\big) \\
			&\hskip 1.2in 
			= k\big(\vf^{n+1}, \tilde{\vu}^{n+1}\big) 
			+ \big(\vB(\tilde{\vu}^n)\Delta W_{n+1},\tilde{\vu}^{n+1}\big),  \nonumber \\
			&\big(\tilde{\vu}^{n+1},\nab p^n\big) - k\big(\nab p^{n+1},\nab p^n\big) =0.\label{eq3.7}
		\end{align}
		
		Substituting \eqref{eq3.7} into \eqref{eq3.6} yields 
		\begin{align}\label{eq3.8}	
			\big(\tilde{\vu}^{n+1}-\tilde{\vu}^n,\tilde{\vu}^{n+1}\big) 
			&+ k\|\nab\tilde{\vu}^{n+1}\|^2 + k^2\big(\nab p^{n+1},\nab p^n\big) \\
			& = k\big(\vf^{n+1},\tilde{\vu}^{n+1} \big) +  \big(\vB(\tilde{\vu}^n)\Delta W_{n+1},\tilde{\vu}^{n+1}\big). \nonumber
		\end{align}
		
		Using the identity $2a(a-b) = a^2-b^2 + (a-b)^2$ and the following identity from \eqref{eq3.4b}
		$$k\big(\nab p^{n+1},\nab(p^{n+1}-p^n)\big) = \big(\tilde{\vu}^{n+1}-\tilde{\vu}^n,\nab p^{n+1}\big),$$ and taking expectations on both sides of  \eqref{eq3.8}, we get
		\begin{align}\label{eq3.10}	&\frac{1}{2}\mE\bigl[\|\tilde{\vu}^{n+1}\|^2-\|\tilde{\vu}^{n}\|^2  +\|\tilde{\vu}^{n+1}-\tilde{\vu}^n\|^2 \bigr] + \mE \bigl[ k\|\nab \tilde{\vu}^{n+1}\|^2 
			+ k^2 \|\nab p^{n+1}\|^2 \bigr] \\
			&\quad = k\mE[\big(\tilde{\vu}^{n+1}-\tilde{\vu}^n,\nab p^{n+1}\big)] + k\mE\big[\big(\vf^{n+1} ,\tilde{\vu}^n\big)\big] \nonumber \\
			&\qquad +\mE\bigl[ \big(\vB(\tilde{\vu}^n)\Delta W_{n+1},\tilde{\vu}^{n+1}\big) \bigr]. \nonumber
		\end{align}
		
		Next, we bound each term on the right-hand side of \eqref{eq3.10}. First, by Schwarz and Young's inequalities,   we have
		\begin{align}\label{eq3.11}	
			k\mE\big[\big(\tilde{\vu}^{n+1}-\tilde{\vu}^n,\nab p^{n+1}\big) \big] &\leq \frac{3}{10}\mE\bigl[\|\tilde{\vu}^{n+1}-\tilde{\vu}^n\|^2 \bigr] 
			+ \frac56 k^2\mE \bigl[\|\nab p^{n+1}\|^2 \bigr], \\
			k\mE\big[\big(\vf^{n+1} ,\tilde{\vu}^n\big)\big] &\leq Ck\mE\big[\|\vf^{n+1}\|^2\big] + \frac{k}{4}\mE\big[\|\nab \tilde{\vu}^{n+1}\|^2\big]. \label{eq3.12b}
		\end{align}
		In addition, by It\^o isometry, we have
		\begin{align}\label{eq3.12}	
			\mE\bigl[\big(\vB(\tilde{\vu}^n)\Delta W_{n+1},\tilde{\vu}^{n+1}\big)\bigr] &= \mE\bigl[\big(\vB(\tilde{\vu}^n)\Delta W_{n+1},\tilde{\vu}^{n+1}-\tilde{\vu}^n\big)\bigr]\\
			\nonumber
			&\leq 2\mE\bigl[\|\vB(\tilde{\vu}^n)\Delta W_{n+1}\|^2\bigr] 
			+ \frac{1}{8}\mE \bigl[\|\tilde{\vu}^{n+1}-\tilde{\vu}^n\|^2 \bigr]\\
			\nonumber
			&= 2k\mE\bigl[\|\vB(\tilde{\vu}^n)\|^2\bigr] + \frac{1}{8}\mE \bigl[\|\tilde{\vu}^{n+1}-\tilde{\vu}^n\|^2\bigr]\\
			\nonumber
			&\leq C k\mE\bigl[\|\tilde{\vu}^n\|^2 \bigr]+ \frac{1}{8}\mE \bigl[\|\tilde{\vu}^{n+1}-\tilde{\vu}^n\|^2\bigr].
		\end{align}
		
		Substituting \eqref{eq3.11}--\eqref{eq3.12} into \eqref{eq3.10} gives 
		\begin{align}\label{eq3.13}	&\frac{1}{2}\mE\bigl[\|\tilde{\vu}^{n+1}\|^2-\|\tilde{\vu}^{n}\|^2\bigr] + %\Bigl(\frac{3}{8}-\frac{1}{4\delta_1}\Bigr)
			\frac{1}{20}\mE \bigl[\|\tilde{\vu}^{n+1}-\tilde{\vu}^n\|^2 \bigr] + \frac{3k}{4}\mE \bigl[\|\nab \tilde{\vu}^{n+1}\|^2 \bigr]\\ \nonumber
			&\qquad\qquad\qquad + %\big(1-\delta_1\big)
			\frac16 k^2\mE\bigl[ \|\nab p^{n+1}\|^2\bigr] 
			\leq Ck\mE\big[\|\vf^{n+1}\|^2\big] + Ck\mE \bigl[\|\tilde{\vu}^n\|^2\bigr].
		\end{align}
		Applying the summation operator $\sum_{n=0}^{m}$ to \eqref{eq3.13} for any $0\leq m \leq M-1$ yields 
		\begin{align}
			&\frac{1}{2}\mE\big[\|\tilde{\vu}^{m+1}\|^2-\|\tilde{\vu}^{0}\|^2\big]+ \frac{1}{20} \sum_{n=0}^{m}\mE\bigl[ \|\tilde{\vu}^{n+1}-\tilde{\vu}^n\|^2 \bigr] 
			+ k\sum_{n=0}^{m}\mE \bigl[\|\nab \tilde{\vu}^{n+1}\|^2 \bigr] \\
			&\qquad + \frac16 k^2\sum_{n=0}^{m}\mE \bigl[\|\nab p^{n+1}\|^2 \bigr]  
			\leq Ck\sum_{n=0}^{m}\mE \bigl[\|\tilde{u}^n\|_{L^2}^2 \bigr] + Ck\sum_{n=0}^{m}\mE\big[\|\vf^{n+1}\|^2\big]. \nonumber
		\end{align}
		
		Finally, by the discrete Gronwall's inequality we obtain 
		\begin{align*}
			\frac{1}{2}\max_{1\leq n \leq M}\mE\bigl[\|\tilde{\vu}^{n}\|^2\bigr] &+ \frac1{20}\sum_{n=1}^{M}\mE \bigl[\|\tilde{\vu}^{n}-\tilde{\vu}^{n-1}\|^2\bigr]
			+ \frac{3k}{4}\sum_{n=1}^{M}\mE \bigl[\|\nab \tilde{\vu}^{n}\|^2\bigr]\\
			& +  \frac16 k^2\sum_{n=1}^{M}\mE \bigl[\|\nab p^{n}\|^2\bigr] 
			\leq C\bigg(\mE \bigl[\|\vu_0\|_{L^2}^2\bigr] + k\sum_{n=0}^{m}\mE\big[\|\vf^{n+1}\|^2\big]\bigg),
		\end{align*}
		which implies the desired estimate \eqref{eq3.5} and \eqref{eq3.5b}.

It remains to show \eqref{eq3.56b}. To the end, testing \eqref{eq3.4a} by $-\Delta \tilde{\vu}^{n+1}$ we get 
	\begin{align}
		&\frac12\bigl[\|\nab\tilde{\vu}^{n+1}\|^2 - \|\nab\tilde{\vu}^n\|^2 + \|\nab(\tilde{\vu}^{n+1} - \tilde{\vu}^n)\|^2\bigr] + k\|\Delta\tilde{\vu}^{n+1}\|^2\\\nonumber
		&\qquad = k\bigl(\nab p^n,\Delta \tilde{\vu}^{n+1}\bigr) -k\bigl(\vf^{n+1},\Delta \tilde{\vu}^{n+1}\bigr) + \bigl(\nab\vB(\tilde{\vu}^{n})\Delta W_{n+1}, \nab\tilde{\vu}^{n+1}\bigr).
	\end{align}
Proceeding with similar arguments as used above yields  
	\begin{align}\label{eq3.16}
		&\frac12\mE\bigl[\|\nab\tilde{\vu}^{m+1}\|^2\bigr] + \frac12 k\sum_{n=0}^{m}\mE\bigl[\|\Delta\tilde{\vu}^{n+1}\|^2\bigr] \\
		&\hskip 0.8in  \leq \frac12 \mE[\|\nab\vu_0\|^2] + k\sum_{n=0}^m\mE\Bigl[\|\nab p^n\|^2   
		+ \|\vf^{n+1}\|^2  + C \|\nab\tilde{\vu}^n\|^2 \Bigr]. \nonumber
	\end{align}
Then  by the discrete Gronwall inequality we get 
	\begin{align*}
	&\frac12\mE\bigl[\|\nab\tilde{\vu}^{m+1}\|^2\bigr] + \frac12 k\sum_{n=0}^{m}\mE\bigl[\|\Delta\tilde{\vu}^{n+1}\|^2\bigr]\\\nonumber
	&\qquad \leq \Bigl(\frac12 \mE[\|\nab\vu_0\|^2] + k\sum_{n=0}^m\mE\bigl[\|\nab p^n\|^2\bigr]+ k\sum_{n=0}^m\mE\bigl[\|\vf^{n+1}\|^2\bigr]\Bigr)\exp(Ct_m),
	\end{align*}
which and \eqref{eq3.5b} immediately infer \eqref{eq3.56b}. The proof is complete.
\end{proof}
	
%%%%%%%%%%%%%%%
\subsubsection{\bf Error estimates for the standard Chorin scheme}  
	In this subsection we shall derive some error estimates for the time-discrete processes 
	generated by Algorithm 1. To the best of our knowledge, these are the first error estimate results for the standard Chorin scheme in the case general multiplicative noises. 
	For the sake of brevity, but without loss of the generality, we set $\vf = 0$ in this subsection. 
	
	First, we state the following error estimate result for the velocity. 
	
	\begin{theorem}\label{thm3.3} 
		Let  $\{(\tilde{\vu}^n,p^n)\}_{n=0}^M$ be generated by 
		Algorithm 1, then there exists a positive constant $C$ which depends 
		on $D_T, \vu_0,$ and $\vf$ such that
		\begin{align}
			\biggl(\mE\biggl[k\sum_{n=0}^M \|\vu(t_n) &- \tilde{\vu}^n\|^2\biggr]\biggr)^{\frac12} \\\nonumber
			&+ \max_{0\leq \ell \leq M}\biggl(\mE\biggl[\Bigl\|k\sum_{n=0}^{\ell}\nab(\vu(t_n) - \tilde{\vu}^n)\Bigr\|^2\biggr]\biggr)^{\frac12} \leq C k^{\frac14}.
		\end{align}
	\end{theorem}
	
	\begin{proof}
		Let $\ve_{\tilde{\vu}}^m = \vu(t_m) - \tilde{\vu}^m$ and $\displaystyle \E_p^m = P(t_m) - k\sum_{n=0}^m p^n$. 
		Obviously, $\ve_{\tilde{\vu}}^m \in L^2\big(\Ome;\vH^1_{per}(D)\big)$ and $\E_p^m \in L^2(\Ome; H^1_{per}(D)/\mathbb{R})$. In addition, from \eqref{eq2.10a}, we have
		\begin{align}\label{eq3.19}
			&\big(\vu(t_{m+1}) , \vv\big) + \sum_{n=0}^m\int_{t_n}^{t_{n+1}}\big(\nab\vu(s),\nab \vv\big)\,ds + \big(\vv, \nab P(t_{m+1})\big) \\\nonumber
			&\quad = \bigl(\vu_0, \vv\bigr)+  \biggl(\sum_{n=0}^m\int_{t_n}^{t_{n+1}}\big(\vB(\vu(s))\, dW(s), \vv\biggr)\quad  \forall \vv \in \vH^1_{per}(D)\,\,  a.s.
		\end{align}
		Applying the summation operator $\sum_{n=0}^m$ to \eqref{eq3.4a} yields
		\begin{align}\label{eq3.20}
		\bigl(\tilde{\vu}^{m+1},\vv\bigr) &+ k\sum_{n=0}^{m}\bigl(\nab\tilde{\vu}^{n+1},\nab\vv\bigr) + \Bigl(k\sum_{n=0}^m\nab p^n,\vv\Bigr)\\\nonumber
		&=\bigl(\tilde{\vu}_0,\vv\bigr) + \biggl(\sum_{n=0}^m\int_{t_n}^{t_{n+1}}\vB(\tilde{\vu}^n)\, dW(s),\vv\biggr).
		\end{align}
		Subtracting \eqref{eq3.20} from \eqref{eq3.19} we get 
		\begin{align}\label{eq3.21}
			&\big(\ve_{\tilde{\vu}}^{m+1}, \vv\big) + k\Bigl(\sum_{n=0}^m\nab\ve_{\tilde{\vu}}^{n+1},\nab\vv\Bigr) + \big(\nab\E_p^m,\vv\big) \\\nonumber
			&\quad = \sum_{n=0}^m\int_{t_n}^{t_{n+1}}\big(\nab(\vu(t_{n+1}) - \vu(s)),\nab\vv\big)\, ds - \bigl(\nab(P(t_{m+1}) - P(t_m)), \vv\bigr)\\
			&\qquad + \biggl(\sum_{n=0}^m\int_{t_n}^{t_{n+1}}(\vB(\vu(s)) - \vB(\tilde{\vu}^n))\, dW(s),\vv\biggr)\quad \forall\vv\in \vH_{per}^1(D). \nonumber
		\end{align}
	Setting $\vv = \ve_{\tilde{\vu}}^{m+1}$ in \eqref{eq3.21} we obtain 
	\begin{align}\label{eq322}
		&\|\ve_{\tilde{\vu}}^{m+1}\|^2 + k\Bigl(\sum_{n=0}^{m+1}\nab\ve_{\tilde{\vu}}^{n},\nab\ve_{\tilde{\vu}}^{m+1}\Bigr) + \big(\nab\E_p^m,\ve_{\tilde{\vu}}^{m+1}\big) \\\nonumber
		& \quad = \sum_{n=0}^m\int_{t_n}^{t_{n+1}}\big(\nab(\vu(t_{n+1}) - \vu(s)),\nab\ve_{\tilde{\vu}}^{m+1}\big)\, ds\\\nonumber
		&\qquad - \bigl(\nab(P(t_{m+1}) - P(t_m)), \ve_{\tilde{\vu}}^{m+1}\bigr)\\
		&\qquad + \biggl(\sum_{n=0}^m\int_{t_n}^{t_{n+1}}(\vB(\vu(s)) - \vB(\tilde{\vu}^n))\, dW(s),\ve_{\tilde{\vu}}^{m+1}\biggr). \nonumber
	\end{align}
		Similarly, by \eqref{eq2.10b} and \eqref{eq3.4b} we get 
		\begin{align}\label{eq19}
			\div \ve_{\tilde{\vu}}^{n} + k\Delta p^{n} = 0.
		\end{align}
	Applying the summation $\sum_{n=0}^{m+1}$ to \eqref{eq19} and then adding $\pm \Delta P(t_{m+1})$ yield 
		\begin{align*}%\label{eq3.22}
		\div\biggl(\sum_{n=0}^{m+1} \ve_{\tilde{\vu}}^n\biggr) - \Delta \E_p^{m+1} = - \Delta P(t_{m+1}).
		\end{align*}
	Therefore,
	\begin{align}\label{eq3.23}
		\div\ve_{\tilde{\vu}}^{m+1} -\Delta(\E_p^{m+1} - \E_p^m) = -\Delta(P(t_{m+1}) - P(t_m)).
	\end{align}
Testing \eqref{eq3.23} by any $q \in L^2(\Ome; H^1_{per}(D)/\mathbb{R})$, we have
\begin{align}\label{eq323}
	\bigl(\ve_{\tilde{\vu}}^{m+1},\nab q\bigr) = \bigl(\nab(\E_p^{m+1} - \E_p^m),\nab q\bigr) - \bigl(\nab(P(t_{m+1}) - P(t_m)),\nab q\bigr).
\end{align}
Choosing $q = \E_p^m$ in \eqref{eq323} gives 
\begin{align}\label{eq3.24}
	\bigl(\ve_{\tilde{\vu}}^{m+1}, \nab\E_p^{m}\bigr) &= \bigl(\nab(\E_p^{m+1} - \E_p^m),\nab \E_p^m\bigr) - \bigl(\nab(P(t_{m+1}) - P(t_m)),\nab \E_p^m\bigr)\\\nonumber
	&= \bigl(\nab(\E_p^{m+1} - \E_p^m), \nab\E_p^{m+1}\bigr) - \|\nab(\E_p^{m+1} - \E_p^m)\|^2 \\\nonumber
	&\qquad\qquad- \bigl(\nab(P(t_{m+1}) - P(t_m)),\nab \E_p^m\bigr).
\end{align}
Substituting  \eqref{eq3.24} into \eqref{eq322} we obtain
	\begin{align*} %\label{eq3.25}
	&\|\ve_{\tilde{\vu}}^{m+1}\|^2 + k\biggl(\sum_{n=0}^{m+1}\nab\ve_{\tilde{\vu}}^{n},\nab\ve_{\tilde{\vu}}^{m+1}\biggr) + \bigl(\nab(\E_p^{m+1} - \E_p^m),\nab\E_p^{m+1}\bigr) \\\nonumber
	&\quad  = \sum_{n=0}^m\int_{t_n}^{t_{n+1}}\big(\nab(\vu(t_{n+1}) - \vu(s)),\nab\ve_{\tilde{\vu}}^{m+1}\big)\, ds\\\nonumber
	&\qquad + \|\nab(\E_p^{m+1} - \E_p^m)\|^2 +\bigl(\nab(P(t_{m+1}) - P(t_m)), \nab \E_p^m\bigr)\\\nonumber
	&\qquad - \bigl(\nab(P(t_{m+1}) - P(t_m)), \ve_{\tilde{\vu}}^{m+1}\bigr)\\\nonumber
	&\qquad + \biggl(\sum_{n=0}^m\int_{t_n}^{t_{n+1}}(\vB(\vu(s)) - \vB(\tilde{\vu}^n))\, dW(s),\ve_{\tilde{\vu}}^{m+1}\biggr)\\\nonumber
	&\quad \leq  \sum_{n=0}^m\int_{t_n}^{t_{n+1}}\big(\nab(\vu(t_{n+1}) - \vu(s)),\nab\ve_{\tilde{\vu}}^{m+1}\big)\, ds\\\nonumber
	&\qquad + \|\nab(\E_p^{m+1} - \E_p^m)\|^2 +\bigl(\nab(P(t_{m+1}) - P(t_m)), \nab \E_p^m\bigr)\\\nonumber
	&\qquad + \|\nab(P(t_{m+1}) - P(t_m))\|^2 + \frac14 \|\ve_{\tilde{\vu}}^{m+1}\|^2\\\nonumber
	&\qquad + \biggl\|\sum_{n=0}^m\int_{t_n}^{t_{n+1}}(\vB(\vu(s)) - \vB(\tilde{\vu}^n))\, dW(s)\biggr\|^2 + \frac14\|\ve_{\tilde{\vu}}^{m+1}\|^2.
\end{align*}	
Therefore,
\begin{align}\label{eq3.26}
		&\frac12\|\ve_{\tilde{\vu}}^{m+1}\|^2 + k\biggl(\sum_{n=0}^{m+1}\nab\ve_{\tilde{\vu}}^{n},\nab\ve_{\tilde{\vu}}^{m+1}\biggr) + \bigl(\nab(\E_p^{m+1} - \E_p^m),\nab\E_p^{m+1}\bigr) \\\nonumber
		&\quad \leq  \sum_{n=0}^m\int_{t_n}^{t_{n+1}}\big(\nab(\vu(t_{n+1}) - \vu(s)),\nab\ve_{\tilde{\vu}}^{m+1}\big)\, ds\\\nonumber
		&\qquad + \|\nab(P(t_{m+1}) - P(t_m))\|^2 +\|\nab(\E_p^{m+1} - \E_p^m)\|^2\\\nonumber
		&\qquad +\bigl(\nab(P(t_{m+1}) - P(t_m)), \nab \E_p^m\bigr)\\\nonumber
		&\qquad + \biggl\|\sum_{n=0}^m\int_{t_n}^{t_{n+1}}(\vB(\vu(s)) - \vB(\tilde{\vu}^n))\, dW(s)\biggr\|^2.
\end{align}
Using the identity $2a(a-b) = a^2 - b^2 + (a-b)^2$ in \eqref{eq3.26} yields
\begin{align}\label{eq3.27}
	\frac12 \|\ve_{\tilde{\vu}}^{m+1}\|^2 &+ \frac{k}{2}\biggl[\biggl\|\sum_{n=0}^{m+1}\nab\ve_{\tilde{\vu}}^n\biggr\|^2 - \biggl\|\sum_{n=0}^{m}\nab\ve_{\tilde{\vu}}^n\biggr\|^2 + \|\nab\ve_{\tilde{\vu}}^{m+1}\|^2\biggr]\\\nonumber
	&+ \frac12\bigl[\|\nab\E_p^{m+1}\|^2 -\|\nab\E_p^m\|^2 + \|\nab(\E_p^{m+1} - \E_p^m)\|^2\bigr]\\\nonumber
	&\leq \sum_{n=0}^m\int_{t_n}^{t_{n+1}}\big(\nab(\vu(t_{n+1}) - \vu(s)),\nab\ve_{\tilde{\vu}}^{m+1}\big)\, ds\\\nonumber
	&\qquad + \|\nab(P(t_{m+1}) - P(t_m))\|^2 +\|\nab(\E_p^{m+1} - \E_p^m)\|^2\\\nonumber
	&\qquad +\bigl(\nab(P(t_{m+1}) - P(t_m)), \nab \E_p^m\bigr)\\\nonumber
	&\qquad + \Bigl\|\sum_{n=0}^m\int_{t_n}^{t_{n+1}}(\vB(\vu(s)) - \vB(\tilde{\vu}^n))\, dW(s)\Bigr\|^2.
\end{align} 

Next, we apply the summation operator $k\sum_{m=0}^{\ell}$ for $0\leq \ell \leq M-1$, followed by applying the expectation operator $\mE[\cdot]$, to \eqref{eq3.27} to obtain
\begin{align}\label{eq3.28}
	&k\sum_{m=0}^{\ell}\mE\bigl[\|\ve_{\tilde{\vu}}^{m+1}\|^2\bigr] + \mE\biggl[\biggl\|k\sum_{m=0}^{\ell + 1}\nab\ve_{\tilde{\vu}}^m \biggr\|^2\biggr] + k^2\sum_{m=0}^{\ell}\mE\bigl[\|\nab\ve_{\tilde{\vu}}^{m+1}\|^2\bigr]\\\nonumber
	&\qquad + k\mE\bigl[\|\nab\E_p^{\ell + 1}\|^2\bigr] + k\sum_{m=0}^{\ell}\mE\bigl[\|\nab(\E_p^{m+1} - \E_p^m)\|^2\bigr]\\\nonumber
	&\quad \leq 2\mE\biggl[k\sum_{m=0}^{\ell}\sum_{n=0}^m\int_{t_n}^{t_{n+1}}\big(\nab(\vu(t_{n+1}) - \vu(s)),\nab\ve_{\tilde{\vu}}^{m+1}\big)\, ds\biggr]\\\nonumber
	&\qquad + 2\mE\biggl[k\sum_{m=0}^{\ell}\|\nab(P(t_{m+1}) - P(t_m))\|^2\biggr] \\\nonumber
	&\qquad + 2\mE\biggl[k\sum_{m=0}^{\ell}\|\nab(\E_p^{m+1} - \E_p^m)\|^2\biggr]\\\nonumber
	&\qquad +2\mE\biggl[k\sum_{m=0}^{\ell}\bigl(\nab(P(t_{m+1}) - P(t_m)), \nab \E_p^m\bigr)\biggr]\\\nonumber
	&\qquad + 2\mE\biggl[k\sum_{m=0}^{\ell}\Bigl\|\sum_{n=0}^m\int_{t_n}^{t_{n+1}}(\vB(\vu(s)) - \vB(\tilde{\vu}^n))\, dW(s)\Bigr\|^2\biggr]\\\nonumber
	&:= {\tt I + II + III + IV + V}.
\end{align}

Now we estimate each term on the right-hand side of \eqref{eq3.28} as follows.

By using the discrete and continuous H\"older inequality estimates \eqref{eq2.20b},  \eqref{eq2.9} and \eqref{eq3.5}, we obtain
\begin{align}\label{eq3.29}
	{\tt I} &= 2\mE\Bigl[k\sum_{m=0}^{\ell}\Bigl(\sum_{n=0}^m\int_{t_n}^{t_{n+1}}\nab(\vu(t_{n+1}) - \vu(s))\, ds,\nab\ve_{\tilde{\vu}}^{m+1}\Bigr)\Bigr]\\\nonumber
	&\leq  2\mE\biggl[k\sum_{m=0}^{\ell}\biggl\|\sum_{n=0}^m\int_{t_n}^{t_{n+1}}\nab(\vu(t_{n+1}) - \vu(s))\, ds\biggr\|\,\|\nab\ve_{\tilde{\vu}}^{m+1}\|\biggr]\\\nonumber
	&\leq 2\mE\biggl[\Bigl(k\sum_{m=0}^{\ell}\biggl\|\sum_{n=0}^{m}\int_{t_n}^{t_{n+1}}\nab(\vu(t_{n+1}) - \vu(s))\, ds\biggr\|^2\Bigr)^{\frac12}\,\Bigl(k\sum_{m=0}^{\ell}\|\nab\ve_{\tilde{\vu}}^n\|^2\Bigr)^{\frac12}\biggr]\\\nonumber
	&\leq C\mE\biggl[\Bigl(k\sum_{m=0}^{\ell}\sum_{n=0}^{m}\int_{t_n}^{t_{n+1}}\|\nab(\vu(t_{n+1}) - \vu(s))\|^2\, ds\Bigr)^{\frac12}\,\Bigl(k\sum_{m=0}^{\ell}\|\nab\ve_{\tilde{\vu}}^n\|^2\Bigr)^{\frac12}\biggr]\\\nonumber
	&\leq C\Bigl(\mE\Bigl[k\sum_{m=0}^{\ell}\sum_{n=0}^{m}\int_{t_n}^{t_{n+1}}\|\nab(\vu(t_{n+1})-\vu(s))\|^2\, ds\Bigr]\Bigr)^{\frac12}\Bigl(\mE\Bigl[k\sum_{m=0}^{\ell}\|\nab\ve_{\tilde{\vu}}^n\|^2\Bigr]\Bigr)\\\nonumber
	&\leq Ck^{\frac12}.
\end{align}

Next, by using \eqref{holder_pressure} we have
\begin{align}
	{\tt II } &=  2\mE\biggl[k\sum_{m=0}^{\ell}\|\nab(P(t_{m+1}) - P(t_m))\|^2\biggr]  \leq Ck.
\end{align}
By using \eqref{holder_pressure} and the stability estimate \eqref{eq3.5b} we obtain
\begin{align}\label{eq3.31}
	{\tt III} &= 2\mE\biggl[k\sum_{m=0}^{\ell}\|\nab(\E_p^{m+1} - \E_p^m)\|^2\biggr]\\\nonumber
	&= 2\mE\biggl[k\sum_{m=0}^{\ell}\|\nab(P(t_{m+1}) - P(t_m)) - k\nab p^{m+1}\|^2\biggr]\\\nonumber
	&\leq C\mE\biggl[k\sum_{m=0}^{\ell}\|\nab(P(t_{m+1}) - P(t_m))\|^2\biggr] + Ck\mE\biggl[k^2\sum_{m=0}^{\ell}\|\nab p^{m+1}\|^2\biggr] \\
	&\leq Ck.\nonumber
\end{align}
It follows from the It\^o isometry and \eqref{eq2.20a} that 
\begin{align}
	{\tt V} &= 2k\sum_{m=0}^{\ell}\mE\biggl[\biggl\|\sum_{n=0}^m\int_{t_n}^{t_{n+1}}\bigl(\vB(\vu(s)) - \vB(\tilde{\vu}^n)\bigr)\, dW(s)\biggr\|^2\biggr]\\\nonumber
	&= 2k\sum_{m=0}^{\ell}\mE\biggl[\sum_{n=0}^m\int_{t_n}^{t_{n+1}}\|\vB(\vu(s)) - \vB(\tilde{\vu}^n)\|^2\, ds\biggr]\\\nonumber
	&\leq Ck\sum_{m=0}^{\ell}\sum_{n=0}^m\int_{t_n}^{t_{n+1}}\mE[\|\vu(s) - \tilde{\vu}^n\|^2]\, ds \\\nonumber
	&\leq Ck\sum_{m=0}^{\ell}k\sum_{n=0}^{m}\mE[\|\ve_{\vu}^n\|^2] + Ck\sum_{m=0}^{\ell}\sum_{n=0}^m\int_{t_n}^{t_{n+1}}\mE[\|\vu(s) - \vu(t_n)\|^2]\, ds \\\nonumber
	&\leq  Ck\sum_{m=0}^{\ell}k\sum_{n=0}^{m}\mE[\|\ve_{\vu}^n\|^2] +Ck.
\end{align}

To bound term {\tt IV}, we first derive its rewriting as follows:
\begin{align}\label{eq3.32}
	{\tt IV} &= 2\mE\biggl[k\sum_{m=0}^{\ell}\bigl(\nab(P(t_{m+1}) - P(t_m)), \nab \E_p^m\bigr)\biggr]\\\nonumber
	&=2\mE\biggl[k\sum_{m=0}^{\ell}\bigl(\nab(P(t_{m+1}) - P(t_m)), \nab \E_p^{m+1}\bigr)\biggr] \\\nonumber
	&\quad + 2\mE\biggl[k\sum_{m=0}^{\ell}\bigl(\nab(P(t_{m+1}) - P(t_m)), \nab (\E_p^m - \E_p^{m+1})\bigr)\biggr].
\end{align}
By using the summation by parts, the first term above can be rewritten as 
\begin{align}\label{eq3.33}
	&2\mE\biggl[k\sum_{m=0}^{\ell}\bigl(\nab(P(t_{m+1}) - P(t_m)), \nab \E_p^{m+1}\bigr)\biggr] \\\nonumber
	&\quad = 2k\mE\bigl[\bigl(\nab P(t_{\ell+1}),\nab\E_p^{\ell +1}\bigr)\bigr] -2k\mE\biggl[\sum_{m=0}^{\ell}\bigl(\nab P(t_m),\nab(\E_p^{m+1} - \E_p^m)\bigr)\biggr].
\end{align}
Substituting \eqref{eq3.33} into \eqref{eq3.32} yields 
\begin{align} \label{eq3.35}
	{\tt IV} &= 2k\mE\bigl[\bigl(\nab P(t_{\ell+1}),\nab\E_p^{\ell +1}\bigr)\bigr] -2k\mE\biggl[\sum_{m=0}^{\ell}\bigl(\nab P(t_m),\nab(\E_p^{m+1} - \E_p^m)\bigr)\biggr] \\\nonumber
	&\quad+2\mE\biggl[k\sum_{m=0}^{\ell}\bigl(\nab(P(t_{m+1}) - P(t_m)), \nab (\E_p^m - \E_p^{m+1})\bigr)\biggr]\\\nonumber
	&:= {\tt IV_1 + IV_2 + IV_3}.
\end{align}

We now bound each term above. Using the stability \eqref{eq210} we get
\begin{align} \label{eq3.36}
	{\tt IV_1} &\leq Ck\mE\bigl[\|\nab P(t_{\ell + 1})\|^2\bigr] + \frac{k}{4}\mE\bigl[\|\nab\E_p^{\ell+1}\|^2\bigr]\\\nonumber
	&\leq Ck + \frac{k}{4}\mE\bigl[\|\nab\E_p^{\ell+1}\|^2\bigr].
\end{align}
Expectedly,  the term $\frac{k}{4}\mE\bigl[\|\nab\E_p^{\ell+1}\|^2\bigr]$ will be absorbed to the left side of \eqref{eq3.28} later.

To bound term ${\tt IV_2}$, we reuse the estimation from {\tt III} in \eqref{eq3.31} together with the stability of $P$ given in \eqref{eq210} to get 
\begin{align}
	{\tt IV_2} &\leq 2\mE\biggl[k\sum_{m=0}^{\ell}\|\nab P(t_m)\|\,\|\nab(\E_p^{m+1} - \E_p^m)\|\biggr] \\\nonumber
	&\leq C\mE\biggl[\Bigl(k\sum_{m=0}^{\ell}\|\nab P(t_m)\|^2\Bigr)^{\frac12}\, \Bigl(k\sum_{m=0}^{\ell}\|\nab(\E_p^{m+1} - \E_p^m)\|^2\Bigr)^{\frac12}\biggr]\\\nonumber
	&\leq C\Bigl(\mE\Bigl[k\sum_{m=0}^{\ell}\|\nab P(t_m)\|^2\Bigr]\Bigr)^{\frac12}\, \Bigl(\mE\Bigl[k\sum_{m=0}^{\ell}\|\nab(\E_p^{m+1} - \E_p^m)\|^2\Bigr]\Bigr)^{\frac12}\\\nonumber
	&\leq Ck^{\frac12}.
\end{align}

Using again \eqref{eq3.31} and \eqref{holder_pressure} we have
\begin{align}\label{eq3.38}
	{\tt IV_3} &\leq Ck\sum_{m=0}^{\ell}\|\nab(P(t_{m+1}) - P(t_m))\|^2 + C\mE\biggl[k\sum_{m=0}^{\ell}\|\nab(\E_p^{m+1} - \E_p^m)\|^2\biggr]\\\nonumber
	&\leq Ck.
\end{align}

Substituting estimates \eqref{eq3.36}--\eqref{eq3.38} into \eqref{eq3.35} yields
\begin{align}\label{eq3.39}
	{\tt IV} \leq Ck^{\frac12} + \frac{k}{4}\mE\bigl[\|\nab\E_p^{\ell+1}\|^2\bigr].
\end{align}
Now, substituting the estimates for {\tt I, II, III, IV, V} into \eqref{eq3.28} and using the notation $X^{\ell} = k\sum_{m=0}^{\ell}\mE[\|\ve_{\tilde{\vu}}^m\|^2]$ we obtain
\begin{align*} %\label{eq3.40}
	&X^{\ell +1} + \mE\Bigl[\Bigl\|k\sum_{m=0}^{\ell + 1}\nab\ve_{\tilde{\vu}}^m \Bigr\|^2\Bigr] + k^2\sum_{m=0}^{\ell}\mE\bigl[\|\nab\ve_{\tilde{\vu}}^{m+1}\|^2\bigr]\\\nonumber
	&\quad + \frac{3k}{4}\mE\bigl[\|\nab\E_p^{\ell + 1}\|^2\bigr] + k\sum_{m=0}^{\ell}\mE\bigl[\|\nab(\E_p^{m+1} - \E_p^m)\|^2\bigr]\\\nonumber
	&\leq Ck^{\frac12} + Ck\sum_{m=0}^{\ell}X^m. 
%	&\leq Ck^{\frac12}\exp(Ct_{\ell}),
\end{align*}
Thus, it follows from the discrete Gronwall inequality that 
\begin{align*} %\label{eq3.40}
&X^{\ell +1} + \mE\Bigl[\Bigl\|k\sum_{m=0}^{\ell + 1}\nab\ve_{\tilde{\vu}}^m \Bigr\|^2\Bigr] + k^2\sum_{m=0}^{\ell}\mE\bigl[\|\nab\ve_{\tilde{\vu}}^{m+1}\|^2\bigr]\\\nonumber
&\quad + \frac{3k}{4}\mE\bigl[\|\nab\E_p^{\ell + 1}\|^2\bigr] + k\sum_{m=0}^{\ell}\mE\bigl[\|\nab(\E_p^{m+1} - \E_p^m)\|^2\bigr]\\\nonumber
&\leq Ck^{\frac12}\exp(Ct_{\ell}),
\end{align*}
which yields the desired error estimate for the velocity approximation.
\end{proof}

Next, we derive an error estimate for the pressure approximation. Our main idea is to
estimate the pressure error in a time averaged fashion. 
	
	\begin{theorem}\label{thm3.4} 
		Let $\{(\tilde{\vu}^m,p^m)\}_{m=0}^M$ be generated by Algorithm 1. Then, there exists a positive constant $C$ which depends on $D_T,\vu_0,\vf,$ and $\beta$ such that 
		\begin{align}\label{eq3.41}
			\bigg(\mE\bigg[k\sum_{m=0}^M\Bigl\|P(t_m) - k\sum_{n=0}^m p^n\Bigr\|^2\bigg]\bigg)^{\frac12} \leq Ck^{\frac14},
		\end{align}
		where $\beta$ denotes the stochastic inf-sup constant (see below).
	\end{theorem} 
	
	\begin{proof}
		We first recall the following inf-sup condition (cf. \cite{BS2008}): 
		\begin{align}\label{inf-sup}
			\sup_{\vv \in \vH^1_{per}(D)} \frac{\bigl(q,\div \vv\bigr)}{\|\nab\vv\|}\geq 	\beta\, \|q\| \qquad\forall q\in  L^2_{per}(D)/\mathbb{R},
		\end{align}
		where $\beta$ is a positive constant.
		
		Below we reuse all the notations from Theorem \ref{thm3.3}. From the error equation  \eqref{eq3.21} we obtain for all $\vv\in \vH^1_{per}(D)$
		\begin{align}\label{eq3.43}
			\bigl(\E_p^m,\div \vv\bigr) &= \bigl(\ve_{\tilde{\vu}}^{m+1},\vv\bigr) + \biggl(k\sum_{n=0}^{m}\nab\ve_{\tilde{\vu}}^{n+1},\nab\vv\biggr)\\\nonumber
			&\quad-\biggl(\sum_{n=0}^m\int_{t_n}^{t_{n+1}}\nab(\vu(t_{n+1}) - \vu(s))\, ds, \nab\vv\biggr)\\\nonumber
			&\quad+ \bigl(\nab(P(t_{m+1}) - P(t_m)), \vv\bigr)\\\nonumber
			&\quad- \biggl(\sum_{n=0}^m\int_{t_n}^{t_{n+1}}(\vB(\vu(s)) - \vB(\tilde{\vu}^n))\, dW(s), \vv\biggr) .
		\end{align}
	By using Schwarz and Poincar\'e inequalities on the right side of \eqref{eq3.43}, we get
	\begin{align}
		\bigl(\E_p^m,\div \vv\bigr) &\leq C\|\ve_{\tilde{\vu}}^{m+1}\|\|\nab\vv\| + \biggl\|k\sum_{n=0}^m\nab\ve_{\tilde{\vu}}^{n+1}\biggr\|\|\nab\vv\|\\\nonumber
		&\quad + \biggl\|\sum_{n=0}^m\int_{t_n}^{t_{n+1}}\nab(\vu(t_{n+1})-\vu(s))\, ds\biggr\|\|\nab\vv\|\\ \nonumber
		&\quad + C\|\nab(P(t_{m+1}) - P(t_m))\|\|\nab\vv\|\\\nonumber
		&\quad + C\biggl\|\sum_{n=0}^m\int_{t_n}^{t_{n+1}}(\vB(\vu(s)) - \vB(\tilde{\vu}^n))\, dW(s)\biggr\| \|\nab\vv\|
	\end{align}
Applying \eqref{inf-sup} yields 
\begin{align}\label{eq3.45}
	\beta\|\E_p^m\| &\leq \sup_{\vv \in \vH^1_{per}(D)}\frac{\bigl(\E_p^m,\div \vv\bigr)}{\|\nab\vv\|}\\\nonumber
	&\leq C\|\ve_{\tilde{\vu}}^{m+1}\| + \biggl\|k\sum_{n=0}^m\nab\ve_{\tilde{\vu}}^{n+1}\biggr\|\\\nonumber
	&\quad+\biggl\|\sum_{n=0}^m\int_{t_n}^{t_{n+1}}\nab(\vu(t_{n+1})-\vu(s))\, ds\biggr\|\\
	\nonumber
	&\quad+ C\|\nab(P(t_{m+1}) - P(t_m))\|\\\nonumber
	&\quad+C\biggl\|\sum_{n=0}^m\int_{t_n}^{t_{n+1}}(\vB(\vu(s)) - \vB(\tilde{\vu}^n))\, dW(s)\biggr\|.
\end{align}
Next, squaring both sides of \eqref{eq3.45} followed by applying operators $k\sum_{m=0}^{\ell}$ and $\mE[\cdot]$, we obtain
\begin{align}\label{eq3.46}
	\beta^2\mE\biggl[k\sum_{m=0}^{\ell}\|\E_p^m\|^2\biggr] 
	&\leq C\mE\biggl[k\sum_{m=0}^{\ell}\|\ve_{\tilde{\vu}}^{m+1}\|^2\biggr] + Ck\sum_{m=0}^{\ell}\mE\biggl[\biggl\|k\sum_{n=0}^m\nab\ve_{\tilde{\vu}}^{n+1}\biggr\|^2\biggr]\\\nonumber
	&\,\,+Ck\sum_{m=0}^{\ell}\mE\biggl[\biggl\|\sum_{n=0}^m\int_{t_n}^{t_{n+1}}\nab(\vu(t_{n+1})-\vu(s))\, ds\biggr\|^2\biggr]\\\nonumber
	&\,\,+ Ck\sum_{m=0}^{\ell}\mE\bigl[\|\nab(P(t_{m+1}) - P(t_m))\|^2\bigr]\\\nonumber
	&\,\,+Ck\sum_{m=0}^{\ell}\mE\biggl[\biggl\|\sum_{n=0}^m\int_{t_n}^{t_{n+1}}(\vB(\vu(s)) - \vB(\tilde{\vu}^n))\, dW(s)\biggr\|^2\biggr]\\   \nonumber
	&\,\,:= {\tt I + II + III + IV + V}.
\end{align}	

We now bound each term above as follows. Using Theorem \ref{thm3.3} we get
\begin{align}\label{eq3.47}
	{\tt I + II} &= C\mE\biggl[k\sum_{m=0}^{\ell}\|\ve_{\tilde{\vu}}^{m+1}\|^2\biggr] + Ck\sum_{m=0}^{\ell}\mE\biggl[\biggl\|k\sum_{n=0}^m\nab\ve_{\tilde{\vu}}^{n+1}\biggr\|^2\biggr]\\
	\nonumber
	&\leq Ck^{\frac12}.
\end{align}
By the H\"older continuities given in \eqref{eq2.20b} and \eqref{holder_pressure}, we have 
\begin{align}\label{eq3.48}
	{\tt III + IV} &= Ck\sum_{m=0}^{\ell}\mE\biggl[\biggl\|\sum_{n=0}^m\int_{t_n}^{t_{n+1}}\nab(\vu(t_{n+1})-\vu(s))\, ds\biggr\|^2\biggr]\\\nonumber
	&\qquad+ Ck\sum_{m=0}^{\ell}\mE\bigl[\|\nab(P(t_{m+1}) - P(t_m))\|^2\bigr]\\\nonumber
	&\leq Ck\sum_{m=0}^{\ell}\mE\biggl[\sum_{n=0}^m\int_{t_n}^{t_{n+1}}\|\nab(\vu(t_{n+1})-\vu(s))\|^2\, ds\biggr]\\\nonumber
	&\quad+ Ck\sum_{m=0}^{\ell}\mE\bigl[\|\nab(P(t_{m+1}) - P(t_m))\|^2\bigr]\\\nonumber
	&\leq Ck.
\end{align}
By using the It\^o isometry and Theorem \ref{thm3.3} and \eqref{eq2.20a}, we conclude that
\begin{align}\label{eq3.49}
	{\tt V} &= Ck\sum_{m=0}^{\ell}\mE\biggl[\biggl\|\sum_{n=0}^m\int_{t_n}^{t_{n+1}}(\vB(\vu(s)) - \vB(\tilde{\vu}^n))\, dW(s)\biggr\|^2\biggr]\\ \nonumber
	&=Ck\sum_{m=0}^{\ell}\mE\biggl[\sum_{n=0}^m\int_{t_n}^{t_{n+1}}\|\vB(\vu(s)) - \vB(\tilde{\vu}^n)\|^2\, ds\biggr]\\ \nonumber
	&\leq Ck\sum_{m=0}^{\ell}\mE\biggl[\sum_{n=0}^{m}\int_{t_n}^{t_{n+1}}\|\vu(s) - \tilde{\vu}^n\|^2\, ds\biggr]\\ \nonumber
	&\leq Ck\sum_{m=0}^{\ell}\mE\biggl[\sum_{n=0}^{m}\int_{t_n}^{t_{n+1}}\|\vu(s) - \vu(t_n)\|^2\, ds\biggr] + Ck\sum_{m=0}^{\ell}\mE\biggl[k\sum_{n=0}^{m}\|\ve_{\tilde{\vu}}^n\|^2\biggr]\\
	\nonumber
	&\leq Ck^{\frac12}.
\end{align}
Finally, substituting \eqref{eq3.47}--\eqref{eq3.49} into \eqref{eq3.46} yields 
\begin{align}
	\beta^2\mE\biggl[k\sum_{m=0}^{\ell}\|\E_p^m\|^2 \biggr] \leq Ck^{\frac12} \qquad 
	\mbox{for } 1\leq \ell \leq M.
\end{align}
The proof is complete.
\end{proof}
	
\begin{remark}
		It is interesting to point out that the above proof uses the technique from 
		the (non-splitting) mixed method error analysis although Chorin scheme is a splitting scheme. 
\end{remark}
	
	We conclude this subsection by stating two stability estimates for $(\tilde{\vu}^m,p^m)$ in high norms as immediate corollaries of the above error estimates,  they will be used in the next section in deriving error estimates for a fully discrete finite element Chorin method. We note that these stability estimates improve those given in Lemma \ref{lemma3.1}
	and may not be obtained directly without using the above error estimates.
	
	\begin{corollary}\label{cor3.5} 
		Under the assumptions of Theorem \ref{thm3.3}, there exists a positive constant $C$ which 
		depends on $D_T, \vu_0$ and $\vf$ such that 
		\begin{align}\label{eq3.56}
			& \mE \bigg[k\sum_{m=0}^M \biggl\|k\sum_{n=0}^m\nab p^n\biggr\|^2\bigg] \leq C,\\ \label{eq3.55}
			& \max_{0\leq m \leq M} \mE \biggl[\biggl\|k\sum_{n=0}^m\nab \tilde{\vu}^n\biggr\|^2\biggr] + \mE \bigg[k\sum_{m=0}^M \biggl\|k\sum_{n=0}^m\Delta \tilde{\vu}^n\biggr\|^2\bigg]\leq C.
		\end{align}	
	\end{corollary}
	
	\begin{proof}
		\eqref{eq3.56} follows immediately from estimates \eqref{eq210} and \eqref{eq3.41} and  $\eqref{eq3.55}_1$ follows straightforwardly from the discrete Jensen inequality and $\eqref{eq3.5}_3$. It remains to prove $\eqref{eq3.55}_2$. To that end, testing \eqref{eq3.20} by $-k\sum_{n=0}^{m} \Delta\tilde{\vu}^{n+1}$, we obtain
		\begin{align}\label{eq3.53}
		 \biggl\|k\sum_{n=0}^{m}\Delta\tilde{\vu}^{n+1}\biggr\|^2 &= \Bigl(\tilde{\vu}^{m+1}-\vu_0, k\sum_{n=0}^{m} \Delta\tilde{\vu}^{n+1}\Bigr) \\\nonumber
		 &\quad+ \biggl(k\sum_{n=0}^m\nab p^n, k\sum_{n=0}^m\Delta\tilde{\vu}^{n+1}\biggr) \\\nonumber
		 &\quad- \biggl(\sum_{n=0}^m\int_{t_n}^{t_{n+1}}\vB(\tilde{\vu}^n)\, dW(s), k\sum_{n=0}^m\Delta\tilde{\vu}^{n+1}\biggr)\\\nonumber
		 &\leq \|\tilde{\vu}^{m+1}-\vu_0\|^2 + \frac{1}{4}\biggl\|k\sum_{n=0}^{m}\Delta\tilde{\vu}^{n+1}\biggr\|^2 \\\nonumber
		 &\quad+ \biggl\|k\sum_{n=0}^m\nab p^n\biggr\|^2 + \frac14\biggl\|k\sum_{n=0}^{m}\Delta\tilde{\vu}^{n+1}\biggr\|^2\\\nonumber
		 &\quad+ \biggl\|\sum_{n=0}^m\int_{t_n}^{t_{n+1}}\vB(\tilde{\vu}^n)\, dW(s)\biggr\|^2 + \frac14\biggl\|k\sum_{n=0}^{m}\Delta\tilde{\vu}^{n+1}\biggr\|^2.  
		\end{align}
		After an rearrangement, applying operators $k\sum_{m=0}^M$ and $\mE[\cdot]$ to \eqref{eq3.53} we obtain
		\begin{align}\label{eq3.54}
		\frac14\mE\biggl[k\sum_{m=0}^M\biggl\|k\sum_{n=0}^{m}\Delta\tilde{\vu}^{n}\biggr\|^2\biggr]
		&\leq \mE\biggl[ k\sum_{m=0}^M \biggl(\|\tilde{\vu}^m - \vu^0\|^2   +   \biggl\|k\sum_{n=0}^m\nab p^n\biggr\|^2 \biggr) \biggr] \\\nonumber
			&\quad+ k\sum_{m=0}^M\mE\biggl[\biggl\|\sum_{n=0}^{m-1}\int_{t_n}^{t_{n+1}}\vB(\tilde{\vu}^n)\, dW(s)\biggr\|^2\biggr].
		\end{align}
The first two terms can be bounded by using \eqref{eq3.5} and \eqref{eq3.56}. The third term can be  controlled by using the It\^o isometry and \eqref{eq3.5} as follows:
	\begin{align*}
		k\sum_{m=0}^M\mE\biggl[\biggl\|\sum_{n=0}^{m-1}\int_{t_n}^{t_{n+1}}\vB(\tilde{\vu}^n)\, dW(s)\biggr\|^2\biggr] &= k\sum_{m=0}^M\mE\biggl[\sum_{n=0}^{m-1}\int_{t_n}^{t_{n+1}}\|\vB(\tilde{\vu}^n)\|^2\, ds\biggr]\\\nonumber
		&\leq C k\sum_{m=0}^M\mE\biggl[k\sum_{n=0}^{m-1}\|\tilde{\vu}^n\|^2\, ds\biggr] \leq C.
	\end{align*}
The proof is complete.
\end{proof}

%%%%%%%%%%%%%%%
	\subsection{A modified Chorin projection scheme}\label{sub3.2} 
	In this subsection, we consider a modification of Algorithm 1 which was 
	already pointed out in \cite{CHP12} but did not analyze there.  
	The modification is to perform 
	a Helmholtz decomposition of $\vB(\tilde{\vu}^m)$ at each time step which allows us  to eliminate the curl-free part in Step 1 of Algorithm 1, this then results in a divergent-free 
	Helmholtz projected noise. The goal of this subsection is to present a 
	complete convergence analysis for the modified Chorin scheme which includes deriving stronger error estimates for both velocity and pressure approximations than those for the standard Chorin scheme. We note that this Helmholtz decomposition enhancing technique was
	also used in \cite{Feng1} to improve the standard mixed finite element method for \eqref{eq1.1}.  
	
	\subsubsection{\bf Formulation of the modified Chorin scheme}
	For ease of the presentation, we assume $W(t)$ is
	a real-valued Wiener process and independent of the spatial variable. The case of more general $W(t)$ can be dealt with  similarly. 
	The modified Chorin scheme is given as follows. 
	
	\medskip
	\noindent
	{\bf Algorithm 2} 
	\smallskip
	
	Set $\tilde{\vu}^0 = \vu^0 = \vu_0$. For $m=0,1,\cdots, M-1$, do 
	the following five steps.
	
	\smallskip
	{\em Step 1:} Given $\tilde{\vu}^m \in L^2(\Ome, \vH^1_{per}(D))$, find $\xi^m \in L^2(\Ome,H^1_{per}(D)/\mathbb{R})$ such that $\mP$-a.s.
	\begin{align}\label{eq_3.1}
		%\big(\nab\xi^m,\nab\phi\big) = \big(\vB(\tilde{\vu}^m),\nab\phi\big) \qquad \forall \phi \in H^1_{per}(D).
		\Delta \xi^m = \mbox{div } \vB(\tilde{\vu}^m) \qquad \mbox{in } D.
	\end{align}
	
	{\em Step 2:} Set $\pmb{\eta}^m_{\tilde{\vu}} = \vB(\tilde{\vu}^m) - \nab\xi^m$. Given $\vu^m \in L^2(\Ome,\bH)$ and $\tilde{\vu}^m \in L^2(\Ome,\vH_{per}^1(D))$, find $\tilde{\vu}^{m+1} \in L^2(\Ome,\vH_{per}^1(D))$ such that $\mP$-a.s.
	\begin{align}\label{eq_3.2}
		\tilde{\vu}^{m+1} - k \Delta\tilde{\vu}^{m+1} & = \vu^m  + k\vf^{m+1} + \pmb{\eta}^m_{\tilde{\vu}}\Delta W_{m+1}\qquad\mbox{in } D.
	\end{align}

	\smallskip
	{\em Step 3:} Find $r^{m+1} \in L^2(\Ome,H^1_{per}(D)/\mathbb{R})$ such that $\mP$-a.s.
	\begin{subequations}\label{eq_3.3}
		\begin{alignat}{2}\label{eq_3.3a}		
			-\Delta r^{m+1} &= -\frac{1}{k}\div \tilde{\vu}^{m+1} &&\qquad \mbox{in } D.
			%	\p_{\mathbf{n}} r^{m+1} &= 0 &&\qquad \mbox{on } \p D. \label{eq_3.3b}
		\end{alignat}
	\end{subequations}
	
	\smallskip
	{\em Step 4:} Define $\vu^{m+1} \in L^2(\Ome, \bH)$ as
	\begin{align} \label{eq_step3a}
		\vu^{m+1} &= \tilde{\vu}^{m+1} - k \nab r^{m+1} \qquad \mbox{in } D.
	\end{align}
	
	\smallskip
	{\em Step 5:} Define the pressure approximation $p^{m+1}$ as
	\begin{align}\label{eq_3.5}
		p^{m+1} = r^{m+1} + \frac{1}{k}\xi^{m}\Delta W_{m+1} \qquad \mbox{in } D. 
	\end{align}

	\begin{remark}
		It follows from \eqref{eq2.6b} and \eqref{eq_3.1} that the Helmholtz projection $\pmb{\eta}^m_{\tilde{\vu}}$ can be bounded in terms of $\tilde{\vu}^m$ as follows:
		\begin{align}\label{eq_3.6}
			\|\pmb{\eta}^m_{\tilde{\vu}}\|_{L^2} &\leq \|\vB(\tilde{\vu}^m)\|_{L^2} + \|\nab\xi^m_{\tilde{\vu}}\|_{L^2} \leq 2 \|\vB(\tilde{\vu}^m)\|_{L^2} \leq C\|\tilde{\vu}^m\|_{L^2}
		\end{align}
	\end{remark}

%%%%%%%%%%%%%%%%
	\subsubsection{\bf Stability estimates for the modified Chorin scheme} 
	In this subsection we first state some stability estimates for Algorithm 2.  
	We then recall the Euler-Maruyama 
	scheme for \eqref{eq1.1} and its stability and error estimates from \cite{Feng1}, 
	which will be utilized as a tool in the stability and error analysis of the 
	modified Chorin scheme in the next subsection. 
	
	\begin{lemma}\label{lemma3.2} 
		The discrete processes $\{(\tilde{\vu}^m,r^m)\}_{m=0}^M$ 
		% \in L^2(\Ome,\mV_{per})\times L^2(\Ome,H^1_{per}(D)/\mathbb{R});0\leq m\leq M\}$ 
		defined by Algorithm 2 satisfy 
		\begin{subequations}
			\begin{align}
				\max_{0\leq m \leq M}\mE [\|\tilde{\vu}^m\|^2] + \mE \bigg[\sum_{m=1}^{M} \|\tilde{\vu}^m - \tilde{\vu}^{m-1}\|^2\bigg] + \mE \bigg[k\sum_{m=0}^{M} \|\nab \tilde{\vu}^m\|^2\bigg] &\leq C, \\
				\mE \bigg[k^2\sum_{m=0}^{M}\|\nab r^m\|^2\bigg] &\leq C,
			\end{align}
		\end{subequations}
		where $C$ is a positive constant which depends on $D_T, \vu_0$ and $\vf$.	
	\end{lemma}
	
	Since the proof of this lemma follows the same lines as those of Lemma \ref{lemma3.1}. 
	We omit the proof to save space.
	
	Next, we recall the Helmholtz enhanced Euler-Maruyama scheme for \eqref{eq1.1} which was 
	proposed and analyzed in \cite{Feng1}. This scheme will be used as an auxiliary scheme in our 
	error estimates for the velocity and pressure approximations of Algorithm 2 in the next subsection. The Euler-Maruyama scheme reads as
	\begin{subequations}\label{eq_318}
		\begin{alignat}{2}\label{eq_318a}
			({\vv}^{m+1} - \vv^m) - k\Delta \vv^{m+1} + k\nab q^{m+1} &= k\vf^{m+1} + \pmb{\eta}_{\vv}^m\Delta W_{m+1} &&\qquad \mbox{in}\, D,\\ 
			\div \vv^{m+1} &= 0 &&\qquad\mbox{in}\, D, \label{eq_318ab}
		\end{alignat}
	\end{subequations}
	where $\pmb{\eta}_{\vv}^m = \vB(\vv^m) - \nab\xi^m_{\vv}$ denotes the Helmholtz 
	projection of $\vB(\vv^m)$. 
	
	It was proved in \cite{Feng1} that the following stability and error estimates hold for 
	the solution of the above Euler-Maruyama scheme.  
	
	\begin{lemma}  \label{lem3.1}
		The discrete processes $\{(\vv^m,q^m)\}_{m=0}^M$
		%\in L^2(\Omega,\vH_{per}^1(D))\times L^2(\Omega,H_{per}^1(D)/\mathbb{R}); 0\leq m\leq M\}$
		defined by \eqref{eq_318} satisfy
		\begin{subequations}
			\begin{align}\label{eq3.3}
				\max_{0\leq m\leq M} \mE\bigl[\|\vv^m\|^2\bigr]  +\mE\Bigl[\sum^M_{m=1}\|\vv^m-\vv^{m-1}\|^2\Bigr] + \mE\Bigl[k\sum^M_{m=0}\|\nabla \vv^m\|^2\Bigr] 
				&\leq C,  \\
				\label{eq_3.9}
				\max_{0\leq m\leq M} \mE\Bigl[\|\nabla \vv^m\|^2\Bigr] +\mE\Bigl[\sum^M_{m=1} \Bigl(\|\nabla (\vv^m-\vv^{m-1})\|^2 + k\|{\bf A} \vv^m\|^2\Bigr) \Bigr] 
				&\leq C,  \\
				\mE\Bigl[k\sum_{m=0}^M \|\nabla q^m\|^2\Bigr]
				&\leq	C , 	\label{eq_3.10}
			\end{align}
		\end{subequations}
		%	provided that $\vv^m\in L^2(\Omega,\vH_{per}^1(D)\cap \vH^2(D))$ for $0\leq m\leq M$.
		where $C$ is a positive constant which depends on $D_T, \vu_0$ and $\vf$.
	\end{lemma}
	
	\begin{lemma} 
		There hold the the following error estimates for the discrete processes $\{(\vv^m,q^m)\}_{m=0}^M$:
		\begin{subequations}
			\begin{align}\label{eq_err3.14}
				\max_{0\leq m \leq M} \bigl(\mE\big[\|\vu(t_m)-\vv^m\|^2\big]\bigr)^{\frac12} \hskip 1.5in & 	\\
				+ \bigg(\mE\bigg[k\sum_{m=0}^M\|\nab(\vu(t_m) - \vv^m)\|^2\bigg]\bigg)^{\frac12} &\leq C\sqrt{k}, \nonumber \\
				\label{eq_err3.15}
				\bigg(\mE\bigg[\bigg\|R(t_\ell) - k\sum_{m=0}^\ell q^m\bigg\|^2\bigg]\bigg)^{\frac12} &\leq C\sqrt{k}
			\end{align}
		\end{subequations}
		for $0\leq \ell\leq M$. Where $C$ is a positive constant which depends on $D_T,\vu_0$ and $\vf$. 
	\end{lemma}
	
	\begin{remark}
		We note that the reason for the estimate \eqref{eq_3.10} to hold is 
		because the Helmholtz projected noise $\pmb{\eta}_{\vv}^m$ is divergent-free. 
		Otherwise, the following weaker estimate can only be proved (cf. \cite{Feng}):
		\begin{align*}
			\mE\Bigl[k\sum_{m=0}^M \|\nabla q^m\|^2\Bigr] \leq\frac{C}{k}. 
		\end{align*}
	\end{remark}

	\subsubsection{\bf Error estimates for the modified Chorin scheme} 
	The goal of this subsection is to derive error estimates for both the velocity and pressure approximations generated by Algorithm 2. The anticipated error estimates are  
	stronger than those for the standard Chorin scheme proved in the previous subsection. 
	We note that our error estimate for the velocity approximation recovers 
	the same estimate obtained in \cite[Theorem 3.1]{CHP12} although the analysis given here is a lot simpler. On the other hand, the error estimate for the pressure approximation is  
	apparently new. The main idea of the proofs of this subsection is to use the 
	Euler-Maruyama scheme analyzed in \cite{Feng1} as an auxiliary scheme which bridges the 
	exact solution of \eqref{eq1.1} and the discrete solution of Algorithm 2. 
	
	The follow theorem gives an error estimate in strong norms for the velocity 
	approximation. 
	
	\begin{theorem}\label{thm3.8} 
		Let $\{(\tilde{\vu}^m, p^m)\}_{m=0}^M$ be the solution of Algorithm 2 
		and $\{(\vu(t),$ $P(t));\, 0\leq t\leq T\}$ be the solution of \eqref{eq1.1}. 
		Then there holds the following estimate:
		\begin{align}\label{eq_3.19}
			\max_{0\leq m \leq M}\bigg(\mE\bigg[\|\vu(t_m) &- \tilde{\vu}^m\|^2\bigg]\bigg)^{\frac12}\\\nonumber 
			&+ \bigg(\mE\bigg[k\sum_{m=0}^{M}\|\nab(\vu(t_m)-\tilde{\vu}^m)\|^2\bigg]\bigg)^{\frac12} 
			\leq C\sqrt{k},
		\end{align}
	\end{theorem} 
	where $C$ is a positive constant which depends on $D_T, \vu_0$ and $\vf$.
	
	\begin{proof} Let $\ve^{m}_{\tilde{\vu}} = \vv^{m}-\tilde{\vu}^{m}$, $e_r^{m} = q^{m} - r^{m}$. Then, $\ve_{\tilde{\vu}}^{m} \in L^2(\Ome,\vH_{per}^1(D))$ and $e_r^{m} \in L^2(\Ome,H_{per}^1(D)/\mathbb{R})$. Subtracting \eqref{eq3.4} from \eqref{eq_318}  yields
		\begin{subequations} 
			\begin{alignat}{2}
				\label{eq_3.23}(\ve^{m+1}_{\tilde{\vu}} - \ve^m_{\tilde{\vu}}) - k\Delta \ve^{m+1}_{\tilde{\vu}} + k\nab e^m_r 
				&= - k(\nab(q^{m+1}-q^m)) && \\
				&\qquad + (\pmb{\eta}^m_{\vv}-\pmb{\eta}^m_{\tilde{\vu}})\Delta W_{m+1} &&\qquad\mbox{on} \, D, \nonumber\\
				\label{eq_3.24}\div \ve^{m+1}_{\tilde{\vu}} - k\Delta e_r^{m+1} &= -k\Delta q^{m+1} &&\qquad\mbox{on}\, D,\\
				\label{eq_3.25}\p_{\bf n} e_r^{m+1} &= \p_{\bf n} q^{m+1} &&\qquad\mbox{on}\, \p D.
			\end{alignat}
		\end{subequations}
		%and $\ve^0_{\tilde{\vu}} = 0$ in $D$. 
		Testing \eqref{eq_3.23} with $\ve^{m+1}_{\tilde{\vu}}$ and integrating by parts, 
		we obtain 
		\begin{align}
			\label{eq_3.69}(\ve^{m+1}_{\tilde{\vu}} - \ve^m_{\tilde{\vu}}, &\,\ve^{m+1}_{\tilde{\vu}}) + k\|\nab \ve^{m+1}_{\tilde{\vu}}\|^2 + k(\nab e_r^m,\ve^{m+1}_{\tilde{\vu}}) \\\nonumber
			&= - k(\nab(q^{m+1}-q^m),\ve^{m+1}_{\tilde{\vu}}) + ((\pmb{\eta}_{\vv}^m-\pmb{\eta}^m_{\tilde{\vu}})\Delta W_{m+1},\ve^{m+1}_{\tilde{\vu}}).
		\end{align}
		Using the algebraic identity $2ab = (a^2 - b^2) + (a-b)^2$ in the first term gives 
		\begin{align}\label{eq_3.27}
			\frac12 \bigl( \|\ve_{\tilde{\vu}}^{m+1}\|_{L^2}^2 &- \|\ve_{\tilde{\vu}}^m\|_{L^2}^2\bigr)  +  \frac12  \|\ve_{\tilde{\vu}}^{m+1}-\ve_{\tilde{\vu}}^m\|_{L^2}^2 \\ \nonumber
			&\qquad + k\|\nab \ve_{\tilde{\vu}}^{m+1}\|_{L^2}^2 +k(\nab e_r^m,\ve_{\tilde{\vu}}^{m+1}) \\\nonumber
			& = - k(\nab(q^{m+1}-q^m),\ve_{\tilde{\vu}}^{m+1}) + ((\pmb{\eta}_{\vv}^m- \pmb{\eta}^m_{\tilde{\vu}})\Delta W_{m+1},\ve_{\tilde{\vu}}^{m+1}).
		\end{align}
		
		Next, we derive a reformulation for each of the last term on the left-hand side and  the first term on the right of \eqref{eq_3.27} with a help of \eqref{eq_3.24}. Testing  \eqref{eq_3.24} by any $q \in L^2(\Ome,H_{per}^1(D)/\mathbb{R})$ and using  \eqref{eq_3.25}, we obtain  
		\begin{subequations}
			\begin{align}\label{eq_3.28}
				\bigl(\ve_{\tilde{\vu}}^{m+1},\nab q\bigr) - k\bigl(\nab e_r^{m+1},\nab q\bigr) &= -k\bigl(\nab q^{m+1},\nab q\bigr)\\
				\bigl(\ve_{\tilde{\vu}}^{m+1}-\ve_{\tilde{\vu}}^m,\nab q\bigr) - k\bigl(\nab (e_r^{m+1}-e_r^m),\nab q\bigr) &= -k\bigl(\nab (q^{m+1}-q^m),\nab q\bigr).\label{eq_3.29}
			\end{align}
		\end{subequations}
		
		Setting $q = e_r^{m}$ in \eqref{eq_3.28}, we obtain
		\begin{align}\label{eq_3.30}
			&\bigl(\ve_{\tilde{\vu}}^{m+1},\nab e_r^m\bigr) = k\bigl(\nab e_r^{m+1},\nab e_r^m\bigr) - k\bigl(\nab q^{m+1},\nab e_r^m\bigr),\\\nonumber
			&\hskip 0.6in
			= k\|\nab e_r^{m+1}\|^2 - k\bigl(\nab(e_r^{m+1}-e_r^m),\nab e_r^{m+1}\bigr) - k\bigl(\nab q^{m+1},\nab e_r^m\bigr).
		\end{align}
		and choosing $q = e_r^{m+1}$ in \eqref{eq_3.29}, we have 
		\begin{align}\label{eq_3.31} 
			k\bigl(\nab(e_r^{m+1} - e_r^m),\nab e_r^{m+1}\bigr) &= \bigl(\ve_{\tilde{\vu}}^{m+1} - \ve_{\tilde{\vu}}^m,\nab e_r^{m+1}\bigr) \\\nonumber &\qquad+ k \bigl(\nab(q^{m+1}-q^m),\nab e_r^{m+1}\bigr).
		\end{align}
		Substituting \eqref{eq_3.31} to \eqref{eq_3.30} yields
		\begin{align}\label{eq_3.32}
			\bigl(\ve_{\tilde{\vu}}^{m+1},\nab e_r^m\bigr)	& = k\|\nab e_r^{m+1}\|^2 - \bigl(\ve_{\tilde{\vu}}^{m+1}-\ve_{\tilde{\vu}}^m,\nab e_r^{m+1}\bigr) \\\nonumber 
			&\quad- k \bigl(\nab(q^{m+1}-q^m),\nab e_r^{m+1}\bigr)- k\bigl(\nab q^{m+1},\nab e_r^m\bigr).
		\end{align}
		Alternatively, setting $q = q^{m+1} - q^m$ in \eqref{eq_3.28}, we obtain
		\begin{align}\label{eq_3.33}
			\bigl(\ve_{\tilde{\vu}}^{m+1},\nab(q^{m+1}-q^m)\bigr) &= k\bigl(\nab e_r^{m+1},\nab(q^{m+1}-q^m)\bigr) \\\nonumber
			&\qquad- k\bigl(\nab q^{m+1},\nab(q^{m+1}-q^m)\bigr).
		\end{align}
		
		Substituting \eqref{eq_3.32} and \eqref{eq_3.33} into \eqref{eq_3.27} and  rearranging  terms, we get 
		\begin{align}
			\label{eq_3.34}&\frac12 \bigl(\|\ve_{\tilde{\vu}}^{m+1}\|_{L^2}^2 - \|\ve_{\tilde{\vu}}^m\|^2\bigr) + \frac12 \|\ve_{\tilde{\vu}}^{m+1} - \ve_{\tilde{\vu}}^m\|^2 + k\|\nab \ve_{\tilde{\vu}}^{m+1}\|^2 + k^2\|\nab e_r^{m+1}\|^2 \\\nonumber
			&\qquad= k\bigl(\ve_{\tilde{\vu}}^{m+1}-\ve_{\tilde{\vu}}^m,\nab e_r^{m+1}\bigr)  + k^2\bigl(\nab q^{m+1},\nab e_r^m\bigr)\\\nonumber
			& \qquad\qquad+ k^2\bigl(\nab q^{m+1},\nab(q^{m+1}-q^m)\bigr) + \bigl((\pmb{\eta}_{\vv}^m - \pmb{\eta}_{\tilde{\vu}}^m)\Delta W_{m+1},\ve_{\tilde{\vu}}^{m+1}\bigr).
		\end{align}	
		
		Finally, we bound each term on the right-hand side of \eqref{eq_3.34}. By Young's inequality, for $\delta_1,\delta_2,\delta_3 > 0$ we obtain 
		\begin{subequations}\label{eq_3.78x}
			\begin{align}\label{eq_3.78}	
				k\bigl(\ve_{\tilde{\vu}}^{m+1}-\ve_{\tilde{\vu}}^m,\nab e_r^{m+1}\bigr) &\leq \frac{1}{4\delta_1}\|\ve_{\tilde{\vu}}^{m+1}-\ve_{\tilde{\vu}}^m\|^2 + \delta_1 k^2\|\nab e_r^{m+1}\|^2.\\
				k^2\bigl(\nab q^{m+1},\nab e_r^m\bigr) &\leq \frac{1}{4\delta_2}k^2\|\nab q^{m+1}\|^2 + \delta_2k^2\|\nab e_r^m\|^2.\\
				k^2\bigl(\nab q^{m+1},\nab(q^{m+1}-q^m)\bigr) &\leq k^2\|\nab q^{m+1}\|^2 + \frac{1}{4}k^2\|\nab(q^{m+1}-q^m)\|^2.
			\end{align}	
		\end{subequations}
		Rewriting  
		\begin{align}\label{eq_333}
			\bigl((\pmb{\eta}_{\vv}^m - \pmb{\eta}_{\tilde{\vu}}^m)\Delta W_{m+1},\ve_{\tilde{\vu}}^{m+1}\bigr) &= \bigl((\pmb{\eta}_{\vv}^m - \pmb{\eta}_{\tilde{\vu}}^m)\Delta_{m+1} W,\ve_{\tilde{\vu}}^{m+1} - \ve_{\tilde{\vu}}^m\bigr) \\\nonumber
			&\qquad\qquad+ \bigl((\pmb{\eta}_{\vv}^m - \pmb{\eta}_{\tilde{\vu}}^m)\Delta W_{m+1},\ve_{\tilde{\vu}}^{m}\bigr).
		\end{align}
		Since the expectation of the second term on the right-hand side of \eqref{eq_333} vanishes due to the martingale property of It\^o's integral, we only need to estimate the first term. Again, rewriting 
		\begin{align}
			&\bigl( (\pmb{\eta}_{\vv}^m - \pmb{\eta}_{\tilde{\vu}}^m)\Delta_{m+1} W,\ve_{\tilde{\vu}}^{m+1} - \ve_{\tilde{\vu}}^m\bigr) 
			= - \bigl(\nab(\xi_{\vv}^m - \xi_{\tilde{\vu}}^m)\Delta_{m+1} W,\ve_{\tilde{\vu}}^{m+1} - \ve_{\tilde{\vu}}^m\bigr) \\
			&\hskip 1.8 in 
			+\bigl((\vB(\vv^m) - \vB(\tilde{\vu}^m))\Delta W_{m+1},\ve_{\tilde{\vu}}^{m+1} - \ve_{\tilde{\vu}}^m\bigr) . \nonumber
		\end{align}
		Then we have
		\begin{subequations}\label{eq_3.38x}
			\begin{align}
				\mE\big[\bigl(\nab(\xi_{\vv}^m &- \xi_{\tilde{\vu}}^m)\Delta W_{m+1},\ve_{\tilde{\vu}}^{m+1} - \ve_{\tilde{\vu}}^m\bigr)\big] \\\nonumber
				&\leq 2k\mE\big[\|\nab(\xi_{\vv}^m - \xi_{\tilde{\vu}}^m)\|^2\big] + \frac18\mE\big[\|\ve_{\tilde{\vu}}^{m+1} - \ve_{\tilde{\vu}}^m\|^2\big]\\\nonumber
				&\leq 2k\mE\big[\|(\vB(\vv^m) - \vB(\tilde{\vu}^m))\|^2\big] + \frac18\mE\big[\|\ve_{\tilde{\vu}}^{m+1} - \ve_{\tilde{\vu}}^m\|^2\big]\\\nonumber
				&\leq Ck\mE\big[\|\ve_{\tilde{\vu}}^m\|^2\big] + \frac18\mE\big[\|\ve_{\tilde{\vu}}^{m+1} - \ve_{\tilde{\vu}}^m\|^2\big],\\
				\label{eq_3.38}	\mE\big[\bigl((\vB(\vv^m)&-\vB(\tilde{\vu}^m))\Delta W_{m+1},\ve_{\tilde{\vu}}^{m+1}\bigr)\big] \\\nonumber
				&\leq \frac{1}{4\delta_3}k\mE\big[\|\bigl(\vB(\vv^m) - \vB(\tilde{\vu}^m)\bigr)\|_{L^2}^2\big] + \delta_3\mE\big[\|\ve_{\tilde{\vu}}^{m+1}-\ve_{\tilde{\vu}}^m\|^2\big]\\\nonumber 
				&\leq \frac{Ck}{4\delta_3}\mE\big[\|\ve_{\tilde{\vu}}^{m}\|^2\big]+ \delta_3\mE\big[\|\ve_{\tilde{\vu}}^{m+1}-\ve_{\tilde{\vu}}^m\|^2\big].
			\end{align}
		\end{subequations}
		Now, substituting \eqref{eq_3.78x} and \eqref{eq_3.38x} into \eqref{eq_3.34} and  taking  expectation on both sides, we obtain 
		\begin{align}\label{eq_3.39}
			&\frac12\mE \bigl[\|\ve_{\tilde{\vu}}^{m+1}\|^2 - \|\ve_{\tilde{\vu}}^m\|^2\bigr] + \bigg(\frac38 - \frac{1}{4\delta_1} - \delta_3\bigg)\mE\bigl[\|\ve_{\tilde{\vu}}^{m+1}-\ve_{\tilde{\vu}}^m\|^2 \\\nonumber
			&\qquad+ k\mE\bigl[\|\nab \ve_{\tilde{\vu}}^{m+1}\|^2\bigr]+ k^2\bigl(1-\delta_1-\delta_2\bigr)\mE\bigl[\|\nab e_r^{m+1}\|^2\bigr] \\\nonumber
			&\quad = \frac12 k^2\mE\bigl[\|\nab q^{m+1}\|^2\bigr] + \frac12 k^2 \mE\bigl[\|\nab(q^{m+1}-q^m)\|^2\bigr] 
			+ \frac{1}{4\delta_3}Ck\mE\bigl[\|\ve_{\tilde{\vu}}^m\|^2\bigr].
		\end{align}	
		Choosing $\delta_1 = \frac{17}{20}, \delta_2 = \frac{1}{10}, \delta_3 = \frac{1}{16}$,  taking the summation for $0\leq m\leq \ell\leq M-1$, and using \eqref{eq_3.10} and the discrete Gronwall's inequality, we get
		\begin{align}\label{eq_3.40}	
			\frac12\mE\bigl[\|\ve_{\tilde{\vu}}^{\ell+1}\|^2\bigr] &+ \frac{5}{272}\mE\bigg[\sum_{m=0}^{\ell}\|\ve_{\tilde{\vu}}^{m+1}-\ve_{\tilde{\vu}}^m\|^2\bigg] + \mE\bigg[k\sum_{m=0}^{\ell}\|\nab \ve_{\tilde{\vu}}^{m+1}\|^2\bigg]\\\nonumber
			&+ \frac{1}{20}\mE\bigg[k^2\sum_{m=0}^{\ell}\|\nab e_r^{m+1}\|^2\bigg] \leq \exp\bigl( 4C T\bigr)C\,k.
		\end{align}
		which and \eqref{eq_err3.14} infer the desired estimate. The proof is complete. 
	\end{proof}
	
	An immediate corollary of the above error estimate is the following stronger 
	stability estimates for $\{(\tilde{\vu}^m,r^m)\}$, which may not be obtainable directly 
	and will play an important 
	role in the error analysis of fully discrete counterpart of Algorithm 2
	in the next section. 
	
	\begin{corollary} 
		There exists $C>0$ which depends on $D_T,\vu_0$ and $\vf$ such that
		\begin{subequations}\label{eq_3.87x}
			\begin{align}\label{eq_3.87}
				\max_{0\leq m \leq M} \mE[\|\nab\tilde{\vu}^m\|^2] + \mE\bigg[k\sum_{m=0}^{M}\|\Delta \tilde{\vu}^m\|^2\bigg] &\leq C,\\
				\mE\bigg[k\sum_{m=0}^{M}\|\nab r^m\|^2\bigg] &\leq C. \label{eq_3.87b}
			\end{align}	
		\end{subequations}
	\end{corollary}
	
	\begin{proof}
		The estimate \eqref{eq_3.87b} follows immediately from \eqref{eq_3.40} and the triangular inequality. To prove \eqref{eq_3.87}, testing \eqref{eq3.4a} by $-\Delta\tilde{\vu}^{m+1}$, we obtain
		\begin{align}\label{eq3.42}
			\frac12\mE[\|\nab\tilde{\vu}^{m+1}\|^2 - \|\nab\tilde{\vu}^m\|^2] &+ \frac14\mE[\|\nab(\tilde{\vu}^{m+1}-\tilde{\vu}^m)\|^2] + \frac14 k\mE[\|\Delta\tilde{\vu}^{m+1}\|^2] \\\nonumber
			&\leq k\mE[\|\nab r^{m}\|^2] + k\mE\big[\|\vf^{m+1}\|^2\big] + k\mE[\|\nab\pmb{\eta}_{\tilde{\vu}}^m\|^2].
		\end{align} 
		Here we have used the periodic boundary condition for $\pmb{\eta}_{\tilde{\vu}}^m$ to kill the boundary term which arises from integration by parts. 
		
		Finally, applying the operator $\sum_{m=0}^{\ell}$ to \eqref{eq3.42} and  
		using the discrete Gronwall inequality and \eqref{eq_3.87}$_3$, we obtain the desired estimate. The proof is complete.
	\end{proof}
	
	Similarly, the following estimate holds for $\{\vu^m\}$. 
	\begin{corollary}
		There exists $C>0$ which depends on $D_T,\mathbf{u}_0$ and $\mathbf{f}$
		such that  
		\begin{align}\label{eq_3.43}
			\max_{0\leq m \leq M} \big(\mE[\|\vu(t_m) - \vu^m \|^2]\big)^{\frac12} + \bigg(\mE\bigg[k\sum_{m=0}^M\| \vu(t_m) - \vu^m\|^2\bigg]\bigg)^{\frac12} \leq C\sqrt{k}.
		\end{align}
	\end{corollary}
	
	{The proof of \eqref{eq_3.43} readily follows from  \eqref{eq_step3a} and Theorem \eqref{thm3.8} as well as the estimate \eqref{eq_3.87b}.
	}
	
	Next, we derive error estimates for the pressure approximations $r^m$ 
	and $p^m$ generated by Algorithm 2. First, we state the following lemma.
	
	\begin{lemma}\label{lem3.4} 
		Let $\{r^m\}_{m=0}^M$ be generated by Algorithm 2. Then, there exists a constant $C>0$
		depending on $D_T, \vu_0, \vf$ and $\beta$ such that for $0\leq \ell\leq M$
		\begin{align}\label{eq_3.42}
			\bigg(\mE\bigg[\bigg\|k\sum_{m=1}^\ell (r^{m} -r^{m-1})\bigg\|^2\bigg]\bigg)^{\frac12} \leq C\sqrt{k}.
		\end{align}
		
	\end{lemma}
	\begin{proof} 
		The idea of the proof is to utilize the inf-sup condition \eqref{inf-sup}. 
		Testing \eqref{eq_step3a} by any $\vv \in L^2(\Ome;\vH^1_{per}(D))$, we obtain
		\begin{align*}
			k\bigl(r^{m+1},\div \vv\bigr) &= \bigl(\vu^{m+1} - \tilde{\vu}^{m+1},\vv\bigr),\\\nonumber
			k\bigl(r^{m},\div \vv\bigr) &= \bigl(\vu^{m} - \tilde{\vu}^{m},\vv\bigr).
		\end{align*}
		Then, subtracting the above equations yields 
		\begin{align}\label{eq_345}
			k\bigl(r^{m+1}-r^m,\div \vv\bigr) = \bigl((\vu^{m+1} -\vu^m) - (\tilde{\vu}^{m+1} - \tilde{\vu}^m),\vv\bigr).
		\end{align}
		Applying the summation operator $\sum_{m=0}^{\ell}$ for $0 \leq \ell\leq M-1$ to \eqref{eq_345}, we get
		\begin{align}\label{eq_346}
			\bigg(k\sum_{m=0}^{\ell}(r^{m+1}-r^m),\div \vv\bigg) &= \bigl((\vu^{\ell+1} -\vu^0) - (\tilde{\vu}^{\ell+1} - \tilde{\vu}^0),\vv\bigr) \\\nonumber
			&= \bigl(\ve_{\vu}^{\ell+1} - \ve_{\tilde{\vu}}^{\ell+1},\vv\bigr)\\\nonumber
			&\leq C\bigl(\|\ve_{\vu}^{\ell+1}\| + \|\ve_{\tilde{\vu}}^{\ell+1}\|\bigr)\|\nab\vv\|,
		\end{align}
		where $\ve_{\vu}^{m}$ and $\ve_{\tilde{\vu}}^m$ are the same as defined in the preceding subsection and we have used the fact that $\vu^0 - \tilde{\vu}^0 = 0$.
		
		Finally, by using the inf-sup condition \eqref{inf-sup} and then taking the expectation we get 
		\begin{align*}
			\beta^2\mE\bigg[\Bigl\|k\sum_{m=0}^{\ell}(r^{m+1}-r^m)\Bigr\|^2 \bigg]  
			&\leq C\Bigl( \mE\bigl[\|\ve_{\vu}^{\ell+1}\|^2\bigr] + \mE\bigl[\|\ve_{\tilde{\vu}}^{\ell+1}\|^2\bigr] \Bigr),
		\end{align*}
		which and the estimates for $\ve_{\vu}^{\ell+1}$ and $\ve_{\tilde{\vu}}^{\ell+1}$ infer the  desired estimate \eqref{eq_3.42}. The proof is complete. 	
	\end{proof}
	
	We then are ready to state the following error estimate result for $r^m$. 
	
	\begin{theorem}\label{thm3.6}
		Let $\{r^m\}_{m=0}^M$ be generated by Algorithm 2 and $R(t)$ be defined in \eqref{R_t}.
		Then there exists a constant $C>0$ depending on $D_T,\vu_0, \vf$ and $\beta$ such hat
		for $0\leq \ell\leq M$
		\begin{align}\label{eq_344}
			\bigg(\mE\bigg[ \bigg\|R(t_\ell) - k\sum_{m=0}^\ell r^m \bigg\|^2\bigg]\bigg)^{\frac12} \leq C\sqrt{k}.
		\end{align}
	\end{theorem}
	
	\begin{proof} 
		Subtracting \eqref{eq_318} from \eqref{eq3.4a} and then testing the resulting equation by  $\vv \in L^2(\Ome;\vH^1_{per}(D))$, we obtain 
		\begin{align}\label{eq_3.44}
			\bigl(\ve_{\tilde{\vu}}^{m+1} - \ve_{\tilde{\vu}}^m,\vv\bigr) + k\bigl(\nab\ve_{\tilde{\vu}}^{m+1},\nab\vv\bigr) &- k\bigl(q^{m+1} - r^m,\div\vv\bigr) \\\nonumber
			&= \bigl(\big(\pmb{\eta}_{\vv}^m-\pmb{\eta}_{\tilde{\vu}}^m\big)\Delta W_{m+1},\vv\bigr).
		\end{align}	
		Applying the summation operator $\sum_{m=0}^{\ell}$ to \eqref{eq_3.44} for 
		$0 \leq \ell\leq M-1$ yields  
		\begin{align}\label{eq_3.445}
			\bigg(k\sum_{m=0}^{\ell}(q^{m+1}-r^m),\div\vv \bigg) &= \bigl(\ve_{\tilde{\vu}}^{\ell+1} - \ve_{\tilde{\vu}}^0,\vv\bigr) 
			+ \bigg(k\sum_{m=0}^{\ell}\nab\ve_{\tilde{\vu}}^{m+1},\nab\vv\bigg) \\\nonumber
			&\qquad -\bigg(\sum_{m=0}^{\ell}\big(\pmb{\eta}_{\vv}^m-\pmb{\eta}_{\tilde{\vu}}^m\big)\Delta W_{m+1},\vv\bigg)\\\nonumber
			&= {\tt I} + {\tt II} + {\tt III}.
		\end{align}
		
		Now, we bound each term on the right-hand side of \eqref{eq_3.445} as follows.
		By Schwarz and Poincar\'e inequalities, we get 
		\begin{align*}
			|{\tt I}| + |{\tt II}| &\leq C\|\ve_{\tilde{\vu}}^{\ell+1}\|\|\nab \vv\| + k\sum_{m=0}^{\ell}\|\nab\ve_{\tilde{\vu}}^{m+1}\|\cdot\|\nab\vv\|,\\
			|{\tt III}| &\leq \big\|\sum_{m=0}^{\ell}(\pmb{\eta}_{\vv}^m - \pmb{\eta}_{\tilde{\vu}}^m)\Delta W_{m+1}\big\|\cdot \|\nab\vv\| \\\nonumber
			&\leq \big\|\sum_{m=0}^{\ell}(\vB(\vv^m) - \vB(\tilde{\vu}^m))\Delta W_{m+1}\|\cdot \|\nab\vv\| \\\nonumber
			&\qquad + \big\|\sum_{m=0}^{\ell}(\nab(\xi_{\vv}^m - \xi_{\tilde{\vu}}^m))\Delta W_{m+1}\big\|\cdot\|\nab\vv\| \\\nonumber
			& = \bigl({\tt III_1 + III_2}\bigr)\,\|\nab\vv\|.
		\end{align*}
		By  It\^o isometry, the assumptions on $\vB$ and Theorem \ref{thm3.8}, we obtain 
		\begin{align*}
			\mE[{\tt III_1^2}] &= k\sum_{m=0}^{\ell}\mE\big[\|\vB(\vv^m) - \vB(\tilde{\vu}^m)\|^2\big] \leq Ck\sum_{m=0}^{\ell}\mE\big[\|\ve_{\tilde{\vu}}^m\|^2\big] \leq Ck. \\
			\mE[{\tt III_2^2}] &= k\sum_{m=0}^{\ell}\mE\big[\|\nab(\xi_{\vv}^m - \xi_{\tilde{\vu}}^m)\|^2\big].
		\end{align*}
		Moreover, using the Poisson equation as defined in Step 1 of Algorithm 1, we have
		$$\|\nab(\xi_{\vv}^m - \xi_{\tilde{\vu}}^m)\|\leq\|\vB(\vv^m) - \vB(\tilde{\vu}^m)\|.$$ Thus,
		\begin{align*}%\label{eq_3.53}
			\mE[{\tt III}_2^2] &\leq k\sum_{m=0}^{\ell}\mE[\|\vB(\vv^m) - \vB(\tilde{\vu}^m)\|^2] \leq k\sum_{m=0}^{\ell}\mE[\|\ve_{\tilde{\vu}}^m\|^2] \leq Ck.
		\end{align*}
		
		Finally, it follows from the inf-sup condition \eqref{inf-sup}, Theorem \ref{thm3.8},  \eqref{eq_3.445}, and all above inequalities that
		\begin{align}
			\bigg(\mE\bigg[\bigg\|k\sum_{m=0}^{\ell}(q^{m+1}-r^m)\bigg\|^2\bigg]\bigg)^{\frac12} \leq C\sqrt{k}.
		\end{align}
		The proof is completed by applying the triangular inequality, Lemma \ref{lem3.4} and \eqref{eq_err3.15}.
	\end{proof}
	
	\begin{corollary}\label{cor3.7}
		Let $\{p^m\}_{m=0}^M$ be generated by Algorithm 2.  Then, there exists a constant $C>0$
		which depends on  $D_T, \vu_0, \vf$ and $\beta$ such that for $0\leq \ell\leq M$
		\begin{align}
			\bigg(\mE\bigg[\bigg\|P(t_\ell) - k\sum_{m=0}^{\ell} p^m\bigg\|^2\bigg]\bigg)^{\frac12} \leq C\sqrt{k}.
		\end{align}
		%	where $C \equiv C(D_T, \vu_0,\vf, \beta) > 0$.
	\end{corollary}
	
	\begin{proof} 
		The proof is similar to that of Theorem \ref{thm3.6}. Indeed, we have
		\begin{align}\label{eq_3.56}
			\bigg\|P(t_m) - k\sum_{m=0}^{\ell} p^m\bigg\| &\leq \bigg\|R(t_m) - k\sum_{m=0}^{\ell} r^m\bigg\| \\\nonumber
			&\qquad+ \bigg\|\sum_{m=0}^{\ell}\int_{t_m}^{t_{m+1}}\xi(s)\, dW(s) - \sum_{m=0}^{\ell}\xi^{m}\Delta W_{m+1}\bigg\| \\\nonumber
			&= {\tt I} + {\tt II}.
		\end{align}
		
		Clearly, the first term ``${\tt I}$" can be bounded by using Theorem \ref{thm3.6}. To bound the second term ``${\tt II}$", 
		we first have 
		\begin{align}\label{eq_3.57}
			{\tt II} &\leq \bigg\|\sum_{m=0}^{\ell}(\xi(t_m) - \xi^m)\Delta W_{m+1}\bigg\| + \bigg\|\sum_{m=0}^{\ell}\int_{t_m}^{t_{m+1}}(\xi(s) -\xi(t_m))\,dW(s)\bigg\| \\\nonumber
			&= {\tt II}_1 + {\tt II}_2.
		\end{align}
		Then by It\^o isometry we get
		\begin{align}\label{eq_3.58}
			\mE[{\tt II}_1]^2 &= \mE\bigg[k\sum^{\ell}_{m=0}\|\xi(t_m) -\xi^m\|^2\bigg]  \leq C\mE\bigg[k\sum^{\ell}_{m=0}\|\vB(\vu(t_m)) - \vB(\tilde{\vu}^m)\|^2\bigg]\\\nonumber
			& \leq C\mE\bigg[k\sum^{\ell}_{m=0}\|\ve_{\tilde{\vu}}^m\|^2\bigg] \leq Ck.
		\end{align}
		Similarly, by the H\"older continuity \eqref{eq2.8a} we have 
		\begin{align}\label{eq_3.59}
			\mE[{\tt II}_2]^2 &= \sum_{m=0}^{\ell}\mE\bigg[\int_{t_m}^{t_{m+1}}\|\xi(s)-\xi(t_m)\|^2\,ds\bigg] \\\nonumber
			&\leq \sum_{m=0}^{\ell}\mE\bigg[\int_{t_m}^{t_{m+1}}\|\vu(s)-\vu(t_m)\|^2\,ds\bigg] \leq Ck,
		\end{align}
		
		Finally, substituting \eqref{eq_3.57}--\eqref{eq_3.59} into \eqref{eq_3.56} and using Theorem \ref{thm3.6} give us the desired estimate. The proof is complete.	
	\end{proof}
	
	%%%%%%%
	\section{Fully discrete finite element methods} \label{sec4}
	%\subsection{Preliminaries}
	In this section, we formulate and analyze finite element spatial discretization 
	for Algorithm 1 and 2. To the end, let $\cT_h$ be a quasi-uniform triangulation of the polygonal ($d=2$) or polyhedral ($d=3$) bounded domain $D$. We introduce 
	the following two basic Lagrangian finite element spaces:
	\begin{align}
		V_h &= \{\phi \in C(\overline{D});\quad\phi_{|_{K}} \in P_\ell(K)\quad\forall K\in \cT_h\},\\
		X_h &= \{\phi \in C(\overline{D});\quad\phi_{|_{K}} \in P_\ell(K)\quad\forall K\in \cT_h\},
	\end{align}
	where $P_{\ell}(K)$ $(\ell \geq 1)$ denotes the set of polynomials of degree less than or equal to $\ell$ over the element $K \in \cT_h$. The finite element spaces to be used to
	formulate our finite element methods are defined as follows:
	\begin{align}
		\vH_h = [V_h\cap H^1_{per}(D)]^d,\quad L_h = V_h\cap L^2_{per}(D)/\mathbb{R}, \quad S_h = X_h \cap L^2_{per}(D)/\mathbb{R}.
	\end{align}
	In addition, we introduce spaces
	\begin{align}
		\mV_h = L^2(\Omega,\vH_h),\qquad \mW_h = L^2(\Omega,L_h).
	\end{align}
	
	Recall that the $L^2$-projection $\mathcal{P}^0_h:[L^2_{per}(D)]^d \rightarrow \vH_h$ is defined by   
	\begin{align}\label{eq4.4}
		(\phi-\cP_h^0\phi,\,\xi) = 0\qquad\forall \xi\in \vH_h
	\end{align}
	and the $H^1$-projection $\cP_h^1:H^1_{per}(D)/\mathbb{R} \rightarrow L_h$ is 
	defined by
	\begin{align}
		\bigl((\nab(\chi - \cP_h^1\chi),\nab\eta \bigr) = 0\qquad\forall \eta \in L_h.
	\end{align}
	It is well known \cite{BS2008} that $\cP_h^0$ and $\cP_h^1$ satisfy following   estimates:
	\begin{align}
		\label{eq4.6}\|\phi - \cP_h^0\phi\| + h\|\nab(\phi-\cP_h^0\phi)\| &\leq C\,h^2\|\phi\|_{H^2}\qquad\forall\phi\in \vH^2_{per}(D),\\
		\label{eq4.7}\|\chi - \cP_h^1\chi\| + h\|\nab(\chi-\cP_h^1\chi)\| &\leq C\,h^2\|\chi\|_{H^2}\qquad\forall\chi\in H^1_{per}/\mathbb{R}\cap H^2_{per}(D).
	\end{align}
	
	For the clarity we only 
	consider $P_1$-finite element space in this section (i.e., $\ell=1$), 
	the results of this section can be easily extended to high order finite elements. 
	
	%%%%%%
	\subsection{Finite element methods for the standard Chorin scheme}
	Approximating the velocity space and pressure space respectively by the finite element 
	spaces $\vH_h$ and $L_h$ in Algorithm 1, we then obtain the fully discrete 
	finite element version of the standard Chorin scheme given below as Algorithm 3. 
	We also note that a similar algorithm was proposed in \cite{CHP12}. 
	
	\medskip
	\noindent
	{\bf Algorithm 3} 
	
	Let $n\geq 0$. Set $\tilde{\vu}_h^0 = \cP_h^0\vu_0$. For $n=0,1,2,\cdots$ do the following steps:
	\smallskip
	
	{\em Step 1:} Given $\vu^n_h \in L^2(\Ome,\vH_h)$ and $\tilde{\vu}^n_h \in L^2(\Ome,\vH_h)$, find $\tilde{\vu}^{n+1}_h \in L^2(\Ome,\vH_h)$ such that $\mP$-a.s.
	\begin{align}\label{A3_step1}
		\bigl(\tilde{\vu}^{n+1}_h, &\vv_h \bigr)\, +\, k \bigl(\nab\tilde{\vu}^{n+1}_h,\nab\vv_h\bigr) \\\nonumber
		& = \bigl(\vu_h^n,\vv_h \bigr)  \,+\,k\bigl(\vf^{n+1}, \vv_h\bigr) \,+\, \bigl(\vB(\tilde{\vu}^n_h)\Delta W_{n+1},\nab \vv_h\bigr)\qquad\forall\vv_h\in \vH_h.
	\end{align}

	\smallskip
	{\em Step 2:} Find $p^{n+1}_h \in L^2(\Ome,L_h)$ such that $\mP$-a.s.
	\begin{align}\label{A3_step2} 
		\bigl(\nab p_h^{n+1},\nab\phi_h\bigr) = \frac{1}{k}\bigl(\tilde{\vu}_h^{n+1},\nab\phi_h\bigr)\qquad \forall\phi_h\in L_h.
	\end{align}
	
	\smallskip
	{\em Step 3:} Define $\vu^{n+1}_h \in L^2(\Ome, \vH_h)$ by
	\begin{align}
		%	\bigl(\vu_h^{n+1},\vv_h\bigr)= \bigl(\tilde{\vu}_h^{n+1},\vv_h \bigr) - k \bigl(\nab p_h^{n+1},\vv_h \bigr) \qquad \forall\vv_h\in \vH_h.
		\vu_h^{n+1} = \tilde{\vu}_h^{n+1} -k \nab p_h^{n+1}.
	\end{align}
	
	\medskip
	As mentioned in Section \ref{sec-1}, eliminating $\vu^n$ in \eqref{A3_step1}  using \eqref{A3_step2}, we obtain
	\begin{subequations}\label{eq4.11}
		\begin{align}\label{eq4.11a}
			(\tilde{\vu}^{n+1}_h &- \tilde{\vu}^n_h,\vv_h)\, +\, k(\nab \tilde{\vu}^{n+1}_h,\nab\vv_h) \\\nonumber
			& + k(\nab p^{n}_h,\vv_h) = k\big(\vf^{n+1},\vv_h\big) +  (\vB(\tilde{\vu}_h^n)\Delta W_{n+1},\vv_h)\qquad\forall\vv_h\in \vH_h,\\
			\label{eq4.11b}&(\tilde{\vu}_h^{n+1},\nab\phi_h) = k(\nab p_h^{n+1},\nab\phi_h) \qquad \forall\phi_h\in L_h.
		\end{align}
		
	\end{subequations}

Next, we state the stability estimates for $\{(\tilde{\vu}_h^n, p_h^n)\}_{n=0}^M$ in the following lemma, which will be used in the fully discrete error analysis later.
Since its proof follows from the same lines of that for Lemma \ref{lemma3.1}, we omit it 
to save space.

\begin{lemma}\label{lemma_fully_stability}
	Let $\{(\tilde{\vu}_h^n,p_h^n)\}_{n=0}^M$ be generated by Algorithm 3, then there holds 
	\begin{subequations}
		\begin{align}\label{eq4.14a}
			\max_{0\leq n \leq M}\mE [\|\tilde{\vu}_h^n\|^2] + \mE \bigg[\sum_{n=1}^{M} \|\tilde{\vu}_h^n - \tilde{\vu}_h^{n-1}\|^2\bigg] + \mE \bigg[k\sum_{n=0}^{M} \|\nab \tilde{\vu}_h^n\|^2\bigg] 
			&\leq C, \\
			\label{eq4.14b}	\mE \bigg[k\sum_{n=0}^{M}\|\nab p_h^n\|^2\bigg] &\leq \frac{C}k,
		\end{align}
	\end{subequations}
	where $C$ is a positive constant depending only on $D_T, u_0$, and $\vf$.	
\end{lemma}

	The following theorem provides an error estimate in a strong norm for the finite element solution of Algorithm 3.

	\begin{theorem}\label{them4.1} 
		Let $\{\big(\tilde{\vu}^m, p^m\big)\}_{m=0}^M$ and $\{(\tilde{\vu}_h^m,p_h^m)\}_{m=0}^M$ be generated respectively by Algorithm 1 and Algorithm 3.  Then under the assumptions of Lemmas \ref{lemma3.1}, \ref{lemma_fully_stability} and Corollary \ref{cor3.5} there holds
		\begin{align}\label{eq4.14}
			&\bigg(\mE\Bigl[ k\sum_{m=0}^M\bigl\|\tilde{\vu}^m - \tilde{\vu}_h^m \bigr\|^2\Bigr]\bigg)^{\frac12} \\\nonumber
			&\qquad\qquad+ \max_{0\leq \ell \leq M}\bigg(\mE\Bigl[\Bigr\|k\sum_{n=1}^{\ell}\nab(\tilde{\vu}^n - \tilde{\vu}_h^n)\Bigr\|^2\Bigr]\bigg)^{\frac12}  \leq C\Bigl(k^{\frac14}\, +\,  h k^{-\frac12} \Bigr),
		\end{align}
		where $C \equiv C(D_T, \vu_0, \vf)$ is a positive constant.
	\end{theorem}
	
	\begin{proof}
		The proof is conceptually similar to that of Theorem \ref{thm3.3}. 
		Setting $\ve_{\tilde{\vu}_h}^m =: \tilde{\vu}^m - \tilde{\vu}_h^m$ and $\veps_{p_h}^m =: p^m - p^m_h$. Without loss of the generality, we assume $\ve_{\tilde{\vu}_h}^0 = 0$ and $\veps_{p_h}^0 = 0$ because they are of high order accuracy, hence are negligible.    
		
		First, applying the summation operator $\sum_{n=0}^m$ to \eqref{eq4.11a}, we obtain
		\begin{align}\label{eq4.15}
			\bigl(\tilde{\vu}_h^{m+1},\vv_h\bigr) &+ k\Bigl(\sum_{n=0}^{m+1}\nab\tilde{\vu}_h^n,\nab\vv_h\Bigr) + k\Bigl(\sum_{n=0}^m\nab p^n_h,\vv_h\Bigr) \\\nonumber
			&= \bigl(\vu^0_h,\vv_h\bigr) + \Bigl(\sum_{n=0}^m \vB(\tilde{\vu}_h^n)\Delta W_{n+1},\vv_h\Bigr) \qquad \forall\vv_h\in \vH_{h}.
		\end{align}
Subtracting \eqref{eq3.20} from \eqref{eq4.15} yields the following error equations:
	\begin{align}\label{eq4.16}
		\bigl(\ve_{\tilde{\vu}_h}^{m+1},\vv_h\bigr) &+ k\Bigl(\sum_{n=0}^{m+1}\nab\ve_{\tilde{\vu}_h}^n,\nab\vv_h\Bigr)+ k\Bigl(\sum_{n=0}^m\nab\veps_{p_h}^n,\vv_h\Bigr) \\\nonumber
		&= \Bigl(\sum_{n=0}^m(\vB(\tilde{\vu}^n) - \vB(\tilde{\vu}_h^n))\Delta W_{n+1},\vv_h\Bigr)\qquad \forall\vv_h\in \vH_{h},\\
	\label{eq4.17}	\bigl(\ve_{\tilde{\vu}_h}^{m+1},\nab\phi_h\bigr) &= k\bigl(\nab\veps_{p_h}^{m+1},\nab\phi_h\bigr)\qquad \forall\phi_h\in L_h. 
	\end{align}
Choosing $\vv_h = \cP_h^0 \ve_{\tilde{\vu}_h}^{m+1} = \ve_{\tilde{\vu}_h}^{m+1} - \ta^{m+1}$; $\ta^m = \tilde{\vu}_h^m - \cP_h^0\tilde{\vu}_h^m$, then \eqref{eq4.16} becomes
	\begin{align}\label{eq4.18}
		\|\ve_{\tilde{\vu}_h}^{m+1}\|^2 &+ k\Bigl(\sum_{n=0}^{m+1}\nab\ve_{\tilde{\vu}_h}^n,\nab\ve_{\tilde{\vu}_h}^{m+1}\Bigr) + k\Bigl(\sum_{n=0}^m\nab\veps_{p_h}^n,\cP_h^0\ve_{\tilde{\vu}_h}^{m+1}\Bigr) \\\nonumber
		&= \bigl(\ve_{\tilde{\vu}_h}^{m+1},\ta^{m+1}\bigr) + k\Bigl(\sum_{n=0}^{m+1}\nab\ve_{\tilde{\vu}_h}^{n},\nab\ta^{m+1}\Bigr)\\\nonumber
		&\qquad + \Bigl(\sum_{n=0}^{m}(\vB(\tilde{\vu}^n) - \vB(\tilde{\vu}_h^n))\Delta W_{n+1},\cP_h^0\ve_{\tilde{\vu}_h}^{m+1}\Bigr).
	\end{align}	
Setting $\phi_h = \sum_{n=0}^m\cP_h^1\veps_{p_h}^n = \sum_{n=0}^m\veps_{p_h}^n -\sum_{n=0}^m\xi^n$
in \eqref{eq4.17}, where $\xi^n = p^n - \cP_h^1p^n$, we obtain
\begin{align}
	\Bigl(\ve_{\tilde{\vu}_h}^{m+1},\sum_{n=0}^m\nab\cP_h^1\veps_{p_h}^n\Bigr) = k\Bigl(\nab\veps_{p_h}^{m+1},\sum_{n=0}^m\nab\cP_h^1\veps_{p_h}^n\Bigr)
\end{align}
In addition, by using the properties of $\cP_h^0-$ and $\cP_h^1$-projection we have
\begin{align}\label{eq4.20}
	&k\biggl(\sum_{n=0}^m\nab\veps_{p_h}^n,\cP_h^0\ve_{\tilde{\vu}_h}^{m+1}\biggr) = k\Bigl(\sum_{n=0}^m\nab\veps_{p_h}^n,\ve_{\tilde{\vu}_h}^{m+1}\Bigr) - k\Bigl(\sum_{n=0}^m\nab\veps_{p_h}^n,\ta^{m+1}\Bigr)\\\nonumber
	&\quad = k\Bigl(\sum_{n=0}^m\nab\cP_h^1\veps_{p_h}^n,\ve_{\tilde{\vu}_h}^{m+1}\Bigr) + k\Bigl(\sum_{n=0}^m\nab\xi^n,\ve_{\tilde{\vu}_h}^{m+1}\Bigr) - k\Bigl(\sum_{n=0}^m\nab\veps_{p_h}^n,\ta^{m+1}\Bigr)\\\nonumber
	&\quad = k^2\Bigl(\nab\veps_{p_h}^{m+1},\sum_{n=0}^m\nab\cP_h^1\veps_{p_h}^n\Bigr) + k\Bigl(\sum_{n=0}^m\nab\xi^n,\ve_{\tilde{\vu}_h}^{m+1}\Bigr) - k\Bigl(\sum_{n=0}^m\nab\veps_{p_h}^n,\ta^{m+1}\Bigr)\\\nonumber
	&\quad = k^2\Bigl(\nab\veps_{p_h}^{m+1},\sum_{n=0}^{m+1}\nab\cP_h^1\veps_{p_h}^n\Bigr) - k^2\bigl(\nab\veps_{p_h}^{m+1},\nab\cP_h^1\veps_{p_h}^{m+1}\bigr) \\\nonumber
	&\qquad+ k\Bigl(\sum_{n=0}^m\nab\xi^n,\ve_{\tilde{\vu}_h}^{m+1}\Bigr)- k\Bigl(\sum_{n=0}^m\nab\veps_{p_h}^n,\ta^{m+1}\Bigr)\\\nonumber
	&\quad = k^2\Bigl(\nab\veps_{p_h}^{m+1},\sum_{n=0}^{m+1}\nab\veps_{p_h}^n\Bigr) - k^2\Bigl(\nab\veps_{p_h}^{m+1},\sum_{n=0}^{m+1}\nab\xi^n\Bigr) \\ \nonumber
	&\qquad - k^2\bigl(\nab\veps_{p_h}^{m+1},\nab\cP_h^1\veps_{p_h}^{m+1}\bigr)  + k\Bigl(\sum_{n=0}^m\nab\xi^n,\ve_{\tilde{\vu}_h}^{m+1}\Bigr)- k\Bigl(\sum_{n=0}^m\nab\veps_{p_h}^n,\ta^{m+1}\Bigr).
\end{align}
Moreover, by using the orthogonality property of $\cP_h^1$, we have 
\begin{align}
	- k^2\bigl(\nab\veps_{p_h}^{m+1},\nab\cP_h^1\veps_{p_h}^{m+1}\bigr) &= - k^2\bigl(\nab(\veps_{p_h}^{m+1} - \cP_h^1\veps_{p_h}),\nab\cP_h^1\veps_{p_h}^{m+1}\bigr) \\\nonumber
	&\qquad+ k^2\|\nab\cP_h^1\veps_{p_h}^{m+1}\|^2 
	= k^2\|\nab\cP_h^1\veps_{p_h}^{m+1}\|^2,
\end{align}
which helps to reduce \eqref{eq4.20} into 
\begin{align}\label{eq4.22}
	k\biggl(\sum_{n=0}^m\nab\veps_{p_h}^n,\cP_h^0\ve_{\tilde{\vu}_h}^{m+1}\biggr) &= k^2\Bigl(\nab\veps_{p_h}^{m+1},\sum_{n=0}^{m+1}\nab\veps_{p_h}^n\Bigr) - k^2\Bigl(\nab\veps_{p_h}^{m+1},\sum_{n=0}^{m+1}\nab\xi^n\Bigr) \\\nonumber
	&\quad + k^2\|\nab\cP_h^1\veps_{p_h}^{m+1}\|^2 + k\Bigl(\sum_{n=0}^m\nab\xi^n,\ve_{\tilde{\vu}_h}^{m+1}\Bigr)\\\nonumber
	&\quad- k\Bigl(\sum_{n=0}^m\nab\veps_{p_h}^n,\ta^{m+1}\Bigr).
\end{align}
Substituting \eqref{eq4.22} into \eqref{eq4.18} and rearranging terms yield 
\begin{align}\label{eq4.23}
		&\|\ve_{\tilde{\vu}_h}^{m+1}\|^2 + k\Bigl(\sum_{n=0}^{m+1}\nab\ve_{\tilde{\vu}_h}^n,\nab\ve_{\tilde{\vu}_h}^{m+1}\Bigr) + k^2\Bigl(\nab\veps_{p_h}^{m+1},\sum_{n=0}^{m+1}\nab\veps_{p_h}^n\Bigr)\\\nonumber
		&\qquad  + k^2\|\nab\cP_h^1\veps_{p_h}^{m+1}\|^2\\\nonumber
		&\quad = k^2\Bigl(\nab\veps_{p_h}^{m+1},\sum_{n=0}^{m+1}\nab\xi^n\Bigr) +
		 k\Bigl(\div\ve_{\tilde{\vu}_h}^{m+1},\sum_{n=0}^m\xi^n\Bigr) \\\nonumber
		 &\qquad+ k\Bigl(\sum_{n=0}^m\nab\veps_{p_h}^n,\ta^{m+1}\Bigr) + \bigl(\ve_{\tilde{\vu}_h}^{m+1},\ta^{m+1}\bigr) + k\Bigl(\sum_{n=0}^{m+1}\nab\ve_{\tilde{\vu}_h}^{n},\nab\ta^{m+1}\Bigr)\\\nonumber
		 &\qquad + \Bigl(\sum_{n=0}^{m}(\vB(\tilde{\vu}^n) - \vB(\tilde{\vu}_h^n))\Delta W_{n+1},\cP_h^0\ve_{\tilde{\vu}_h}^{m+1}\Bigr). 
\end{align}

Next, we use the identity $2a(a-b) = a^2- b^2 + (a-b)^2$ to create telescoping sums on the left side of \eqref{eq4.23} followed by taking the expectation to get
\begin{align}\label{eq4.24}
			&\mE\bigl[\|\ve_{\tilde{\vu}_h}^{m+1}\|^2\bigr] + \frac{k}{2}\mE\biggl[\biggl\|\sum_{n=0}^{m+1}\nab\ve_{\tilde{\vu}_h}^n\biggr\|^2 - \biggl\|\sum_{n=0}^{m}\nab\ve_{\tilde{\vu}_h}^n\biggr\|^2 + \|\nab\ve_{\tilde{\vu}_h}^{m+1}\|^2\biggr]\\\nonumber
			&\qquad + \frac{k^2}{2}\mE\biggl[\biggl\|\sum_{n=0}^{m+1}\nab\veps_{p_h}^n\biggr\|^2 - \bigg\|\sum_{n=0}^{m}\nab\veps_{p_h}^n\biggr\|^2 + \|\nab\veps_{p_h}^{m+1}\|^2\biggr]\\ \nonumber
	&\qquad + k^2\mE\bigl[\|\nab\cP_h^1\veps_{p_h}^{m+1}\|^2\bigr]\\\nonumber
	&\quad = k^2\mE\Bigl[\Bigl(\nab\veps_{p_h}^{m+1},\sum_{n=0}^{m+1}\nab\xi^n\Bigr)\Bigr] +
	k\mE\Bigl[\Bigl(\div\ve_{\tilde{\vu}_h}^{m+1},\sum_{n=0}^m\xi^n\Bigr)\Bigr] \\\nonumber
	&\qquad+ k\mE\Bigl[\Bigl(\sum_{n=0}^m\nab\veps_{p_h}^n,\ta^{m+1}\Bigr)\Bigr] + \mE\Bigl[\bigl(\ve_{\tilde{\vu}_h}^{m+1},\ta^{m+1}\bigr)\Bigr] \\\nonumber
	&\qquad + k\mE\Bigl[\Bigl(\sum_{n=0}^{m+1}\nab\ve_{\tilde{\vu}_h}^{n},\nab\ta^{m+1}\Bigr)\Bigr]\\\nonumber
	&\qquad + \mE\Bigl[\Bigl(\sum_{n=0}^{m}(\vB(\tilde{\vu}^n) - \vB(\tilde{\vu}_h^n))\Delta W_{n+1},\cP_h^0\ve_{\tilde{\vu}_h}^{m+1}\Bigr)\Bigr].
\end{align}
Now, applying $k\sum_{m=0}^{\ell}$ for $0\leq \ell < M$ to \eqref{eq4.24} we obtain
\begin{align}\label{eq4.25}
	&k\sum_{m=0}^{\ell}\mE\bigl[\|\ve_{\tilde{\vu}_h}^{m+1}\|^2\bigr] + \frac12\mE\Bigl[\Bigl\|k\sum_{n=0}^{\ell+1}\nab\ve_{\tilde{\vu}_h}^n\Bigr\|^2\Bigr] + \frac{k^2}{2}\sum_{m=0}^{\ell}\mE\Bigl[\|\nab\ve_{\tilde{\vu}_h}^{m+1}\|^2\Bigr]\\\nonumber
	&\qquad+ \frac{k}{2}\mE\Bigl[\Bigl\|k\sum_{n=0}^{\ell+1}\nab\veps_{p_h}^n\Bigr\|^2\Bigr] + \frac{k^3}{2}\sum_{m=0}^{\ell}\mE\Bigl[\|\nab\veps_{p_h}^{m+1}\|^2  + 2\|\nab\cP_h^1\veps_{p_h}^{m+1}\|^2 \Bigr]\\\nonumber
	&\quad = \mE\Bigl[k^3\sum_{m=0}^{\ell}\Bigl(\nab\veps_{p_h}^{m+1},\sum_{n=0}^{m+1}\nab\xi^n\Bigr)\Bigr] +
	\mE\Bigl[k^2\sum_{m=0}^{\ell}\Bigl(\div\ve_{\tilde{\vu}_h}^{m+1},\sum_{n=0}^m\xi^n\Bigr)\Bigr] \\\nonumber
	&\qquad+ \mE\Bigl[k^2\sum_{m=0}^{\ell}\Bigl(\sum_{n=0}^m\nab\veps_{p_h}^n,\ta^{m+1}\Bigr)\Bigr] + \mE\Bigl[k\sum_{m=0}^{\ell}\bigl(\ve_{\tilde{\vu}_h}^{m+1},\ta^{m+1}\bigr)\Bigr] \\\nonumber
	&\qquad + \mE\Bigl[k^2\sum_{m=0}^{\ell}\Bigl(\sum_{n=0}^{m+1}\nab\ve_{\tilde{\vu}_h}^{n},\nab\ta^{m+1}\Bigr)\Bigr]\\\nonumber
	&\qquad + k\sum_{m=0}^{\ell}\mE\Bigl[\Bigl(\sum_{n=0}^{m}(\vB(\tilde{\vu}^n) - \vB(\tilde{\vu}_h^n))\Delta W_{n+1},\cP_h^0\ve_{\tilde{\vu}_h}^{m+1}\Bigr)\Bigr]\\\nonumber
	&\quad := {\tt I + \cdots + VI}.
\end{align}
Next, we bound the right-hand side of \eqref{eq4.25} as follows. 
By using the discrete H\"older inequality we get
\begin{align}
	{\tt I} &= k\mE\Bigl[k\sum_{m=0}^{\ell}\Bigl(\nab\veps_{p_h}^{m+1},k\sum_{n=0}^{m+1}\nab\xi^n\Bigr)\Bigr]\\\nonumber
	&\leq k\mE\biggl[k\sum_{m=0}^{\ell}\|\nab\veps_{p_h}^{m+1}\|\biggl\|k\sum_{n=0}^{m+1}\nab\xi^n\biggr\|\biggr]\\\nonumber
	&\leq k\mE\Bigl[\Bigl(k\sum_{m=0}^{\ell}\|\nab\veps_{p_h}^{m+1}\|^2\Bigr)^{\frac12}\Bigl(k\sum_{m=0}^{\ell}\Bigl\|k\sum_{n=0}^{m+1}\nab\xi^n\Bigr\|^2\Bigr)^{\frac12}\Bigr]\\\nonumber
	&\leq k\Bigl(\mE\Bigl[k\sum_{m=0}^{\ell}\|\nab\veps_{p_h}^{m+1}\|^2\Bigr]\Bigr)^{\frac12}\Bigl(\mE\Bigl[k\sum_{m=0}^{\ell}\Bigl\|k\sum_{n=0}^{m+1}\nab\xi^n\Bigr\|^2\Bigr]\Bigr)^{\frac12}.
\end{align}
In addition, by using the stability estimates from \eqref{eq3.5b} and \eqref{eq4.14b}, we have 
\begin{align*}
	\mE\Bigl[k\sum_{m=0}^{\ell}\|\nab\veps_{p_h}^{m+1}\|^2\Bigr] \leq \frac{C}{k},
\end{align*}
Similarly, using \eqref{eq4.7} and the stability estimate from \eqref{eq3.56} we get
\begin{align*}
	\mE\biggl[k\sum_{m=0}^{\ell}\Bigl\|k\sum_{n=0}^{m+1}\nab\xi^n\Bigr\|^2\biggr]&= 	\mE\biggl[k\sum_{m=0}^{\ell}\Bigl\|\nab\Bigl( k\sum_{n=0}^{m+1} p^n - \cP_h^1 \Bigl(k\sum_{n=0}^{m+1} p^n\Bigr)\Bigr)\Bigr\|^2\biggr]\\\nonumber
	& \leq C\mE\biggl[k\sum_{m=0}^{\ell}\Bigl\|k\sum_{n=0}^{m+1}\nab p^n\Bigr\|^2\biggr]\leq C.
\end{align*}
Therefore, ${\tt I} \leq Ck^{\frac12}$.

Next, using the fact that $\Bigl\|\sum_{n=0}^m\xi^n\Bigr\| \leq Ch\Bigl\|\sum_{n=0}^m\nab p^n\Bigr\|$ we have
\begin{align}
	{\tt II} &= \mE\Bigl[k^2\sum_{m=0}^{\ell}\Bigl(\div\ve_{\tilde{\vu}_h}^{m+1},\sum_{n=0}^m\xi^n\Bigr)\Bigr]\\\nonumber
	&\leq \frac{k^2}{8}\sum_{m=0}^{\ell}\mE\bigl[\|\nab\ve_{\tilde{\vu}_h}^{m+1}\|^2\bigr] + C\mE\Bigl[k^2\sum_{m=0}^{\ell}\Bigl\|\sum_{n=0}^m\xi^n\Bigr\|^2\Bigr]\\\nonumber
	&\leq \frac{k^2}{8}\sum_{m=0}^{\ell}\mE\bigl[\|\nab\ve_{\tilde{\vu}_h}^{m+1}\|^2\bigr] + \frac{Ch^2}{k}\mE\Bigl[k\sum_{m=0}^{\ell}\Bigl\|k\sum_{n=0}^m\nab p^n\Bigr\|^2\Bigr]\\\nonumber
	&\leq \frac{k^2}{8}\sum_{m=0}^{\ell}\mE\bigl[\|\nab\ve_{\tilde{\vu}_h}^{m+1}\|^2\bigr] + \frac{Ch^2}{k},
\end{align}
where \eqref{eq3.56} was used to obtain the last inequality. Expectedly, the first term 
will be absorbed to the left-hand side of \eqref{eq4.25} later.

Next, using summation by parts we obtain 
\begin{align}\label{eq4.28}
	{\tt III} &= \mE\Bigl[k^2\sum_{m=0}^{\ell}\Bigl(\sum_{n=0}^m\nab\veps_{p_h}^n,\ta^{m+1}\Bigr)\Bigr]\\\nonumber
	&= \mE\Bigl[k^2\Bigl(\sum_{n=0}^{\ell + 1}\ta^n, \sum_{n=0}^{\ell+1}\nab\veps_{p_h}^n\Bigr)\Bigr] - \mE\Bigl[k^2\sum_{m=0}^{\ell}\Bigl(\sum_{n=0}^m\ta^n,\nab\veps_{p_h}^{m}\Bigr)\Bigr]\\\nonumber
	&\leq Ck\mE\Bigl[\Bigl\|\sum_{n=0}^{\ell+1}\ta^n\Bigr\|^2\Bigr] +\frac{k^3}{8}\mE\Bigl[\Bigl\|\sum_{n=0}^{\ell+1}\nab\veps_{p_h}^{n}\Bigr\|^2\Bigr]\\\nonumber
	&\qquad+ C\mE\Bigl[k\sum_{m=0}^{\ell}\Bigl\|\sum_{n=0}^m\ta^n\Bigr\|^2\Bigr] + \frac{k^3}{8}\sum_{m=0}^{\ell}\mE\bigl[\|\nab\veps_{p_h}^m\|^2\bigr].
\end{align}
In addition, we can use \eqref{eq4.6} and \eqref{eq3.55} to control the first and third terms on the right side of \eqref{eq4.28} as follows:
\begin{align}
	&Ck\mE\Bigl[\Bigl\|\sum_{n=0}^{\ell+1}\ta^n\Bigr\|^2\Bigr] + C\mE\Bigl[k\sum_{m=0}^{\ell}\Bigl\|\sum_{n=0}^m\ta^n\Bigr\|^2\Bigr] \\\nonumber
	&\qquad \leq Ch^4k\mE\Bigl[\Bigl\|\sum_{n=0}^{\ell+1}\Delta\tilde{\vu}^n\Bigr\|^2\Bigr] + Ch^4\mE\Bigl[k\sum_{m=0}^{\ell}\Bigl\|\sum_{n=0}^m\Delta\tilde{\vu}^n\Bigr\|^2\Bigr]
	\leq \frac{Ch^4}{k^2}.
\end{align}
Therefore, \begin{align}\label{eq4.30}
	{\tt III} \leq \frac{Ch^4}{k^2} + \frac{k^3}{8}\mE\biggl[\biggl\|\sum_{n=0}^{\ell+1}\nab\veps_{p_h}^{n}\biggr\|^2\biggr] + \frac{k^3}{8}\sum_{m=0}^{\ell}\mE\bigl[\|\nab\veps_{p_h}^m\|^2\bigr],
\end{align}
Again, expectedly, the last two terms on the right-hand side of \eqref{eq4.30} will be absorbed to the left side of \eqref{eq4.25} later.
		
It follows from \eqref{eq4.6} and \eqref{eq3.56b} that
\begin{align}
	{\tt IV} &= \mE\biggl[k\sum_{m=0}^{\ell}\bigl(\ve_{\tilde{\vu}_h}^{m+1},\ta^{m+1}\bigr)\biggr]\\\nonumber
	&\leq  k\sum_{m=0}^{\ell}\mE\bigl[\|\ta^{m+1}\|^2\bigr]+ \frac14 k\sum_{m=0}^{\ell}\mE\bigl[\|\ve_{\tilde{\vu}_h}^{m+1}\|^2\bigr] \\\nonumber
	&\leq Ch^4 \mE\Bigl[k\sum_{m=0}^{\ell}\|\Delta\tilde{\vu}^{m+1}\|^2\Bigr] + \frac14 k\sum_{m=0}^{\ell}\mE\bigl[\|\ve_{\tilde{\vu}_h}^{m+1}\|^2\bigr] \\\nonumber
	&\leq \frac{Ch^4}{k} + \frac14 k\sum_{m=0}^{\ell}\mE\bigl[\|\ve_{\tilde{\vu}_h}^{m+1}\|^2\bigr].
\end{align}	

To estimate term {\tt V}, we approach similarly as done for term {\tt III}. Namely, fist we use the summation by parts and then use \eqref{eq4.6} and \eqref{eq3.56}.
\begin{align}
	{\tt V} &= \mE\biggl[k^2\sum_{m=0}^{\ell}\Bigl(\sum_{n=0}^{m+1}\nab\ve_{\tilde{\vu}_h}^{n},\nab\ta^{m+1}\biggr)\Bigr]\\\nonumber
	&=\mE\biggl[k^2\Bigl(\sum_{n=0}^{\ell+1}\nab\ve_{\tilde{\vu}_h}^{n},\sum_{n=0}^{\ell+1}\nab\ta^{n}\Bigr)\biggr] - \mE\biggl[k^2\sum_{m=0}^{\ell}\Bigl(\sum_{n=0}^m\nab\ta^n, \nab\ve_{\tilde{\vu}_h}^{m+1}\Bigr)\biggr]\\\nonumber
	&\leq \frac18\mE\Bigl[\Bigl\|k\sum_{n=0}^{\ell+1}\nab\ve_{\tilde{\vu}_h}^n\Bigr\|^2\Bigr] + C\mE\Bigl[\Bigl\|k\sum_{n=0}^{\ell+1}\nab\ta^n\Bigr\|^2\Bigr]\\\nonumber
	&\quad + \frac{1}{8}\mE\Bigl[k^2\sum_{m=0}^{\ell}\|\nab\ve_{\tilde{\vu}_h}^{m+1}\|^2\Bigr] + C\mE\Bigl[k^2\sum_{m=0}^{\ell}\Bigl\|\sum_{n=0}^{m}\nab\ta^n\Bigr\|^2\Bigr] \\\nonumber
	&\leq \frac18\mE\Bigl[\Bigl\|k\sum_{n=0}^{\ell+1}\nab\ve_{\tilde{\vu}_h}^n\Bigr\|^2\Bigr] + \frac{1}{8}\mE\Bigl[k^2\sum_{m=0}^{\ell}\|\nab\ve_{\tilde{\vu}_h}^{m+1}\|^2\Bigr]  \\\nonumber
	&\quad+ Ch^2\mE\Bigl[\Bigl\|k\sum_{n=0}^m\Delta\tilde{\vu}^n\Bigr\|^2\Bigr] + \frac{Ch^2}{k}\mE\Bigl[k\sum_{m=0}^{\ell}\Bigl\|k\sum_{n=0}^m\Delta\tilde{\vu}^{n}\Bigr\|^2\Bigr]\\\nonumber
	&\leq \frac18\mE\Bigl[\Bigl\|k\sum_{n=0}^{\ell+1}\nab\ve_{\tilde{\vu}_h}^n\Bigr\|^2\Bigr] + \frac{1}{8}\mE\Bigl[k^2\sum_{m=0}^{\ell}\|\nab\ve_{\tilde{\vu}_h}^{m+1}\|^2\Bigr] + \frac{Ch^2}{k}.
\end{align}

We use the It\^o isometry to handle the noise term as follows:
\begin{align}
	{\tt VI} &= k\sum_{m=0}^{\ell}\mE\Bigl[\Bigl(\sum_{n=0}^{m}(\vB(\tilde{\vu}^n) - \vB(\tilde{\vu}_h^n))\Delta W_{n+1},\cP_h^0\ve_{\tilde{\vu}_h}^{m+1}\Bigr)\Bigr]\\\nonumber
	&\leq \frac14 k\sum_{m=0}^{\ell}\mE\bigl[\|\ve_{\tilde{\vu}_h}^{m+1}\|^2\bigr] + k\sum_{m=0}^{\ell}\mE\Bigl[\Bigl\|\sum_{n=0}^m(\vB(\tilde{\vu}^n) - \vB(\tilde{\vu}_h^n))\Delta W_{n+1}\Bigr\|^2\Bigr]\\\nonumber
	&= \frac14 k\sum_{m=0}^{\ell}\mE\bigl[\|\ve_{\tilde{\vu}_h}^{m+1}\|^2\bigr] + k\sum_{m=0}^{\ell}\mE\Bigl[k\sum_{n=0}^{m}\|\vB(\tilde{\vu}^n) - \vB(\tilde{\vu}_h^n)\|^2\Bigr]\\\nonumber
	&\leq \frac14 k\sum_{m=0}^{\ell}\mE\bigl[\|\ve_{\tilde{\vu}_h}^{m+1}\|^2\bigr] + Ck\sum_{m=0}^{\ell}k\sum_{n=0}^m\mE\bigl[\|\ve_{\tilde{\vu}_h}^n\|^2\bigr].
\end{align}

Finally, substituting the above estimates for terms {\tt I, II, III, IV, V, VI} into \eqref{eq4.25} yields the following inequality for $X^{\ell} = k\sum_{m=0}^{\ell} \mE\bigl[\|\ve_{\tilde{\vu}_h}^m\|^2\bigr]$:
\begin{align}\label{eq4.34}
	&\frac{1}{2}X^{\ell+1} + \frac38\mE\Bigl[\Bigl\|k\sum_{n=0}^{\ell+1}\nab\ve_{\tilde{\vu}_h}^n\Bigr\|^2\Bigr] + \frac{k^2}{4}\sum_{m=0}^{\ell}\mE\Bigl[\|\nab\ve_{\tilde{\vu}_h}^{m+1}\|^2\Bigr]\\\nonumber
	&\qquad + \frac{3k}{8}\mE\Bigl[\Bigl\|k\sum_{n=0}^{\ell+1}\nab\veps_{p_h}^n\Bigr\|^2\Bigr] + \frac{3k^3}{8}\sum_{m=0}^{\ell}\mE\bigl[\|\nab\veps_{p_h}^{m+1}\|^2\bigr]\\\nonumber
	&\qquad + k^3\sum_{m=0}^{\ell}\mE\bigl[\|\nab\cP_h^1\veps_{p_h}^{m+1}\|^2\bigr]\\\nonumber
	&\leq Ck^{\frac12} + \frac{Ch^2}{k} + Ck\sum_{m=0}^{\ell} X^m. 
	%\\\nonumber
	%&\leq C\Bigl(k^{\frac12} + \frac{h^2}{k}\Bigr)\exp(Ct_{\ell}).
\end{align}
The desired error estimate \eqref{eq4.14} then follows from an application of the 
discrete Gronwall inequality  to \eqref{eq4.34}.  The proof is complete.
\end{proof}

	Next, we state an error estimate result for the pressure approximation generated by Algorithm 3
	in a time-averaged fashion. Recall that an important advantage of Chorin-type schemes is to allow the use of a pair of independent finite element spaces which are not required to satisfy a discrete inf-sup condition, a price for this advantage is to make error estimates for  
	the pressure approximations become more complicated even in the deterministic case. 
	The idea for circumventing the difficulty is to utilize the following perturbed 
	inf-sup inequality  (cf. \cite{hughes}): there exists $\delta >0$ independent of $h>0$, 
	such that   
	\begin{equation}\label{lbb_disk}
		\frac{1}{\delta^2} \Vert q_h\Vert^2 \leq \sup_{{\bf v}_h \in {\vH}_{h}}
		\frac{\vert(q_h, {\rm div}\, {\bf v}_h)\vert^2}{\Vert \nabla {\bf v}_h\Vert^2}
		+ h^2 \Vert \nabla q_h\Vert^2 \qquad \forall\, q_h \in S_{h}\, ,
	\end{equation} 
	which was also used in \cite{Feng1} to derive an error estimate for a pressure-stabilization 
	scheme for \eqref{eq1.1}. 
	
	\begin{theorem}\label{them4.2} 
		Under the assumptions of Theorem \ref{them4.1}, there exists a positive constant $C \equiv C(D_T, \vu_0, \vf, \delta)$ such that
		\begin{align}\label{eq_4.37}
			\bigg(\mE\bigg[k\sum_{m=0}^M\Bigl\|k\sum_{m=1}^m \bigl( p^n -p_h^n \bigr)  \Bigr\|^2\bigg]\bigg)^{\frac12} \leq C\Bigl(k^{\frac14} \,+\,  h k^{-\frac12} \Bigr),
		\end{align}
	\end{theorem}
	
	\begin{proof} 
		We reuse all the notations from the proof of Theorem \ref{them4.1}. First, from the error equations \eqref{eq4.16} we have
		\begin{align}\label{eq4.37}
			\Bigl(k\sum_{n=0}^m\veps_{p_h}^n,\div \vv_h\Bigr) &= \bigl(\ve_{\tilde{\vu}_h}^{m+1},\vv_h\bigr) + \Bigl(k\sum_{n=0}^{m+1}\nab\ve_{\tilde{\vu}_h}^n,\nab\vv_h\Bigr) \\\nonumber
			&\qquad- \Bigl(\sum_{n=0}^m(\vB(\tilde{\vu}^n) - \vB(\tilde{\vu}_h^n))\Delta W_{n+1},\vv_h\Bigr) \,\,\forall \vv_h\in \vH_{h}.
		\end{align}
	Using the Schwarz inequality on the right-hand side of \eqref{eq4.37} yields 
	\begin{align}\label{eq4.388}
		\Bigl|\Bigl(k\sum_{n=0}^m\veps_{p_h}^n,\div \vv_h\Bigr)\Bigr| &= C\|\ve_{\tilde{\vu}_h}^{m+1}\|\|\nab\vv_h\| + \Bigl\|k\sum_{n=0}^{m+1}\nab\ve_{\tilde{\vu}_h}^n\Bigr\|\|\nab\vv_h\| \\\nonumber
		&\quad+ C\Bigl\|\sum_{n=0}^m(\vB(\tilde{\vu}^n) - \vB(\tilde{\vu}_h^n))\Delta W_{n+1}\Bigr\|\|\nab\vv_h\|.
	\end{align}
Next, using \eqref{lbb_disk} we conclude that
\begin{align}\label{eq4.399}
	\frac{1}{\delta^2}\Bigl\|k\sum_{n=0}^m\veps_{p_h}^n\Bigr\|^2 - h^2\Bigl\|k\sum_{n=0}^m\nab\veps_{p_h}^n\Bigr\|^2 &\leq \sup_{{\bf v}_h \in {\vH}_{h}}\frac{\Bigl|\Bigl(k\sum_{n=0}^m\veps_{p_h}^n,\div \vv_h\Bigr)\Bigr|^2}{\|\nab\vv_h\|^2}\\\nonumber
	&\leq C\|\ve_{\tilde{\vu}_h}^{m+1}\|^2 + C\Bigl\|k\sum_{n=0}^{m+1}\nab\ve_{\tilde{\vu}_h}^{n}\Bigr\|^2\\\nonumber
	& + C\Bigl\|\sum_{n=0}^m(\vB(\tilde{\vu}^n) - \vB(\tilde{\vu}_h^n))\Delta W_{n+1}\Bigr\|^2.
\end{align}
Then, applying operators $k\sum_{m=0}^{\ell}$ and $\mE[\cdot]$ on both sides we obtain
\begin{align}\label{eq4.40}
	\frac{1}{\delta^2}\mE\biggl[k\sum_{m=0}^{\ell}\Bigl\|k\sum_{n=0}^m\veps_{p_h}^n\Bigr\|^2\biggr] &\leq h^2 \mE\biggl[k\sum_{m=0}^{\ell}\Bigl\|k\sum_{n=0}^m\nab\veps_{p_h}^n\Bigr\|^2\biggr] \\\nonumber
	&\quad + Ck\sum_{m=0}^{\ell}\mE\bigl[\|\ve_{\tilde{\vu}_h}^{m+1}\|^2\bigr] \\\nonumber
	&\quad + Ck\sum_{m=0}^{\ell}\mE\Bigl[\Bigl\|k\sum_{n=0}^{m+1}\nab\ve_{\tilde{\vu}_h}^n\Bigr\|^2\Bigr]\\\nonumber
	&\quad + Ck\sum_{m=0}^{\ell}\mE\Bigl[\Bigl\|\sum_{n=0}^m(\vB(\tilde{\vu}^n) - \vB(\tilde{\vu}_h^n))\Delta W_{n+1}\Bigr\|^2\Bigr]\\\nonumber
	&\quad := {\tt I + II + III + IV}.
\end{align}
We now bound each terms on the right-hand side of \eqref{eq4.40}. By using the discrete Jensen inequality and the stability estimates from \eqref{eq3.5b} and\eqref{eq4.14b} we get
\begin{align}
	{\tt I} &= h^2 \mE\biggl[k\sum_{m=0}^{\ell}\Bigl\|k\sum_{n=0}^m\nab\veps_{p_h}^n\Bigr\|^2\biggr] \leq h^2 \mE\biggl[k\sum_{m=0}^{\ell}k\sum_{n=0}^m\|\nab\veps_{p_h}^n\|^2\biggr]\leq \frac{Ch^2}{k}.
\end{align}
Using Theorem \ref{them4.1}, terms {\tt II} and {\tt III} can be bounded as follows:
\begin{align}
	{\tt II + III} &= Ck\sum_{m=0}^{\ell}\mE\bigl[\|\ve_{\tilde{\vu}_h}^{m+1}\|^2\bigr] + Ck\sum_{m=0}^{\ell}\mE\biggl[\biggl\|k\sum_{n=0}^{m+1}\nab\ve_{\tilde{\vu}_h}^n\biggr\|^2\biggr] \\\nonumber
	&\leq C\Bigl(k^{\frac12} + \frac{Ch^2}{k}\Bigr).
\end{align}
Finally, using It\^o's isometry and Theorem \ref{them4.1} we have
\begin{align}
	{\tt IV} &= Ck\sum_{m=0}^{\ell}\mE\biggl[\biggl\|\sum_{n=0}^m(\vB(\tilde{\vu}^n) - \vB(\tilde{\vu}_h^n))\Delta W_{n+1}\biggr\|^2\biggr]\\\nonumber
	&= Ck\sum_{m=0}^{\ell}\mE\biggl[k\sum_{n=0}^m\|\vB(\tilde{\vu}^n) - \vB(\tilde{\vu}_h^n)\|^2\biggr]\\\nonumber
	&\leq Ck\sum_{m=0}^{\ell}\mE\biggl[k\sum_{n=0}^m\|\ve_{\tilde{\vu}_h}^n\|^2\biggr]\\\nonumber
	&\leq C\Bigl(k^{\frac12} + \frac{Ch^2}{k}\Bigr).
\end{align}
The proof is complete after combining the above estimates.
\end{proof}
	
	We are now ready to state the following global error estimate theorem for Algorithm 3
	which is a main result of this paper.
	
	\begin{theorem}\label{thm4.3} 
		Under the assumptions of Theorems \ref{thm3.3}, \ref{thm3.4} and Theorems \ref{them4.1} and \ref{them4.2},  there hold the following error estimates:
		\begin{align}\label{eq4.38}
		\biggl(\mE\Bigl[k\sum_{m=0}^M\|\vu(t_m) - \tilde{\vu}^m_h\|^2\Bigr]\biggr)^{\frac12} &+ \max_{0\leq \ell \leq M} \biggl(\mE\Bigl[\Bigl\|k\sum_{m=0}^{\ell}\nab(\vu(t_m) - \tilde{\vu}^m_h)\Bigr\|^2\Bigr]\biggr)^{\frac12}\\\nonumber &\qquad\leq C\Bigl(k^{\frac14} +  h k^{-\frac12} \Bigr),\\
		\biggl(\mE\Bigl[k\sum_{m=0}^M\Bigl\|P(t_m) - k\sum_{n=0}^m p^n_h&\Bigr\|^2\Bigr]\biggr)^{\frac12} \leq C\Bigl(k^{\frac14} +  h k^{-\frac12} \Bigr),
		\end{align}
		where $C \equiv C(D_T, \vu_0, \vf, \beta,\delta)$ is positive constant independent of $k$ and $h$.
	\end{theorem}

	%%%%%%%%%
	\subsection{Finite element methods for the modified Chorin scheme} 
	In this subsection, we first formulate a finite element spatial discretization for 
	Algorithm 2 and then present a complete convergence analysis by deriving error 
	estimates which are stronger than those obtained above  for the standard Chorin scheme.

	\medskip
	\noindent
	{\bf Algorithm 4} 
	
	Let $m\geq 0$. Set $\tilde{\vu}_h^0 = \cP_h^0\vu_0$. For $m=0,1,2,\cdots$ do the following steps:
	\smallskip
	
	\smallskip
	{\em Step 1:} For given $\tilde{\vu}^m_h \in L^2(\Ome,\vH_h)$, find $\xi^m_h \in L^2(\Ome,S_h)$ by solving the 
	following Poisson problem: for $\mP$-a.s.
	\begin{align}\label{eq_4.7}
		\big(\nab\xi^m_h,\nab\phi_h\big) = \big(\vB(\tilde{\vu}^m_h),\nab\phi_h\big) \qquad \forall \phi_h \in S_h.
	\end{align}
	
	{\em Step 2:} Set \, $\pmb{\eta}^m_{\tilde{\vu}_{h}} = \vB(\tilde{\vu}^m_h) - \nab\xi^m_h$. For given $\vu^m_h \in L^2(\Ome,\vH_h)$ and $\tilde{\vu}^m_h \in L^2(\Ome,\vH_h)$, find $\tilde{\vu}^{m+1}_h \in L^2(\Ome,\vH_h)$  by solving the following problem: for $\mP$-a.s.
	\begin{align}\label{eq_4.8}
		\big(\tilde{\vu}^{m+1}_h,\vv_h\big) &+ k \big(\nab\tilde{\vu}^{m+1}_h,\nab\vv_h\big) \\\nonumber
		& = \big(\vu^m_h,\vv_h\big)  + k\big(\vf^{m+1},\vv_h\big) + \big(\pmb{\eta}^m_{\tilde{\vu}_{h}}\Delta W_{m+1},\vv_h\big)\qquad\forall\, \vv_h \in \vH_h.
	\end{align}
	
	\smallskip
	{\em Step 3:} Find $r^{m+1}_{h} \in L^2(\Ome,L_h)$ by solving the following Poisson problem: for $\mP$-a.s.
	\begin{align}\label{eq_4.9}
		\big(\nab r^{m+1}_h,\nab\phi_h\big) &= \frac{1}{k}\big( \tilde{\vu}^{m+1}_h,\,\nab\phi_h\big)\qquad \forall \phi_h \in L_h.
	\end{align}
	
	\smallskip
	{\em Step 4:} Define $\vu^{m+1}_h \in L^2(\Ome, \vH_h)$ by 
	\begin{align}\label{eq_4.10}
		%\big(\vu^{m+1}_h,\,\pmb{\psi}_h\big) &= \big(\tilde{\vu}^{m+1}_h,\,\pmb{\psi}_h\big) - k\big( \nab r^{m+1}_h,\pmb{\psi}_h\big) \qquad \forall \pmb{\psi}_h \in \vH_h.
		\vu^{m+1}_h = \tilde{\vu}^{m+1}_h -k \nab r^{m+1}_h. 
	\end{align}
	
	\smallskip
	{\em Step 5: } Define $p^{m+1}_{h} \in L^2(\Ome,L_h)$ by
	\begin{align}\label{eq_4.11}
		p^{m+1}_h = r^{m+1}_h + \frac{1}{k}\xi^{m}_h\Delta W_{m+1}. 
	\end{align}	
	
	Since each step involves a coercive problem, hence,  
	Algorithm 4 is well defined. The next theorem establishes a convergence rate 
	for the finite element approximation of the velocity field. Since the proof 
	follows the same lines as those in the proof of Theorem \ref{them4.1}, we omit it
	to save space. 
	%We also note that a similar result  was mentioned in \cite{CHP12} without a proof. 

	\begin{theorem}\label{thm4.1} 
		Let $\{\tilde{\vu}^m\}_{m=0}^M $ and $\{\tilde{\vu}_h^m\}_{m=0}^M$ be generated respectively by Algorithm 2 and 4. 
		%Assume that $\|\tilde{\vu}^0 - \tilde{\vu}_h^0\|_{L^2}^2 \leq C\,h^2$. 
		Then, there exists a constant $C \equiv C(D_T, \vu_0, \vf) > 0$ such that
		\begin{align}
			\max_{1\leq m \leq M}\biggl(\mE\bigl[\|\tilde{\vu}^m - \tilde{\vu}_h^m\|^2 \bigl]\biggr)^{\frac12} 
			&+ \biggl(\mE\biggl[k\sum_{m=1}^M \|\nab(\tilde{\vu}^m - \tilde{\vu}_h^m)\|^2\biggr]\biggr)^{\frac12} \\\nonumber
			&\leq C\biggl(\sqrt{k} + h +  h^2 k^{-\frac12} \biggr).
		\end{align}
		
	\end{theorem}

	In the next theorem, we establish an error estimate for the pressure approximation of the  modified Chorin finite element method given by Algorithm 4.
	
	\begin{theorem}\label{thm4.2} 
		Let $\{r^m\}_{m=1}^M$ and $\{r_h^m\}_{m=1}^M$ be generated respectively by Algorithm 2 and 4. Then, there exists a constant $C \equiv C(D_T,\vu_0,\vf, \delta)>0$ such that
		\begin{align}\label{eq4.52}
			\biggl(\mE\biggl[\biggl\|k\sum_{m=1}^{M}(r^m - r^m_h)\biggr\|^2\biggr]\biggr)^{\frac12} \leq C\biggl(\sqrt{k} + h +   h^2 k^{-\frac12}\biggr).
		\end{align}	
	\end{theorem}
	
	\begin{proof} 
		Let $\e_{\tilde{\vu}}^m = \tilde{\vu}^m - \tilde{\vu}_h^m$ and $\varepsilon_r^m = r^m - r^m_h$. It is easy to check that $(\e_{\tilde{\vu}}^m,\varepsilon_r^m)$ satisfies the following error equation:
		\begin{align}\label{eq_4.13}
			\bigl(\e_{\tilde{\vu}}^{m+1}-\e_{\tilde{\vu}}^m,\vv_h\bigr) + k\bigl(\nab\e_{\tilde{\vu}}^{m+1},&\,\nab \vv_h\bigr) + k\bigl(\nab \varepsilon_r^m,\vv_h\bigr) \\\nonumber
			&= \bigl( (\pmb{\eta}_{\tilde{\vu}}^m - \pmb{\eta}_{\tilde{\vu}_{h}}^m)\Delta W_{m+1},\vv_h\bigr) \qquad\forall\,\vv_h \in \vH_h.
		\end{align}
		Applying the summation operator $\sum_{m=0}^{\ell} \,(0\leq \ell \leq M-1)$ 
		to \eqref{eq_4.13} yields 
		\begin{align}
			\bigl(\e_{\tilde{\vu}}^{\ell+1}-\e_{\tilde{\vu}}^0,\vv_h\bigr) &+ \biggl(k\sum_{m=0}^{\ell}\nab\e_{\tilde{\vu}}^{m+1},\nab \vv_h\biggr) + \bigl(k\sum_{m=0}^{\ell}\nab \varepsilon_r^m,\vv_h\bigr) \\\nonumber
			&= \biggl(\sum_{m=0}^{\ell}(\pmb{\eta}_{\tilde{\vu}}^m - \pmb{\eta}_{\tilde{\vu}_{h}}^m)\Delta W_{m+1},\vv_h\biggr) \qquad\forall\,\vv_h \in \vH_h.
		\end{align}
		Thus, 
		\begin{align}\label{eq_4.30}
			\biggl(k\sum_{m=0}^{\ell}\varepsilon_r^m, \div \vv_h\biggr) &= \bigl(\e_{\tilde{\vu}}^{\ell+1} - \e_{\tilde{\vu}}^0, \vv_h\bigr) + \biggl(k\sum_{m=0}^{\ell}\nab\e_{\tilde{\vu}}^{m+1},\nab \vv_h\biggr) \\ \nonumber 
			&\quad -  \biggl(\sum_{m=0}^{\ell}(\pmb{\eta}_{\tilde{\vu}}^m 
			- \pmb{\eta}_{\tilde{\vu}_{h}}^m)\Delta W_{m+1},\vv_h\biggr) 
			=:{\tt I} + {\tt II} + {\tt III}.
		\end{align}
		Using Poincar\'e and Schwarz inequalities, the three terms on the right-hand side of \eqref{eq_4.30} can be bounded as follows:
		\begin{align}\label{eq_4.31}
			|{\tt I}| &\leq C\bigl(\|\e_{\tilde{\vu}}^{\ell+1}\| + \|\e_{\tilde{\vu}}^0\| \bigr)\|\nab\vv_h\|,\\
			|{\tt II}| &\leq \biggl\|k\sum_{m=0}^{\ell}\nab\e_{\tilde{\vu}}^{m+1}\biggr\|\|\nab \vv_h\|,\\
			\label{eq_4.33}
			|{\tt III}| &\leq C\bigg\|\sum_{m=0}^{\ell}(\pmb{\eta}_{\tilde{\vu}}^m - \pmb{\eta}_{\tilde{\vu}_h}^m)\Delta W_{m+1}\bigg\| \bigl\|\nab\vv_h \bigr\|.
		\end{align}
		Using \eqref{lbb_disk} and \eqref{eq_4.31}--\eqref{eq_4.33} in \eqref{eq_4.30}, 
		we obtain 
		\begin{align}\label{eq_4.34}
			\frac{1}{\delta^2}\biggl\|k\sum_{m=0}^{\ell} \varepsilon_r^m\biggr\|^2 - h^2\biggl\|k\sum_{m=0}^{\ell} \nab\varepsilon_r^m\biggr\|^2 
			&\leq C\bigl\|\e_{\tilde{\vu}}^{\ell+1} \bigr\|^2 
			+ k\sum_{m=0}^{\ell} \bigl\|\nab\e_{\tilde{\vu}}^{m+1} \bigr\|^2 \\\nonumber
			&\quad+ C\biggl\|\sum_{m=0}^{\ell}(\pmb{\eta}_{\tilde{\vu}}^m - \pmb{\eta}_{\tilde{\vu}_h}^m)\Delta W_{m+1}\biggr\|^2.
		\end{align}
		Using it\^o isometry, the last term above can be bounded as
		\begin{align}\label{eq_4.35}
			\mE\biggl[\biggl\|\sum_{m=0}^{\ell}(\pmb{\eta}_{\tilde{\vu}}^m - \pmb{\eta}_{\tilde{\vu}_h}^m)\Delta_{m+1} W\biggr\|^2\biggr] 
			&= \mE\biggl[k\sum_{m=0}^{\ell} \bigl\|\pmb{\eta}_{\tilde{\vu}}^m - \pmb{\eta}_{\tilde{\vu}_h}^m \bigr\|^2\biggr] \\\nonumber
			&\leq C\mE\biggl[k\sum_{m=0}^{\ell}\|\e_{\tilde{\vu}}^m\|^2\biggr] + Ch^2.
		\end{align}
		Substituting \eqref{eq_4.35} into \eqref{eq_4.34} yields
		\begin{align}
			\frac{1}{\delta^2}\mE\biggl[\biggl\|k\sum_{m=0}^{\ell} \varepsilon_r^m\biggl\|^2\biggr] &\leq C\mE\bigl[\|\e_{\tilde{\vu}}^{\ell+1}\|^2 \bigr] + \mE\biggl[k\sum_{m=0}^{\ell}\|\nab\e_{\tilde{\vu}}^{m+1}\|^2\biggr] \\\nonumber
			&\quad + C\mE\biggl[k\sum_{m=0}^{\ell}\|\e_{\tilde{\vu}}^m\|^2\biggr]
			+ h^2 \mE\biggl[k\sum_{m=0}^{\ell}\|\nab\varepsilon_r^m\|^2\biggr] +  Ch^2.
		\end{align}
		Finally, the desired estimate \eqref{eq4.52} follows from Theorem \ref{thm4.1} 
		and Step 3 of Algorithm 2 and 4. The proof is complete.
	\end{proof}

	\begin{corollary}
		Let $\{p^m\}_{m=1}^M $ and $\{p_h^m\}_{m=1}^M $ be generated respectively by Algorithm 1 and  2. Then, there exists a positive constant $C \equiv C(D_T, \vu_0,\vf, \delta)$ such that
		\begin{align}
			\biggl(\mE\biggl[\biggl\|k\sum_{m=1}^M(p^m - p^m_h)\biggr\|^2\biggr]\biggr)^{\frac12} \leq C\biggl(\sqrt{k} + h +   h^2 k^{-\frac12}\biggr).
		\end{align}
	\end{corollary}
	
	\begin{proof} 
		Since the proof follows the same lines as those of the proof for Corollary \ref{cor3.7},
		we only highlight the main steps. By definition of $\{p^m\}$ and $\{p^m_h\}$, we have
		\begin{align}
			\biggl\|k\sum_{m=1}^M(p^m - p^m_h)\biggr\| &\leq 	\biggl\|k\sum_{m=1}^M(r^m - r^m_h)\biggr\| + 	\biggl\|\sum_{m=1}^M(\xi^m_{\tilde{\vu}} - \xi_{\tilde{\vu}_h}^m)\Delta W_{m+1}\biggr\| \\\nonumber
			&=: {\tt I} + {\tt II}.
		\end{align}
		
		Term  ${\tt I}$ can be easily bounded by using Theorem \ref{thm4.2}. 
		To bound ${\tt II}$, by It\^o isometry we get 
		\begin{align*}
			\mE[{\tt II}]^2 = \mE\biggl[k\sum_{m=1}^M\|\xi^m - \xi_h^m\|^2\biggr].
		\end{align*} 
		In addition, by \eqref{eq_3.1} and \eqref{eq_4.7}, we get
		\begin{align}\label{eq4.27}
			\bigl\|\nab(\xi_{\tilde{\vu}}^m - \xi_{\tilde{\vu}_h}^m)\bigr\|^2 
			&\leq C\bigl\|\vB(\tilde{\vu}^m) - \vB(\tilde{\vu}^m_h) \bigr\|^2 
			+ Ch^2 \bigl\|\xi_{\tilde{\vu}}^m \bigr\|_{H^2} 	\\\nonumber
			&\leq C\bigl\|\vB(\tilde{\vu}^m) - \vB(\tilde{\vu}^m_h) \bigr\|^2 + Ch^2 \bigl\|\div(\vB(\tilde{\vu}^m)) \bigr\|^2 \\\nonumber
			&\leq C \bigl\|\vB(\tilde{\vu}^m) - \vB(\tilde{\vu}^m_h) \bigr\|^2 + Ch^2 \bigl\|\nab(\vB(\tilde{\vu}^m)) \bigr\|^2\\\nonumber
			&\leq C \bigl\|\vB(\tilde{\vu}^m) - \vB(\tilde{\vu}^m_h) \bigr\|^2 + Ch^2 \bigl\|\nab\tilde{\vu}^m \bigr\|^2.
		\end{align} 
		
		Finally,  by Poincar\'e inequality, Lipschitz continuity of $\vB$, Theorem \ref{thm4.1} and \eqref{eq4.27}, we obtain
		\begin{align*}
			\mE[{\tt II}]^2 = \mE\biggl[k\sum_{m=1}^M\|\xi^m - \xi_h^m\|^2\biggr] \leq C\biggl(k + h^2 + \frac{h^4}{k}\biggr).
		\end{align*}
		The proof is complete.
	\end{proof}
	
	We conclude this section by stating the following global error estimate theorem for Algorithm 4, 
	which is another main result of this paper.
	
	\begin{theorem}\label{thm4.7}
		Let $({\bf u}, P)$ be the solution of (\ref{eq1.1}) 
		and $\{ ({ \tilde{\vu}}^m_{h}, r_{h}^m, p_{h}^m)\}_{m=1}^M$ be the solution of Algorithm 4.
		Then, 
		there exists a constant $C  \equiv C (D_T,\vu_0,\vf,\beta,\delta)>0$ such that 
		\begin{align*}
			\max_{1\leq m \leq M} \Bigl(\mathbb{E}\bigl[\bigl\|\vu(t_m)-\tilde{\vu}^m_{h} \bigr\|^2\,\bigl]\Bigr)^{\frac12}
			&+\Bigl( \mathbb{E}\Bigl[ k\sum_{m=1}^{M} \bigl\|\nabla (\vu(t_n) -\tilde{\vu}^n_{h}) \bigr\|^2\,\Bigr] \Bigr)^{\frac12} \\
			&\quad \leq C \biggl( \sqrt{k}+  h +   h^2 k^{-\frac12} \biggr)\,, \\ 
			\Bigg(\mathbb{E} \bigg[ \bigg\|R(t_m) -k\sum^m_{n=1}r^n_{h} \bigg\|^2\, \bigg]\Bigg)^{\frac12} 	
			&  + \Bigg( \mathbb{E}\bigg[\bigg\|P(t_m) -k\sum^m_{n=1}p^n_{h} \bigg\|^2\,\bigg] \Bigg)^{\frac12} \\
			&\quad \leq C \, \biggl(\sqrt{k}+ h +  h^2 k^{-\frac12}  \biggr)\, .
		\end{align*}
	\end{theorem}
	
	\begin{remark}
		The above error estimates are of the same nature as those obtained in \cite{Feng} 
		for the standard Euler-Maruyama mixed finite element method. On the other hand,
		the error estimates obtained in \cite{Feng1} for the Helmholtz enhanced Euler-Maruyama
		mixed finite element method do not have the growth term $h^2 k^{-\frac12}$. 
		Hence, in the case of general multiplication noise, the Helmholtz enhanced Euler-Maruyama
		mixed finite element method performs better than the Helmholtz enhanced Chorin 
		finite element method in terms of rates of convergence. 
	\end{remark}	
	
	%%%%%%%
	\section{Numerical experiments}\label{sec5} 	
	In this section, we present  two 2D numerical tests to guage the performance of the proposed numerical methods/algorithms. The first test is to verify the convergent rates proved in Theorem \ref{thm4.3} for Algorithm 3 while the second test is to validate the convergent rates proved in Theorem \ref{thm4.7}. 
	In both tests the computational domain is chosen as $D = (0,1)\times (0,1)$, the $P_1-P_1$ equal-order pair of finite element spaces are used for spatial discretization, the constant source function $\mathbf{f}=(1,1)$ is applied, 
	the terminal time is $T = 1$, the fine time and space mesh sizes   
	$k_0 = \frac{1}{4096}$ and $h =\frac{1}{50}$ are used to compute the numerical true solution,  and the number of realizations is set as $N_p = 500$ for the first test and $N_p = 800$ for the second one. 
	Moreover, to evaluate errors in strong norms, we use the following numerical 
	integration formulas: for any $0 \leq m \leq M$
	\begin{align*}
		\pmb{\mathcal{E}}_{\vu}^m &:=\biggl(\mE\Bigl[\|{\bf u}(t_m) -{\bf u}_h^m(k)\|^2 \Bigr]\biggr)^{\frac12} \approx
		\biggl(\frac{1}{N_p}\sum_{\ell=1}^{N_p} \bigl\|{\bf u }_h^m(k_0,\omega_\ell)
		-{\bf u}_h^m(k,\omega_\ell) \bigr\|^2 \biggr)^{\frac12}\, ,\\
		\pmb{\mathcal{E}}_{\vu,av}^M &:=\biggl(\mE\Bigl[k\sum_{m=0}^M\|{\bf u}(t_m) -{\bf u}_h^m(k)\|^2 \Bigr]\biggr)^{\frac12} \\\nonumber
		&\approx
		\biggl(\frac{1}{N_p}\sum_{\ell=1}^{N_p} \Bigl(k\sum_{m=0}^M\bigl\|{\bf u }_h^m(k_0,\omega_\ell)
		-{\bf u}_h^m(k,\omega_\ell) \bigr\|^2\Bigr) \biggr)^{\frac12}\, ,\\
		%\mathcal{E}_{p}^m &:=\biggl(\mE\Bigl[\|p(t_m)-p_h^m(k)\|^2\Bigr]\Bigr)^{\frac12} \approx
	%	\Bigl(\frac{1}{N_p}\sum_{\ell=1}^{N_p} \bigl\|p_h^m(k_0,\omega_\ell)-p_h^m(k,\omega_\ell) \bigr\|^2\biggr)^{\frac12}\, ,\\
		{\mathcal{E}}_{p,av}^M &:=\biggl(\mE\Bigl[k\sum_{m=0}^M\Big\|P(t_m) -k\sum_{n=0}^m p_h^n(k)\Big\|^2 \Bigr]\biggr)^{\frac12}
		\\
		&\,\,\approx \biggl(\frac{1}{N_p}\sum_{\ell=1}^{N_p}\Bigl(k\sum_{m=0}^M\Bigl\|k_0\sum_{n=1}^{\frac{t_m}{k_0}} p_{h}^n(k_0,\omega_\ell) -k\sum_{n=1}^{\frac{t_m}{k}}p_h^n(k,\omega_\ell)\Bigr\|^2\Bigr) \biggr)^{\frac12}\, .
	\end{align*}

	\medskip
	{\bf Test 1}. In this test, the  nonlinear multiplicative noise function $\vB$ is chosen as $\vB(\vu) = 10\bigl((u_1^2+1)^{\frac12}, (u_2^2+1)^{\frac12}\bigr)$ and the initial value $\vu_0 = (0,0)^{\top}$. Moreover, we choose $\mathbb{R}^J$-valued Wiener process $W$ with increments satisfying 
	\begin{align}
		\Delta W_{m+1} = W^J(t_{m+1},\mathbf{x}) - W^J(t_m,\mathbf{x}) = k_0\sum_{j,\ell=0}^J \sqrt{\lambda_{j,\ell}}\,e_{j,\ell}(\mathbf{x})\beta_{j,\ell}^m,
	\end{align}
	where $\mathbf{x} = (x_1,x_2) \in D$, $\beta_{j,\ell}^m \sim N(0,1)$ and $\{e_{j,\ell}(\mathbf{x})\}_{j,\ell}$ are orthonormal functions defined by $e_{j,\ell}(\mathbf{x}) = g_{j,\ell}\|g_{j,\ell}\|^{-1}$ with % nonsolenoidal functions:
	\begin{align}
		g_{j,\ell}(x_1,x_2) = \sin(j\pi x_1)\sin(\ell \pi x_2)
	\end{align}
	and $\lambda_{j,\ell} = \frac{1}{(j+\ell)^2}\|g_{j,\ell}\|$. In this test, we set $J = 2$, $\nu = 1$.
	
	Figure \ref{fig5.1} displays the convergence rates of the time discretization produced by 
	Algorithm 3 (and Algorithm 1) using different time step size $k$. The left figure shows the convergence rate  $O(k^{\frac14})$ in the $\pmb{\E}_{\vu,av}^M$-norm for the velocity approximation, while the right graph shows the same convergence rate in the $\pmb{\E}^M_{p,av}$-norm for the pressure approximation, both match the theoretical rates proved in our theoretical error estimates. 
	 
	\begin{figure}[thb]
		\begin{center}
			\centerline{
			\includegraphics[scale=0.26]{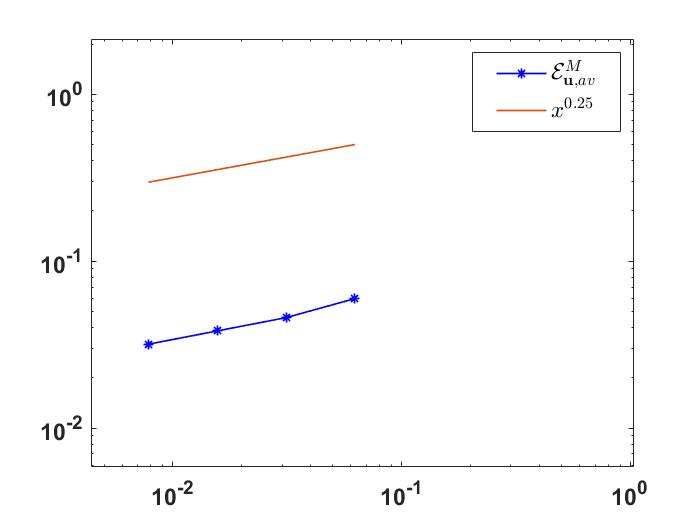}
			\includegraphics[scale=0.26]{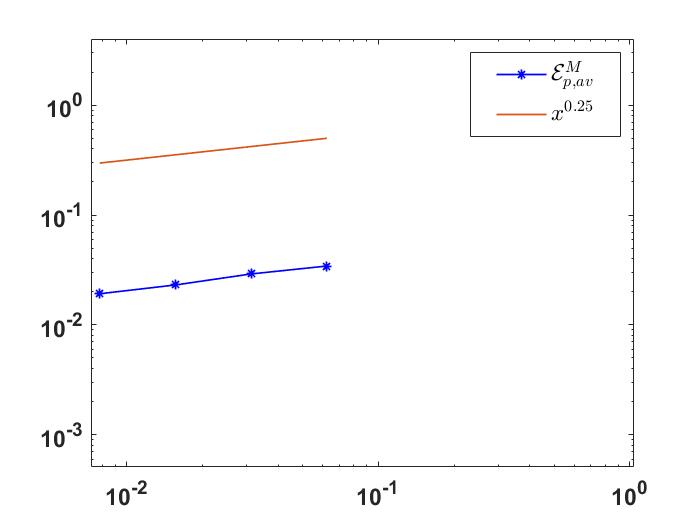}
			}
			\caption{Convergence rates of the time discretization for the velocity (left) and pressure (right) approximations by Algorithm 3 in the $\pmb{\E}^M_{\vu,av}$ norm and $\pmb{\E}^M_{p,av}$ norm respectively.}\label{fig5.1}
		\end{center}
		
	\end{figure}

	Next, we want to verify that the dependence of the error bounds on the factor $k^{-\frac12}$	is valid. To the end, we fix $h = \frac{1}{20}$ and use again different time step size $k$. The numerical results in  Figure \ref{fig5.3} shows  that the errors for both the velocity and pressure approximations increase as the time step size decreases, which proves that the error bounds are indeed proportional to some negative power of $k$.  
	
	\begin{figure}[thb]
		\begin{center}
			\centerline{
			\includegraphics[scale=0.26]{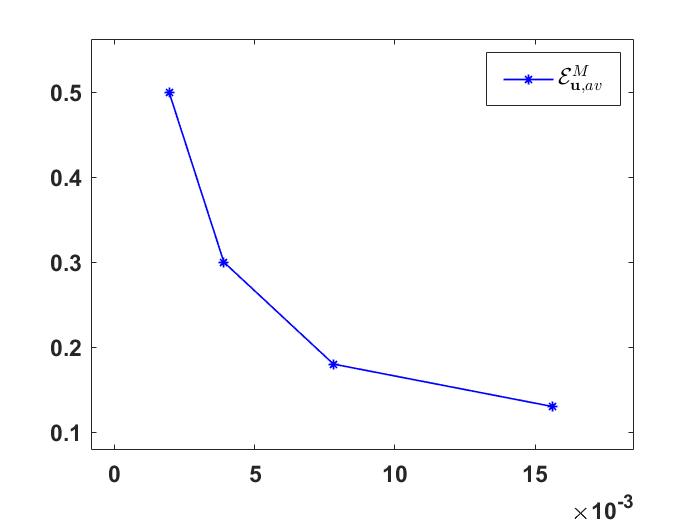}
			\includegraphics[scale=0.26]{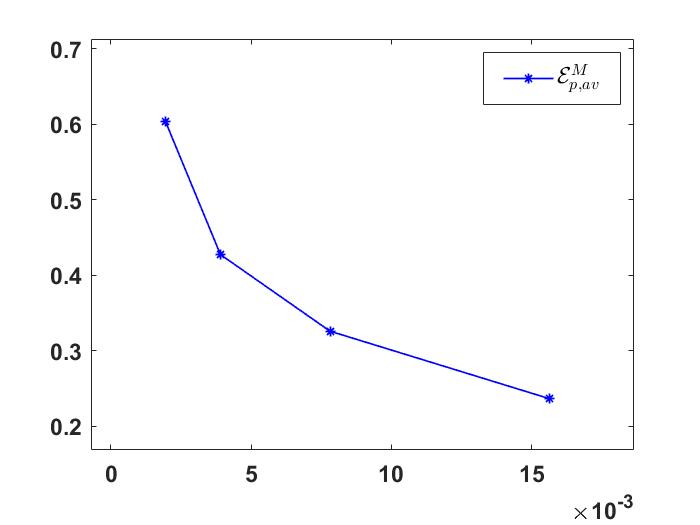}
		}
			\caption{Errors the velocity approximation (left) in $\pmb{\E}^M_{\vu,av}$ norm and the pressure approximation (right) in $\E_{p,av}^M$ norm by Algorithm 3.}\label{fig5.3}
		\end{center}
	\end{figure}
	
	To verify the sharpness of the error bounds on the factor $k^{-\frac12}$,  we implement Algorithm 3 using different pairs $(k,h)$, which satisfy the relation $h\approx k$, and display the numerical results in Figure \ref{fig5.4}.  We observe $\frac14$ order convergence rate for both the velocity and pressure approximations as predicted in Theorem \ref{thm4.3}. 
	
	\begin{figure}[thb]
		\begin{center}
			\centerline{
			\includegraphics[scale=0.26]{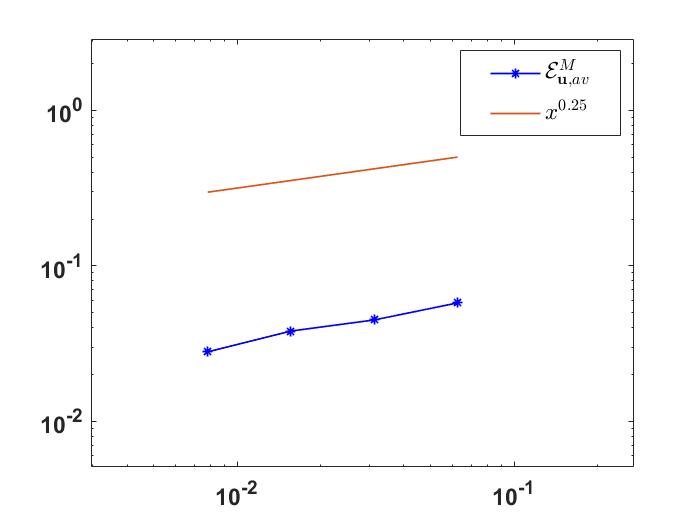}
			\includegraphics[scale=0.26]{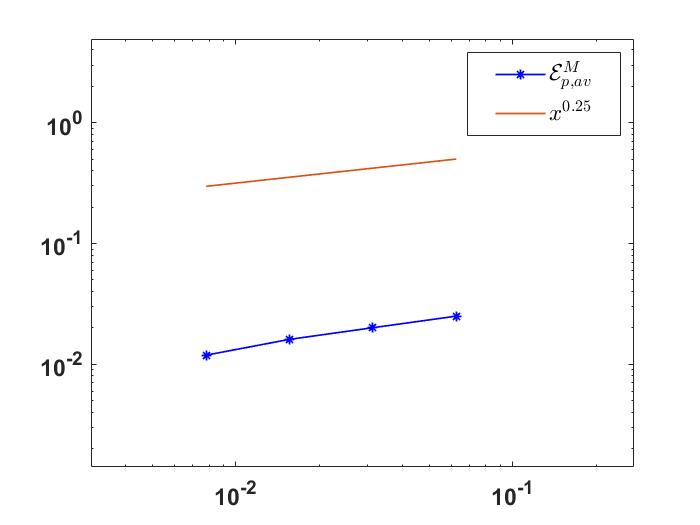}
		}
			\caption{Convergence rates in the $\pmb{\E}^M_{\vu,av}$ norm for the velocity (left) approximation and the $\E^M_{p,av}$ norm for the pressure (right) approximation by Algorithm 3 under the mesh condition $h \approx k$.}\label{fig5.4}
		\end{center}
	\end{figure}

\medskip
	{\bf Test 2.} We use the same test problem as in {\bf Test 1} to validate 
	the theoretical error estimates for our modified Chorin scheme given by Algorithm 4. However, the nonlinear multiplicative noise functions is chosen as $\vB(\vu) = \bigl((u_1^2+1)^{\frac12}, (u_2^2+1)^{\frac12}\bigr)$.
	It should be noted that a similar numerical experiment was done in \cite{CHP12}. 
	However, only the velocity approximation was analyzed and tested, no convergent rate for the pressure approximation was proved or tested in \cite{CHP12}. Here we want to emphasize the optimal convergence rate for the pressure approximation in the time-averaged norm.

	Figure \ref{fig5.5} displays the $\frac12$ order convergence rate in time for both the velocity and pressure approximations by Algorithm 4 as predicted by Theorem \ref{thm4.7}.  We note that 
	the velocity error is measured in the strong norm and the pressure error is 
	measured in a time-averaged norm. 
	
	\begin{figure}[thb]
		\begin{center}
			\includegraphics[scale=0.42]{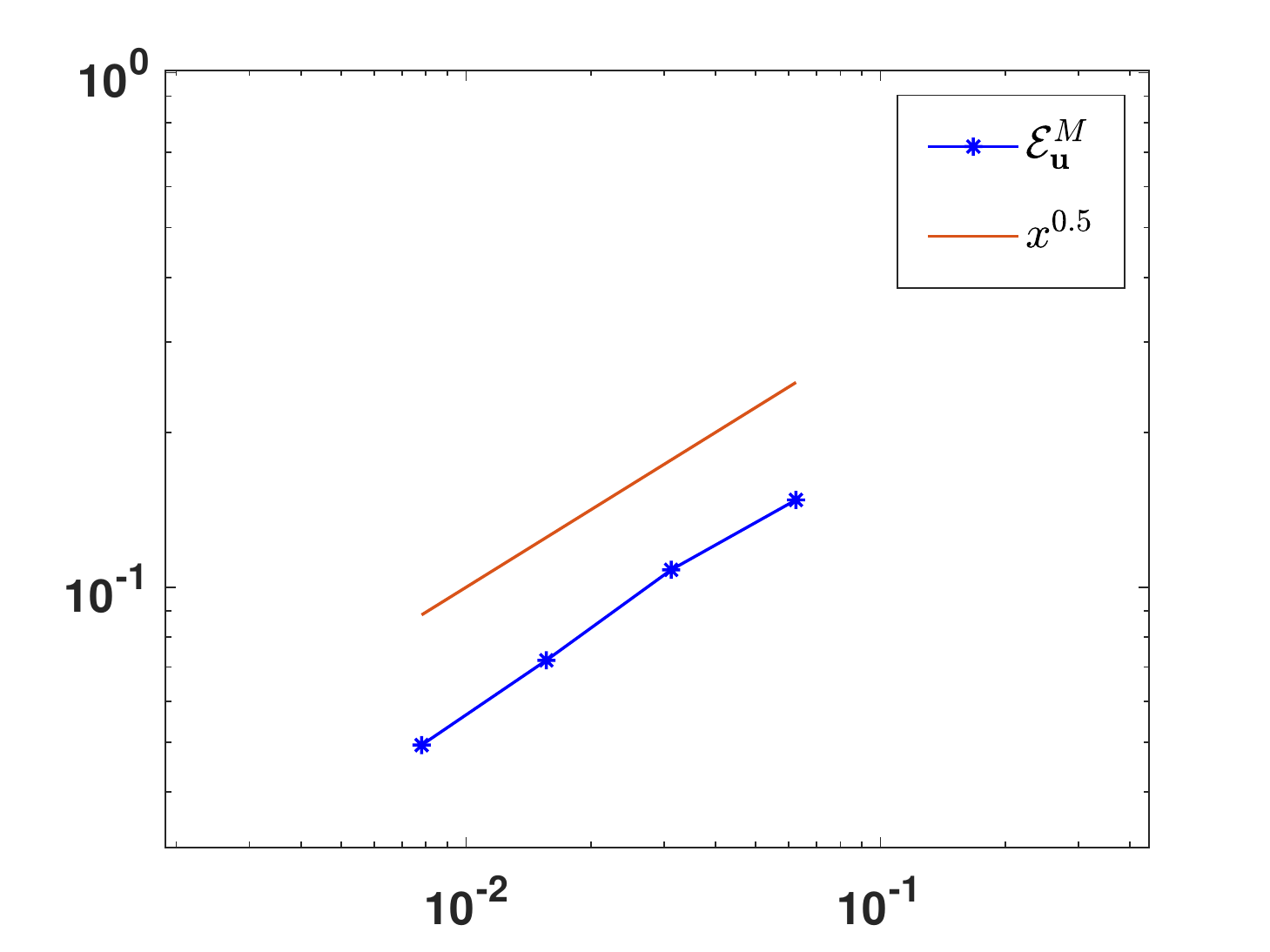}
			\includegraphics[scale=0.42]{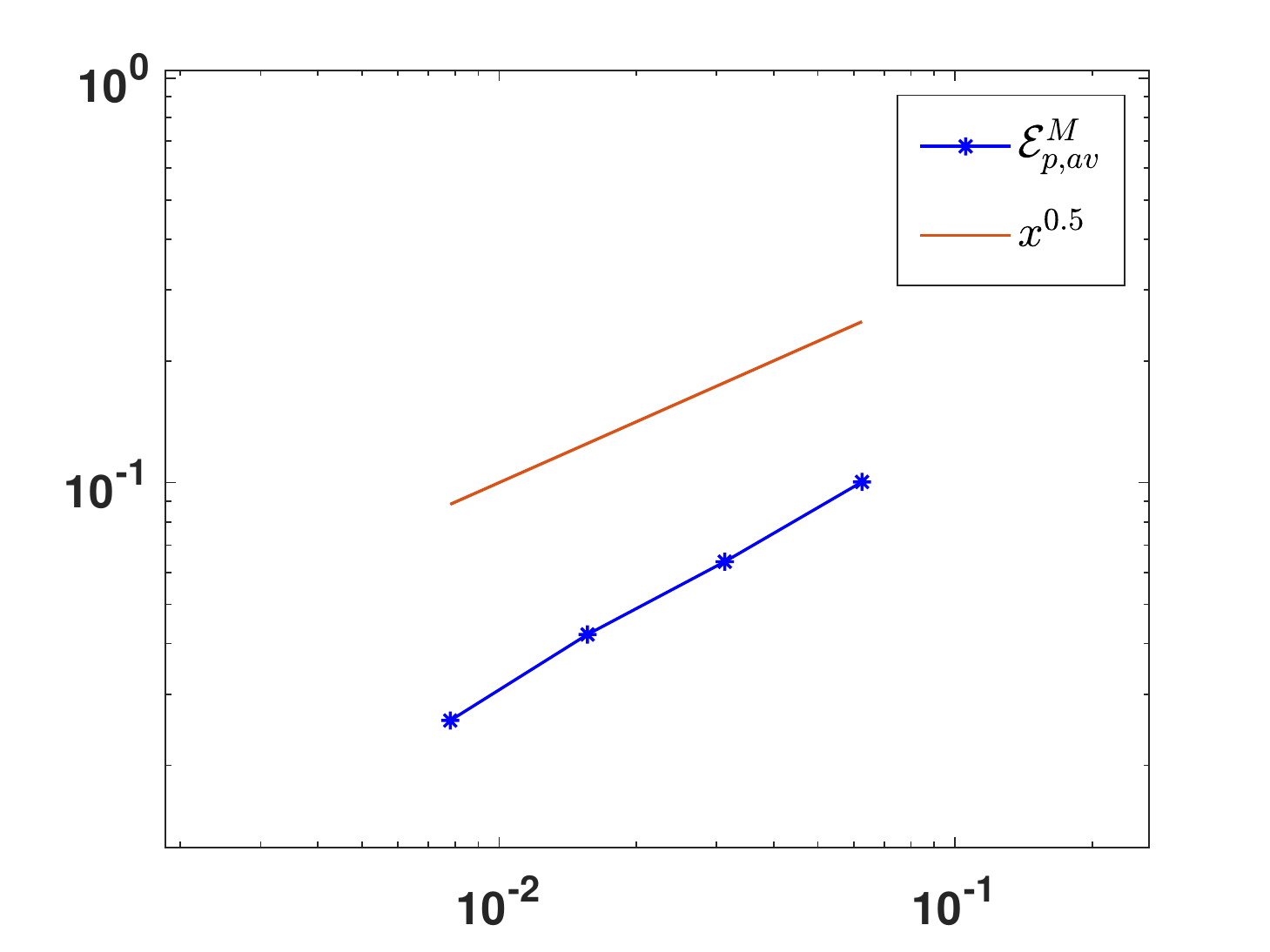}
			\caption{Convergence rates of the time discretization for the velocity in strong norm (left) and pressure  in time-averaged norm (right)  by Algorithm 4.}\label{fig5.5}
		\end{center}
		
	\end{figure}
	
	Similar to {\bf Test 1}, we want to test whether the dependence of the error bounds on the factor $k^{-\frac12}$ is valid and sharp.  To the end, we use the same strategy as we did in {\bf Test 1},  
	namely, we fix mesh size $h=\frac{1}{20}$ and decrease time step size $k$. As expected,  we observe
	that the errors blow up as shown in Figure \ref{fig5.6}. 
	\begin{figure}[thb]
		\begin{center}
			\includegraphics[scale=0.42]{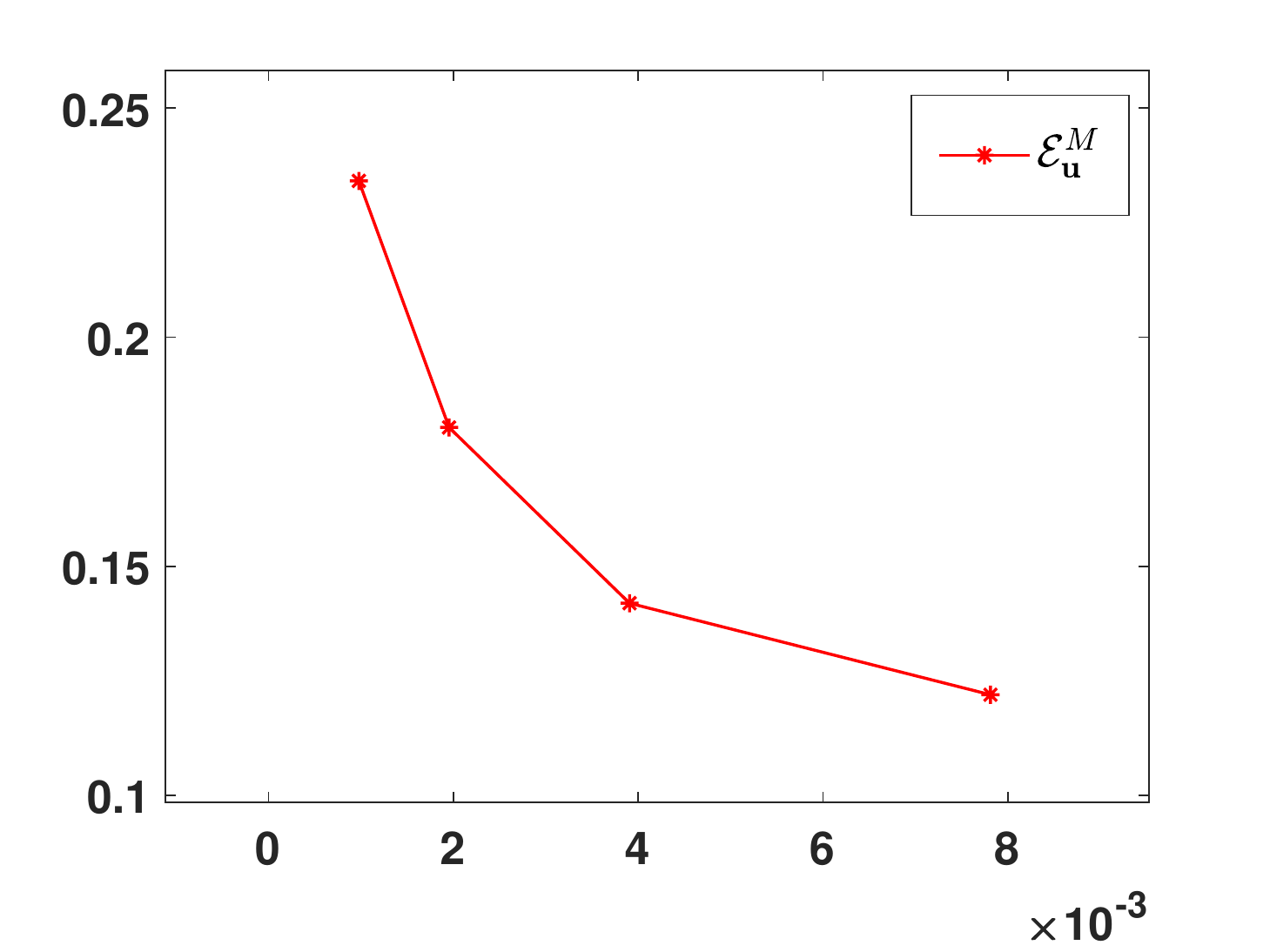}
			\includegraphics[scale=0.42]{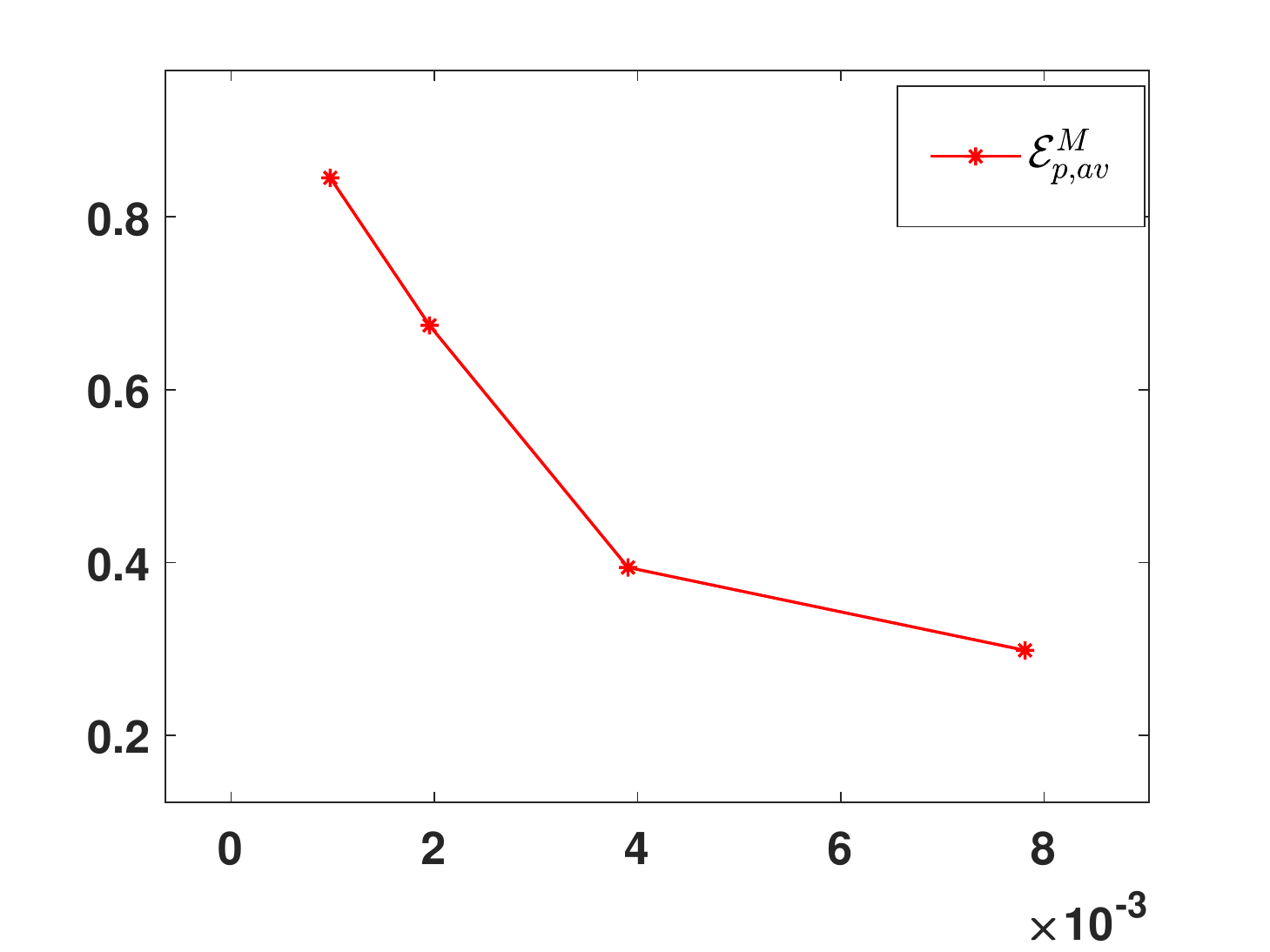}
			\caption{Errors for the velocity approximation in strong norm (left) and pressure approximation
				in time-averaged norm (right) by Algorithm 4.}\label{fig5.6}
		\end{center}
	\end{figure}
	Finally,  Figure \ref{fig5.7} shows the $\frac12$ order convergence rate for both the velocity and pressure approximations by Algorithm 4 when the time step size $k$ and the space mesh size $h$ satisfy the balancing condition $h \approx \sqrt{k}$, which verifies the sharpness of the dependence 
	of the error bounds on on the factor $k^{-\frac12}$  as predicted by Theorem \ref{thm4.7}. \begin{figure}[thb]
		\begin{center}
			\includegraphics[scale=0.42]{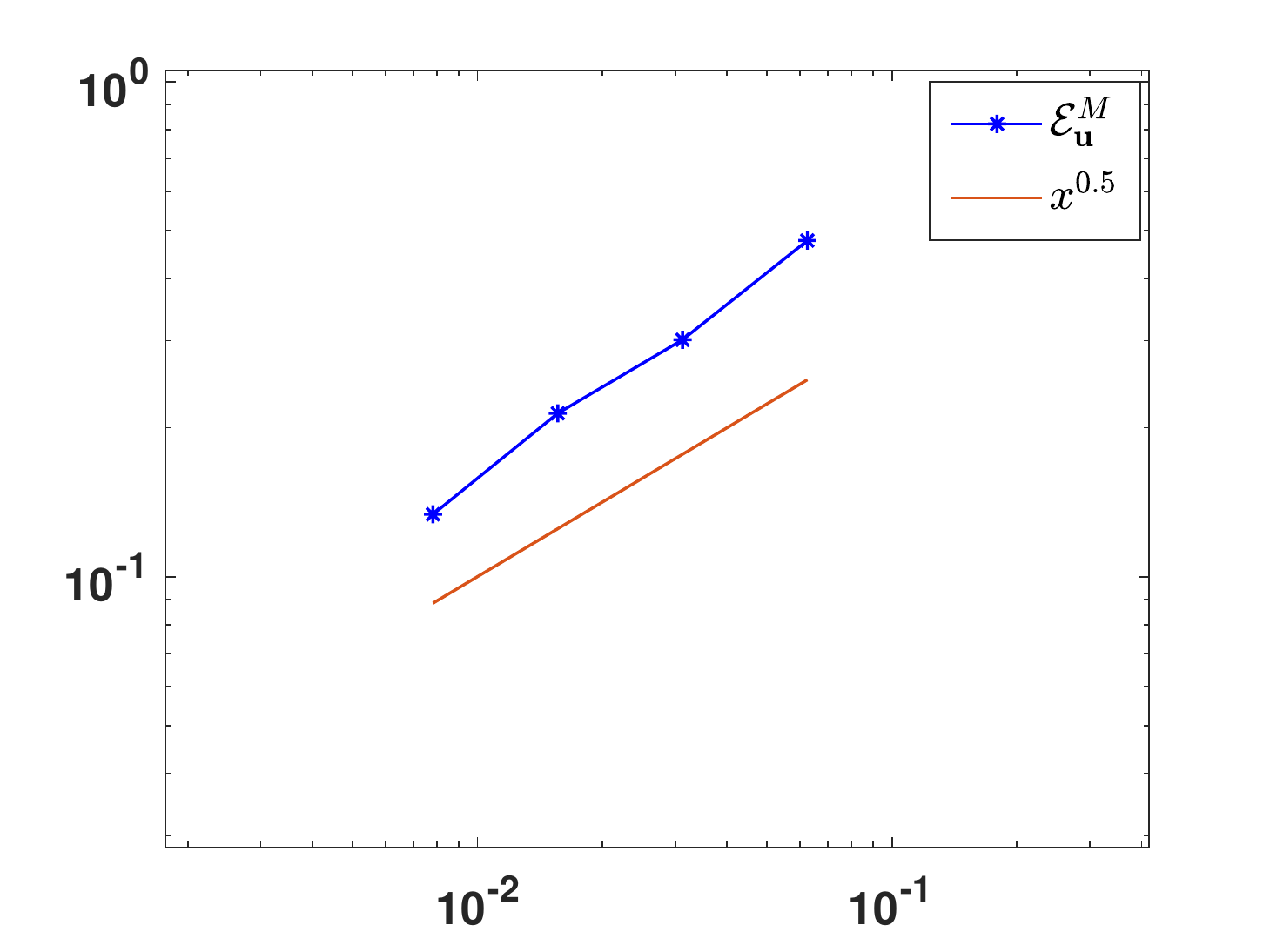}
			\includegraphics[scale=0.42]{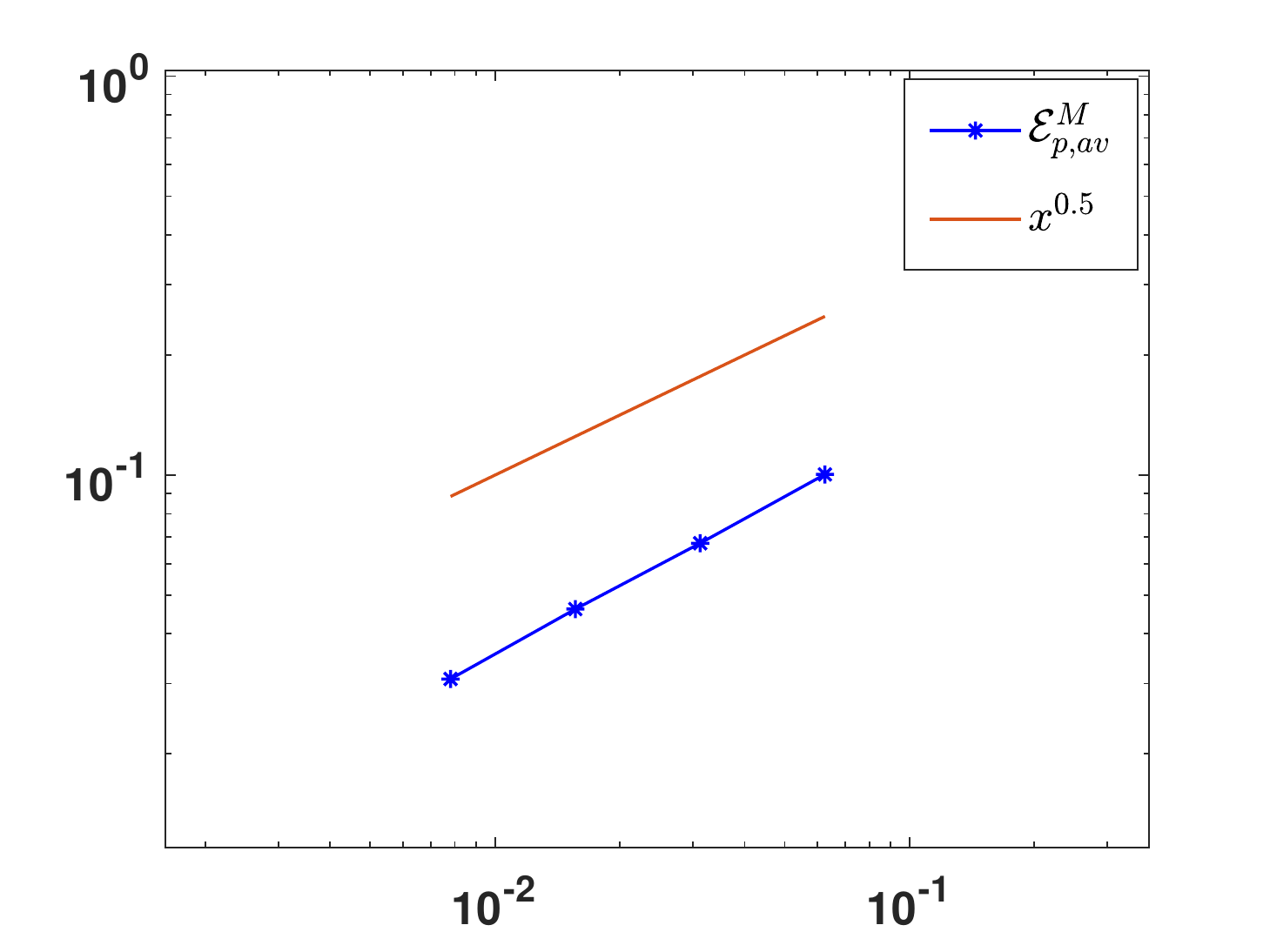}
			\caption{Convergence rates for the velocity approximation in strong norm (left) and pressure approximation in time-averaged norm (right) under the mesh condition $h \approx \sqrt{k}$. }\label{fig5.7}
		\end{center}
		
	\end{figure}

	%\newpage
	%%%%%%%%%%%%

\end{document}